\documentclass[a4paper,11pt]{article}
\usepackage[nointlimits]{amsmath}
\usepackage{amsfonts}
\usepackage{amssymb}
\usepackage{stmaryrd}
\usepackage{amsmath}
\usepackage{amsthm}
\usepackage{latexsym}
\usepackage{graphicx}
\usepackage{layout}
\usepackage[english]{babel}

\usepackage[pdftex, citecolor=black, linkcolor=black, colorlinks, hypertexnames=false, plainpages=false]{hyperref}

\usepackage{color}

\definecolor{blue2}{rgb}{0,0,0.6}
\definecolor{turck}{rgb}{0,0.5,0.5}
\definecolor{sturck}{rgb}{0,0.9,0.8} 
\definecolor{orange}{rgb}{1,0.4,0}
\definecolor{bord}{rgb}{1,0.3,0.3}
\definecolor{violet}{rgb}{0.4,0.2,0.8}
\definecolor{sblue}{rgb}{0.2,0.4,1.}
\definecolor{yellowa}{rgb}{1.,0.8,0}
\definecolor{greena}{rgb}{0.2,0.8,0}
\definecolor{greenb}{rgb}{0,0.5,0.2}
\definecolor{greenc}{rgb}{0.1,0.7,0.3}
\definecolor{greend}{cmyk}{0.7,0,0.3,0}
\definecolor{greene}{cmyk}{0.7,0,0.3,0.1}
\definecolor{greya}{cmyk}{0,0,0,0.1}
\definecolor{greyb}{cmyk}{0,0,0,0.2}
\definecolor{greyc}{cmyk}{0,0,0,0.3}
\definecolor{greyd}{cmyk}{0,0,0,0.4}
\definecolor{greye}{cmyk}{0,0,0,0.5}
\definecolor{greyf}{cmyk}{0,0,0,0.6}
\definecolor{greyg}{cmyk}{0,0,0,0.7}
\definecolor{greyh}{cmyk}{0,0,0,0.8}
\definecolor{greyi}{cmyk}{0,0,0,0.9}

\long\def\notes#1{\ifinner
         {\tiny #1}
         \else
          \marginpar{\protect\tiny #1}%
          \fi}%

\newtheorem{lemma}     {Lemma}[section]
\newtheorem{thm}   [lemma]{Theorem}
\newtheorem{teorema1}   [lemma]{Theorem}
\newtheorem{prop}      [lemma]{Proposition}
\newtheorem{coro}    [lemma]{Corollary}
\newtheorem{cong1}      [lemma]{Conjecture}
\newtheorem{remark1}    [lemma]{Remark}
\newtheorem{defin}     [lemma]{Definition}

\numberwithin{equation}{section}

\newcommand{\und}{\underline}

     \newcommand{\nn}{\nonumber}

\newcommand{\dis}{\displaystyle}

\newcommand{\ssp}{\mathrm{sp}}

\newcommand{\mmmintone}[1]{{\dis{\int\kern -.43cm
-}}_{\kern-.21cm\substack{#1}}\;\;}
\newcommand{\mmmintwo}[2]{{\dis{\int\kern -.43cm
-}}_{\kern-.21cm\substack{#1}}^{\substack{#2}}\;\;}
\newcommand{\submint}{{\scriptstyle{\int\kern -.66em -}}}
\newcommand{\submintone}[1]{{\scriptstyle{\int\kern -.66em
-}}_{\scriptscriptstyle{\kern-.21em\substack{#1}}}}
\newcommand{\fracmint}{{\textstyle{\int\kern -.88em -}}}
\newcommand{\fracmintone}[1]{{\textstyle{\int\kern -.88em
-}}_{\scriptscriptstyle{\kern-.21em\substack{#1}}}\;}

\def\mintone{\protect\mmmintone}

\newcommand{\eps}{\epsilon}

\newcommand{\ga}{\gamma}
\newcommand{\Ga}{\Gamma}
\newcommand{\Om}{\Omega}
\newcommand{\om}{\omega}
\newcommand{\si}{\sigma}
\newcommand{\Si}{\Sigma}

\newcommand{\la}{\lambda}
\newcommand{\La}{\Lambda}

\newcommand{\nada}[1]{}

\title{Potts models in the continuum. Uniqueness and
exponential decay  in the restricted ensembles}

\date{\today}
\author{
A. De Masi 
\thanks{Anna De Masi: Dipartimento di Matematica, Universit\`a di L'Aquila, Via Vetoio, 1, 67010 COPPITO (AQ), Italy; \texttt{demasi@univaq.it}.} 
\and 
I. Merola 
\thanks{Immacolata Merola: Dipartimento di Matematica, Universit\`a di L'Aquila, Via Vetoio, 1, 67010 COPPITO (AQ), Italy; \texttt{merola@univaq.it}.}  
\and 
E. Presutti
\thanks{Errico Presutti, Dipartimento di Matematica, Universit\`a
di Roma Tor Ver\-ga\-ta, 00133 Roma, Italy; \texttt{presutti@mat.uniroma2.it}.} 
\and Y. Vignaud
\thanks{Yvon Vignaud: TU Berlin - Fakult\"{a}t II, Institut f\"{u}r Mathematik, Strasse des 17. Juni 136, D-10623 Berlin, Germany; \texttt{ vignaud@mail.math.tu-berlin.de}.}
}

\begin{document}
\maketitle

\pagenumbering{arabic}

\begin{abstract}

In this paper we study a continuum version of the Potts model, where
particles are points in $\mathbb R^d$, $d\ge 2$, with a spin which
may take $S\ge 3$ possible values. Particles with different spins
repel each other via a Kac pair potential of range $\ga^{-1}$,
$\ga>0$.  In mean field, for any inverse temperature $\beta$ there
is a value of the chemical potential $\la_\beta$ at which $S+1$
distinct phases coexist.  We  introduce a restricted ensemble for
each mean field pure phase which is defined so that the empirical
particles densities are close to the mean field values. Then, in the
spirit of the Dobrushin Shlosman theory, \cite{DS}, we prove  that
while the Dobrushin high-temperatures uniqueness condition does not
hold, yet a finite size condition is verified for $\ga$ small enough
which implies uniqueness and exponential decay of correlations. In a
second paper, \cite{DMPV2}, we will use such a result to implement
the Pirogov-Sinai scheme proving coexistence of $S+1$ extremal DLR
measures.

\end{abstract}

\tableofcontents


\section{{\bf Introduction}}
        \label{sec:e1}
In this paper we consider a continuum version of the classical Potts
model, namely a system of point particles in $\mathbb R^d$ where
each particle has a spin $s\in\{1,..,S\}$, $S>1$, and particles with
different spins repel each other, this being the only interaction
present. When $S=2$ this is a simple version of the famous
Widom-Rowlinson model which has been the first system where phase
transitions in the continuum have been rigorously proved,
\cite{ruelle}, and for $S\ge 2$ and at very
low temperature, a phase coexistence  between the $S$
symmetric phases for continuum Potts models
was established in \cite{GH}.

The mean field version of the continuum Potts model has
been recently studied in \cite{GMRZ}. The phase diagram has
an interesting structure. In the $(\beta,\la)$-plane,
$\beta$ the inverse temperature, $\la$ the chemical
potential, there is a critical curve, see Figure 1, above
which (i.e.\ $\la$ ``large''), there is segregation, namely
there are $S$ pure phases, each one characterized by having
``a most populated species'' (of particles with same spin).
Instead, below the critical curve there is only one phase,
the disordered one where the spin densities are all equal.
The behavior on the critical curve depends on $S$. If $S=2$
there is only the disordered phase while if $S>2$ there is
coexistence, namely there are $S+1$ phases, the ``ordered
phases'' where there is a spin density larger than all the
others and the disordered phase as well.

An analogous phenomenon occurs in the mean
field lattice Potts model
where at a critical temperature there is a first order
phase transition with
coexistence of $S+1$ phases if $S>2$, but
in the continuum there is an extra phenomenon occurring at the
transition, namely the total particles density has
a strictly positive jump when going
from the disordered to an ordered phase.
This can be seen as an example
of interplay between magnetic and elastic properties
and interpreted as a magneto-striction effect, as the
appearance of a net magnetization is accompanied by an
increase of density and thus a decrease of
inter-particles distances.

Our purpose is to prove that the above picture remains valid
if mean field is replaced by a finite
range interaction.   Let  $q=(...,r_i,s_i,...)$, $i=1,..,n$,
$r_i\in \mathbb R^d$, $s_i\in \{1,..,S\}$,
a finite configuration of particles.  We suppose that their
energy is
   \begin{equation}
      \label{e1.1}
H_\la (q) = \frac 12 \sum_{i\ne j} V_\ga(r_i,r_j) \text{\bf
1}_{s_i\ne s_j} - \la n
     \end{equation}
where
   \begin{equation}
      \label{e1.2}
V_\ga(r,r')=\int_{\mathbb R^d} J_{\ga}(r,z)J_{\ga}(z,r')
     \end{equation}
$J_{\ga}(r,r')=\ga^d J (0,\ga(r'-r))$, $\ga>0$
a Kac scaling parameter,
$J(0,r)$ a smooth probability kernel supported by $|r|\le
1/2$.  (Observe that $H_\la (q)$ is independent
of the particles labeling).

To motivate the above choice recall that the mean field energy density
(mean field energy over volume) is
   \begin{equation}
      \label{e1.3}
e_\la(\rho(\cdot)):= \frac 1{2} \sum_{s\ne s'}
\rho(s)\rho(s') - \la \rho_{\rm tot}, \quad \rho_{\rm tot}=\sum_{s} \rho(s)
     \end{equation}
where $\rho(s)$ is the density of particles with spin $s$.  Then
   \begin{equation}
      \label{e1.4}
H_\la (q) = \int e_\la(\rho_{q,r}(\cdot)),\quad \rho_{q,r}(s)= \sum_i
\text{\bf 1}_{s_i=s} J_\ga(r,r_i)
     \end{equation}
Thus $H_\la (q)$ is the integral of the mean field free energy density,
where the latter is computed using the empirical
averages $\rho_{q,r}(s)$.  If $\ga$ is small one may think
that \eqref{e1.1} ``simulates mean field''.
Indeed we will prove in \cite{DMPV2} that

\vskip1cm

    \begin{thm}
    \label{thme1.1}
For any $d\ge 2$, $S>2$ and $\beta >0$ there is $\ga^*>0$ such that
for any $\ga\le \ga^*$ there exist $\la_{\beta,\ga}$ and $S+1$
mutually distinct, extremal
DLR measures at $(\beta,\la_{\beta,\ga})$.
    \end{thm}

\vskip1cm

To keep the statement simple we have not reported all the
information we have on the structure of
the DLR measures referring to \cite{DMPV2} for the full result.
 In particular we know that the particles densities
are close to their mean field values
(for $\ga$ small).  The proof of Theorem \ref{thme1.1}
follows the Pirogov-Sinai strategy which is based on the introduction
of ``restricted ensembles'' where the original phase space
of the system is restricted by constraints which impose local closeness
to one of the putative pure phases, in our case local closeness
of empirical averages to the mean field values in a pure phase.  We need
a full control of such ``restricted ensembles'' and then a general machinery
applies   giving the desired phase transition.  As a difference
with the classical Pirogov-Sinai theory, here the small
parameter is the inverse interaction range $\ga$ instead of the temperature,
as we are ``perturbing'' mean field instead of the ground states, see
for instance the LMP model, \cite{LMP}, where
these ideas have been applied to prove
phase transitions for particles systems in the continuum with
Kac potentials.

In the typical applications of Pirogov-Sinai, restricted
ensembles are studied using cluster expansion which yields
a complete analyticity (in the Dobrushin-Shlosman sense,
\cite{DS}) characterization of the system.  Namely
constraining the system into a restricted ensemble raises
the effective temperature and the state enjoys the
characteristics of high temperature systems.  An analogous
effect has been found in the Ising model with Kac
potentials, \cite{CP}, \cite{BZ}, and in the LMP model, in
both the high-temperatures Dobrushin uniqueness condition
has been proved to hold. This is a ``finite size''
condition, and the Dobrushin uniqueness theorem states that
if such a condition is verified, then there is a unique DLR
state. The importance of the result is that the condition
involves only the analysis of the system in a finite box:
loosely speaking it is a contraction property which states
that compared with the variations of the boundary
conditions, the Gibbs measure has strictly smaller changes,
all this being quantified using the Wasserstein distance.
\emph{Dobrushin's high temperatures} means that the size of the
box [where the conditional measures are compared] can be
chosen small (a single spin in the Ising case) or a small
cube in LMP so that there is no self interaction in Ising
or a negligible interaction among particles of the box (in
LMP) and the main part of the energy is due to the
interaction with the boundary conditions.  The measure and
its variations are then quite explicit and it is possible
to check the validity of the above contraction property.

As explained by Dobrushin and Shlosman, one expects that
when lowering the temperature the above high temperature
property eventually fails, the point however being that it
could be regained if we look at systems still in a finite
box but with a larger size, eventually divergent as
approaching the critical temperature. The problem is that
if the finite size condition involves a large box then self
interactions are important and it is difficult to check
whether the condition is verified.

While it is generally believed that the above picture is
correct, there are however not many examples where it has
been rigorously established.  Unlike Ising with Kac
potentials and LMP, in an interval of values of the
temperature, where the high temperature Dobrushin condition
is valid in restricted ensembles, in the continuum Potts
model we are considering there is numerical evidence (at
least) that it is not verified. We will prove here that a
finite size condition (involving some large boxes where
self interaction is important) is verified in our
restricted ensembles and then prove using the disagreement
percolation techniques introduced in \cite{vb}, \cite{vm},
that our finite size condition implies uniqueness and
exponential decay of correlations and all the properties
needed to implement Pirogov-Sinai, a task accomplished in
\cite{DMPV2}.

\vskip3cm

\part{ Model and main results}


\vskip2cm

\section{{\bf Mean field}}
        \label{sec:e2}

\setcounter{equation}{0}

The ``multi-canonical'' mean field free energy  is
   \begin{equation}
      \label{e2.1}
F^{\rm mf}(\rho) = \frac 12 \sum_{s\ne s'}
\rho_s\rho_{s'} + \frac 1 \beta \sum_s \rho_s [\log
\rho_s-1 ],\quad \rho=\{\rho_1,..,\rho_S\}\in \mathbb
R_+^S
     \end{equation}
where $\rho_s$ represents the density of particles with spin
$s$ and $\beta$ the inverse
temperature, to underline dependence on $\beta$ we
may add it as a subscript.
The ``canonical'' mean field free energy  is instead
   \begin{equation}
      \label{e2.2}
f^{\rm mf}(x) = \inf\big\{F^{\rm mf}(\rho);
\sum_{s} \rho_s=x\big\},\qquad x>0
     \end{equation}
and the mean field free energy $CEf^{\rm mf}(x)$
is the convex envelope of $f^{\rm mf}(x)$.
$F^{\rm mf}_\la(\rho)$,
$f^{\rm mf}_\la(x)$ and $CEf_\la^{\rm mf}(x)$,
$\la\in \mathbb R$ the chemical
potential, are defined by adding the term $-\la x$, where
in the case of $F^{\rm mf}_\la(\rho)$,
$\dis{x=\sum_{s} \rho_s}$.

Observe that for any $a>0$,
   \begin{equation}
      \label{e2.3}
F^{\rm mf}_{\beta,\la}(\rho) =  a^{-2} F^{\rm mf}_{\beta/a,\la'}(a\rho),
\quad  \la= a^{-1}\la' - \frac{\log a}\beta
     \end{equation}
so that if the graph of
$CEf_{\beta,\la}^{\rm mf}(x)$ has a horizontal segment, then
for any $\beta'$, $CEf_{\beta',\la'}^{\rm mf}(x)$ has also a
horizontal segment when $\la'=a\la +\beta^{-1} a\log a$, $a=\beta/\beta'$, which
reduces the analysis of phase transitions to a single temperature, object
of the following considerations.

As shown in \cite{GM} (see the proof of Theorem A.1 therein), the variational problem \eqref{e2.2} is actually reduced to a two-dimensional problem because:

    \begin{lemma}
    \label{lemma2.1}
   \begin{equation}
      \label{e2.4}
f^{\rm mf}(x) = \inf\big\{F^{\rm mf}(\rho);
\sum_{s} \rho(s)=x; \rho_1\ge \rho_2=\cdots= \rho_S\big\}
     \end{equation}
     \end{lemma}

\vskip1cm

 The analysis of \eqref{e2.4} yields:

\vskip1cm

    \begin{thm}
    \label{thme2.1}
Let $S>2$ and $\beta>0$. Then there are $0<x_-<x_+$ such that
$CEf^{\rm mf}_\beta(x)$ coincides with  $f^{\rm mf}_\beta(x)$ in the complement
of $(x_-,x_+)$ and it is  a straight line in $[x_-,x_+]$.  As a consequence
there is $\la_\beta$ such that
$CEf^{\rm mf}_{\beta,\la_\beta}(x)$
has the whole interval  $[x_-,x_+]$ as minimizers, it is strictly convex in
the complement and $D^2f^{\rm mf}_{\beta,\la_\beta}(x_{\pm})>0$.

    \end{thm}

\vskip1cm

By using the scaling property \eqref{e2.3} we then obtain
the phase diagram in Figure \ref{phasediagram}.

\begin{figure}
\centering
\includegraphics[width=.8\textwidth]{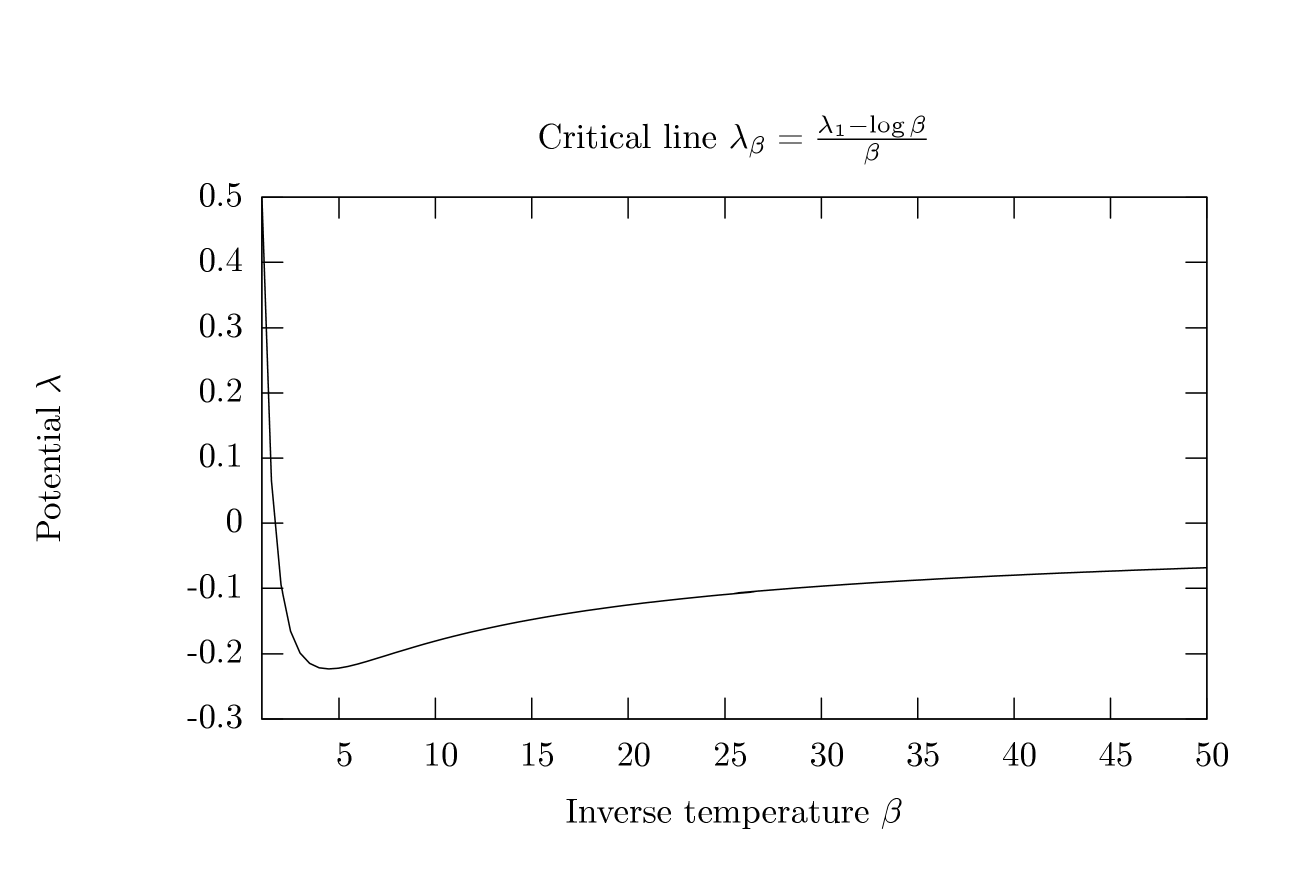}
\caption{Phase Diagram of the Mean field Potts gas}
\label{phasediagram}
\end{figure}

We will next discuss the structure of the minimizers
of $F^{\rm mf}_{\beta,\la_\beta}(\rho)$.

\vskip1cm

    \begin{thm}
    \label{thme2.2}
Let $S>2$, $\beta>0$ and $\la_\beta$ as in Theorem \ref{thme2.1}.
Then $F^{\rm mf}_{\beta,\la_\beta}(\rho)$ has $S+1$ minimizers
denoted by $\rho^{(k)}$, $k=1,..,S+1$. For $k\le S$,
$\rho^{(k)}_k>\rho^{(k)}_s, s\ne k$ and $\rho^{(k)}_s=
\rho^{(k)}_{s'}$ for all $s,s'$ not equal to $k$.  Instead
 $\rho^{(S+1)}_s=\rho^{(S+1)}_1$ for all $s$ and
    \begin{equation}
      \label{e2.5}
 \sum_{s}
\rho^{(1)}_s>  \sum_{s}
\rho^{(S+1)}_s
     \end{equation}
Finally for any $k$
the Hessian matrix $L^{(k)}:=D^2F^{\rm mf}_{\beta,\la_\beta}(\rho^{(k)})$
is strictly positive, namely there is $\kappa^*>0$ such that for any vector $v=v(s),
s\in \{1,..,S\}$,
    \begin{equation}
      \label{2.4}
\langle v,L^{(k)} v\rangle=   \sum_{s,s'} L^{(k)}(s,s') v(s) v(s')
\ge \kappa^*\langle v,v\rangle
     \end{equation}

    \end{thm}

The proof of Theorems \ref{thme2.1} and \ref{thme2.2} is given in
Appendix \ref{sec:C}.

\vskip1cm

The minimizers satisfy the mean field equation
        \begin{equation}
      \label{e2.7}
\rho^{(k)}_s=  \exp\Big\{-\beta\{\sum_{s'\ne s} \rho^{(k)}_{s'} -
\la_\beta\} \Big\}
     \end{equation}
The Hessian $L^{(k)}$ has the explicit form:
   \begin{equation}
      \label{e2.8}
 L^{(k)}(s,s') =\frac
 {\partial^2 F^{\rm mf}_{\beta,\la_\beta }}{\partial\rho_s
 \partial\rho_{s'}}\Big|_{\rho=\rho^{(k)}}= \frac{1}{\beta \rho^{(k)}_s}
 \text{\bf 1}_{s=s'} + \text{\bf 1}_{s\ne s'}
     \end{equation}
%
%
%
%
%

\vskip2cm
 \setcounter{equation}{0}

\section{{\bf Restricted ensembles}}
\label{sec:e3}

The purpose of this paper is to study the system in
restricted ensembles defined by restricting the phase space
to particles configurations which are ``close
to a mean field equilibrium phase''.  Unfortunately the requests
from the Pirogov-Sinai theory will complicate the picture, but let
us do it gradually and start by defining notions as
local equilibrium and
``coarse grained'' variables,
adapted to the present context.

\vskip1cm

\subsection{Geometrical notions}
    \label{sec:e3.1}
We discretize $\mathbb R^d$ by introducing cells of size $\ell>0$,
the mesh parameter $\ell$ will be specified in the next paragraph.

\vskip.5cm

\centerline  {\it The partition $\mathcal D^{(\ell)}$}
 \nopagebreak
$\bullet$\; $\mathcal D^{(\ell)}$, $\ell>0$, denotes the partition
$\{C^{(\ell)}_x, \, x\in \ell \mathbb Z^d\}$ of
$\mathbb R^d$ into the cubes
$C^{(\ell)}_x=\{r\in \mathbb R^d:
x_i \le r_i <x_i+\ell, i=1,..,d\}$ ($r_i$ and $x_i$ the cartesian
components of $r$ and $x$), calling $C^{(\ell)}_r$
the cube  which contains $r$.

$\bullet$\; A set $\La$ is
$\mathcal D^{(\ell)}$-measurable if it is union of cubes in
$\mathcal D^{(\ell)}$ and $\delta_{\rm out}^{\ell}[ \La]$ denotes
the union of
all $\mathcal D^{(\ell)}$ cubes in $\La^c$( the complement of
$\La$) which are connected to $\La$,
two sets being connected if their closures have
non empty intersection.
Analogously,
$\delta_{\rm in}^{\ell}[ \La]$ is the union of all $\mathcal
D^{(\ell)}$ cubes in $\La$ which are connected to $\La^c$.

$\bullet$\; A function $f:\mathbb R^d\to \mathbb R$ is
$\mathcal D^{(\ell)}$-measurable if its inverse images are
$\mathcal D^{(\ell)}$-measurable sets.

\bigskip

\centerline  {\it The basic scales}
 \nopagebreak
There are four main lengths in our analysis:
$\ell_0 \ll \ell_{-,\ga}
\ll \ga^{-1}\ll \ell_{+,\ga}$.  More precisely
let $\alpha_+$, $\alpha_-$ and $ a$ verify
    \begin{equation}
    \label{e3.1.1}
\frac 12 \gg \alpha_+>\alpha_-\gg a >0
    \end{equation}
(the precise meaning of the inequality will
become clear in the course of the proofs), then
    \begin{equation}
    \label{e3.1.2}
    \lim_{\ga \to 0} \frac{\ell_{0}}{\gamma^{-1/2}}=
\lim_{\ga \to 0} \frac{\ell_{-,\ga}}{\gamma^{-(1-\alpha_-)}}=
\lim_{\ga \to 0} \frac{\ell_{+,\ga}}{\gamma^{-(1+\alpha_+)}}=1
    \end{equation}
with the additional request that $\ell_{+,\ga}$ is an integer multiple
of $\ga^{-1}$ which is an  an integer multiple
of $\ell_{-,\ga}$ which is an  integer multiple
of $\ell_{0}$.  The partition $\mathcal D^{(\ell)}$ is coarser than
$\mathcal D^{(\ell')}$ if each cube of the former is union of cubes of
the latter, we will then also say that $\mathcal D^{(\ell')}$
is finer than $\mathcal D^{(\ell)}$.  This happens if and only if
$\ell$ is an integer multiple of $\ell'$, thus $\mathcal D^{(\ell_0)}$
is finer than  $\mathcal D^{(\ell_{-,\ga})}$
which is finer than  $\mathcal D^{(\ga^{-1})}$ which
is finer than  $\mathcal D^{(\ell_{+,\ga})}$.

We will need that
    \begin{equation}
    \label{ea3.1.1}
        \frac {(\alpha_++\alpha_-)d}{2(1-\alpha_-)}<\frac 1{1000},\quad
        8\alpha_++9\alpha_-<\frac 12
    \end{equation}
    \vskip.5cm

{Eventually we define, for any $\mathcal
D^{(\ell_{+,\ga})}$-measurable region $\La$, :}

\begin{equation}
    \label{ea3.1.1b}
       N_{\La}:=\frac{|\La|}{\ell_{+,\ga}^d}
    \end{equation}
    where $|\La|$ is the volume of the region $\La$, thus $N_{\La}$ is
    the number of blocks $C^{(\ell_{+,\ga})}$ inside $\La$.
     \vskip.5cm

\centerline  {\it The accuracy parameter $\zeta$}
 \nopagebreak
Finally, the   parameter $a$ in \eqref{e3.1.1} is not related to a length,
it defines an ``accuracy parameter''
    \begin{equation}
    \label{e3.1.3}
\zeta=\gamma^a
    \end{equation}
whose role will be specified next.

\vskip1cm

\subsection{Local equilibrium}
    \label{sec:e3.2}
A particles configuration  $q$ is a
sequence $(...r_i,s_i....)$ such that for any compact set
$\La$ and any $s\in \{1,..,S\}$,
    \begin{equation}
      \label{e3.2.1}
n(x,s):=|q(s)\cap \La| <\infty,\;\; q(s)=\{r_i,s_i\in q: s_i=s\}
     \end{equation}
We then associate to any such $q$ the empirical densities
   \begin{equation}
      \label{e3.2.2}
 \rho^ {(\ell)}(q;r,s):= \frac{ |q(s)\cap
 C_r^ {(\ell)}|} {\ell^d},\quad s\in \{1,\dots,S\}
     \end{equation}
as functions on $\mathbb R^d\times\{1,..,S\}$ and
 the ``local
phase indicators'' first for any $\rho\in L^1(\mathbb R^d\times
\{1,..,S\})$ ($\rho^{(k)}$ below as in Theorem \ref{thme2.2})
      \begin{equation}
      \label{e3.2.3}
\eta^{(\zeta,\ell)}(\rho;r) = \begin{cases} k &\text{if
$\dis{|\mintone{C^{(\ell)}_r} [\rho(r',s) -\rho^{(k)}_s]|\le
\zeta}$, for all $s\in \{1,..,S\}$}
\\0&\text{otherwise}\end{cases}
     \end{equation}
and then for any particles configuration $q$ as above,
    \begin{equation}
      \label{e3.2.4}
\eta^{(\zeta,\ell)}(q;r) =  \eta^{(\zeta,\ell)}\big(\rho^ {(\ell)}(q;\cdot);r\big)
     \end{equation}
With $\zeta$  and $\ell_{-,\ga}$
 as in \eqref{e3.1.3}
and  \eqref{e3.1.2}, we then define
    \begin{equation}
      \label{e3.2.5}
\mathcal X^{(k)}:=\Big\{ q:
\eta^{(\zeta,\ell_{-,\ga})}(q;r)=k, \,\,\text{for all $r\in \mathbb R^d$} \Big\}
     \end{equation}
$\mathcal X^{(k)}$ is the restricted phase space and
the configurations in
$\mathcal X^{(k)}$ are said to be in local equilibrium in the phase $k$.
Their restrictions to a $\mathcal D^{(\ell_{-,\ga})}$-measurable set $\La$
is denoted by $\mathcal X^{(k)}_\La$ and we will study (in the simplest case)
the Gibbs measure with Hamiltonian $H_\la$ as in \eqref{e1.4} on the phase space
restricted to $\mathcal X^{(k)}$.  As mentioned
in the beginning of this section to apply Pirogov-Sinai  we
will need to complicate the picture, by adding a ``polymer structure''
to the phase space and by modifying the Hamiltonian $H_\la$.

\vskip1cm

\subsection{Polymer configurations}
    \label{sec:e3.3}

A polymer is a pair $\Ga=({\rm sp}(\Ga),\eta_\Ga)$,
${\rm sp}(\Ga)$, the spatial support of $\Ga$,
is a bounded,  connected $\mathcal D^{(\ell_{+,\ga})}$-measurable region
and $\eta_\Ga$, its specification, a
$\mathcal D^{(\ell_{-,\ga})}$-measurable function
on ${\rm sp}(\Ga)$ with values in $\{0,1,..,S+1\}$.
In the applications of Pirogov-Sinai, $\Ga$ will be contours
and $\eta_\Ga$ not as general as above, to keep it simple
we skip all that sticking to the above definition.
We tacitly fix in the sequel $k\in \{1,..,S+1\}$ and
the corresponding phase space $\mathcal X^{(k)}$ and define:

\vskip.5cm

\centerline  {\it Polymer weights}
 \nopagebreak
The weight of $\Ga$ is a function $w(\Ga;q)$, $q\in
\mathcal X^{(k)}$, (dependence on $k$ is not made explicit
in $w$) which depends on the restriction of $q$ to
$\delta_{\rm out}^{\ga^{-1}}[{\rm sp}(\Ga)]$ and which
satisfies the bound
   \begin{equation}
      \label{e3.3.1}
      \sup_{q\in \mathcal X^{(k)}}|w(\Ga;q)|  \le   e^{-c_{\rm
pol} \zeta^2 \ell_{-,\ga}^d N_\Ga},\quad  N_\Ga=\frac{|{\rm sp}(\Ga)|}{\ell_{+,\ga}^d}
     \end{equation}

\vskip.5cm

\centerline  {\it Polymer configurations and weights}
 \nopagebreak
We denote by $\und \Ga$ sequences $...\Ga_i...$ of polymers
with the restriction that any two polymers $\Ga_i$ and
$\Ga_j$, $i\ne j$, are mutually disconnected (i.e.\ the
closures of their spatial supports do not intersect and
they are therefore at least at mutual distance
$\ell_{+,\ga}$).  The collection of all such sequences is
denoted by $\mathcal{B}$ and $\mathcal{B}_\La$,
$\La$ a $\mathcal D^{(\ell_{+,\ga})}$-measurable region,
the subset of $\mathcal{B}$ made by sequences whose
elements $\Ga$ have all sp$(\Ga)$ in $\La$;
$\mathcal{B}^0_\La$ subset of  $\mathcal{B}_\La$ with the
further request that  sp$(\Ga)$ is not connected to
$\La^c$. If $\und \Ga\in \mathcal{B}$ is a finite sequence
we define its weight as
   \begin{equation}
      \label{e3.3.2}
     w(\und \Ga;q)= \prod_{\Ga\in \und \Ga}w(\Ga;q)
     \end{equation}

\vskip1cm

\subsection{The interpolated Hamiltonian}
    \label{sec:e3.4}
Pirogov-Sinai applications also require to change the
Hamiltonian. Let $\La$ be a bounded,
$\mathcal D^{(\ell_{+,\ga})}$-measurable region, $q_\La\in \mathcal X_\La^{(k)}$,
then the ``reference Hamiltonian'' in $\La$ is
    \begin{equation}
    \label{e3.4.1}
h_\La(q_\La)=\sum_s\Big[ \big(\sum_{s'\ne s}
\rho^{(k)}_{s'}\big)-\la_\beta\Big]\;\ell_{0}^{d}\sum_{x\in
\ell_{0}\mathbb Z^d\cap
 \La}\rho^{(\ell_{0})}(q_\La;x,s)
    \end{equation}
where  $\la_\beta$ is the chemical potential introduced in
Theorem \ref{thme2.1}, $\ell_0$ is defined in Subsection
\ref{sec:e3.1}, $\rho^{(\ell_0)}$ in \eqref{e3.2.2}.

For any $t\in[0,1]$ we then define the ``interpolated Hamiltonian''
    \begin{equation}
     \label{e3.4.2}
H_{\La,t}(q_\La|\bar q_{\La^c})=t H_{\La}(q_\La|\bar
q_{\La^c})+(1-t)h_\La(q_\La)
    \end{equation}
where $q_\La \in   \mathcal X^{(k)}_\La$, $\bar
q_{\La^c}\in   \mathcal X^{(k)}_{\La^c}$ and
    \begin{equation}
     \label{e3.4.3}
H_{\La}(q_\La|\bar
q_{\La^c})= H (q_\La\cup\bar
q_{\La^c})- H (\bar q_{\La^c})
    \end{equation}
$H $ as in \eqref{e1.1} with  $\la$
such that $|\la-\la_\beta|\le c \ga^{1/2}$.  Since
$H_{\La,1}(q_\La|\bar q_{\La^c})= H_{\La}(q_\La|\bar
q_{\La^c})$ and $H_{\La,0}(q_\La|\bar q_{\La^c})= h_\La(q_\La)$,
$H_{\La,t}$ interpolates between the true and the reference Hamiltonians.

As we will see in \cite{DMPV2}, $H_{\La,t}(q_\La|\bar q_{\La^c})$
enters in the analysis of the finite volume corrections to
the pressure, a key step in the
implementation of the Pirogov-Sinai strategy.

\vskip1cm

\subsection{DLR  measures}
    \label{sec:e3.5}

The finite volume Gibbs measure in $\La$, $\La$  a bounded,
$\mathcal D^{(\ell_{+,\ga})}$-measurable region, with boundary
condition $\bar q_{\La^c}$, is the following probability on
$\mathcal X_\La^{(k)}\times \mathcal{B}^0_\La$
        \begin{equation}
    \label{e3.5.1}
dG_\La(q_\La, \und \Ga|\bar q_{\La^c}):=\frac {w(\und \Ga;q)
e^{-\beta H_{\La,t} (q_\La| \bar q_{\La^c}) }}{Z_\La(\bar
q_{\La^c})}  d \nu_\La(q_\La)
     \end{equation}
where the free measure $ d \nu_\La(q_\La)$ is
    \begin{equation}
    \label{e3.5.2}
\int_{\mathcal X_\La^{(k)}}f(q_\La)d\nu_\La(q_\La)=\sum_{n=0}^\infty \frac
1{n!}\sum_{s_1,\dots, s_n}\int_{\La^n} f(r_1,s_1,\dots,
r_n,s_n)dr_1\cdots dr_n
    \end{equation}
and where the partition function $Z_\La(\bar q_{\La^c})$
is the normalization factor
which makes the above a probability.  In \eqref{e3.5.1}
the boundary conditions only involve particles
configurations, to define the DLR measures we also need to condition on
the outside polymers.

\vskip.5cm

\centerline  {\it DLR measures}
 \nopagebreak

Given $\und \Ga \in \mathcal{B}$,  $\und
\Ga=(\Ga_1,\Ga_2,\dots)$, we call $\und \Ga_{\La^c}$ the
collection of all pairs $({\rm sp}(\Ga_i)\cap \La^c,
\eta_{{\rm sp}(\Ga_i)\cap \La^c})$ where
$\eta_{{\rm sp}(\Ga_i)\cap \La^c}$ denotes the
restriction of $\eta_\Ga$ to ${\rm sp}(\Ga)\cap \La^c$. 
We then define the probability $dG(q_\La, \und \Ga|\bar
q_{\La^c},\bar{\und \Ga}_{\La^c})$ on $\mathcal
X_\La^{(k)}\times \mathcal B$  by
        \begin{equation}
    \label{e3.5.3}
dG_\La(q_\La, \und \Ga|\bar q_{\La^c},\bar{\und \Ga}_{\La^c}):=
\frac {\text{\bf 1}_{\und \Ga_{\La^c}=\bar{\und
\Ga}_{\La^c}}}{Z_\La(\bar q_{\La^c},\bar{\und \Ga}_{\La^c})}
e^{-\beta H_{\La,t} (q_\La| \bar q_{\La^c}) } \left\{ \prod_{\Ga\in \und
\Ga: {\rm sp}(\Ga)\cap \La\ne\emptyset}w(\und \Ga;q)\right\}
 d \nu_\La(q_\La)
     \end{equation}

\vskip.5cm

A probability $\mu$ on   $\mathcal X^{(k)}\times \mathcal
B$  is DLR if the two properties below hold.

$\bullet$\; it verifies the Peierls bound: for any
$\Ga_1,..,\Ga_k$,
   \begin{equation}
      \label{e3.5.3.0}
  \mu\Big( \{\und \Ga\ni\Ga_1\}\cap \cdots\cap\{\und \Ga\ni\Ga_k\}\Big)   \le   e^{-c_{\rm
pol} \zeta^2 \ell_{-,\ga}^d (N_{\Ga_1}+..+N_{\Ga_k})}
     \end{equation}
 $\bullet$\;
for  any  bounded,
$\mathcal D^{(\ell_{+,\ga})}$-measurable region $\La$ the
conditional probability of $\mu$ given that the particles
configurations in $\La^c$ is   $\bar q_{\La^c}$
and that $\und \Ga_{\La^c}=\bar{\und \Ga}_{\La^c}$ is
$dG_\La(q_\La, \und \Ga|\bar q_{\La^c},\bar{\und \Ga}_{\La^c})$ as given by \eqref{e3.5.3}.

\vskip.5cm

 A few remarks on the above definitions: the
Gibbs measures $dG_\La(q_\La, \und \Ga|\bar q_{\La^c})$
satisfy the Peierls bound \eqref{e3.5.3.0}.  Indeed given
any $\Ga_1,..,\Ga_k$ in $\mathcal B^0_\La$ such that
sp$(\Ga_i)$ is not connected to sp$(\Ga_j)$ for any $i\ne
j$, then, for any $q_\La$,
   \begin{eqnarray*}
&&\hskip-1cm \sum_{\und \Ga\in \mathcal B_\La:
\Ga_1,..\Ga_k \in \und \Ga} w(\und \Ga,q_\La) =
\{\prod_{i=1}^k w(\Ga_i,q_\La)\}\sum_{\und \Ga\in \mathcal
B_\La: \Ga_1,..\Ga_k \in \und \Ga} \;\;\prod_{\Ga\in \und
\Ga, \Ga\ne \Ga_i, i=1,..,k} w(\Ga,q_\La)\\&&\hskip2cm \le
\{\prod_{i=1}^k w(\Ga_i,q_\La)\} \sum_{\und \Ga\in \mathcal
B_\La} w(\Ga,q_\La)
     \end{eqnarray*}
and \eqref{e3.5.3.0} follows from \eqref{e3.3.1}.  On the
other hand we have not specified all the properties of the
weights as they arise in the applications (to the continuum
Potts model) so that in the present context wild things may
happen. For instance  weights still compatible with
\eqref{e3.3.1} may be such that whenever sp$(\Ga)$ contains
$\delta_{\rm out}^{\ell_{+,\ga}}[\Delta]$, $\Delta$ a
bounded, simply connected $\mathcal D^{(\ell_{+,\ga})}$
measurable set, then $w(\Ga,q)=0$ unless sp$(\Ga)\supset
\Delta$.  If the weights had such a property then there are
sequences of finite volume Gibbs measures whose limits are
not supported by $\und \Ga\in \mathcal B$.  Thus a support
property  like \eqref{e3.5.3.0} is necessary in the present
context.

\vskip1cm

\subsection{Main result}
    \label{sec:e3.6}
We fix   $k\in \{1,..,S+1\}$, the statements below
being valid for any such $k$ and for all $\ga$ small enough.
We will employ the following notion: $(q,\und \Ga)$ agrees with
$(q',\und \Ga')$ in $\Delta$
($\Delta$ a $\mathcal D^{(\ell_{+,\ga})}$-measurable set) if
all $\Ga\in \und \Ga$  such that the closure
of sp$(\Ga)$ intersects $\Delta$ are also in $\und \Ga'$
and viceversa and moreover
    \begin{equation}
      \label{e3.6.0}
q\cap \Delta^* = q'\cap \Delta^*,\quad \Delta^*:=
\Delta \bigcup_{\Ga\in \und\Ga} \{{\rm sp}(\Ga)\cup
\delta_{\rm out}^{(\ell_{+,\ga})}
[{\rm sp}(\Ga)]\}
    \end{equation}

\vskip1cm

    \begin{thm}
    \label{thme3.6.1}
For all $\ga$ small enough there is a unique DLR measure $\mu$ and
there are constants $c_1$ and $c_2$ such that the following holds. For
any bounded, $\mathcal D^{(\ell_{+,\ga})}$-measurable regions $\La$
and $\La'\supset \La$ and any boundary conditions ${\bar
q}'_{\La^c}$ and ${\bar q}''_{{\La'}^c}$ there is a coupling $dQ$ of
$dG_\La(q_\La, \und \Ga|{\bar q}'_{\La^c})$ and $dG_{\La'}(q_{\La'},
\und \Ga|{\bar q}''_{{\La'}^c})$ such that if $\Delta$ is any
$\mathcal D^{(\ell_{+,\ga})}$-measurable subset of $\La$:
    \begin{equation}
      \label{e3.6.1}
      \dis{
  Q\Big(\{\text{\,$(q'_\La,\und \Ga')$ and $(q''_{\La'},\und \Ga'')$
  agree in $ \Delta$} \}\Big)
  \ge 1- c_1 e^{- c_2 \frac {{\rm
  dist}(\Delta,\La^c)}{\ell_{+,\ga}}}}
     \end{equation}

    \end{thm}

\vskip1cm

\subsection{A finite size condition}
    \label{subsec:e3.7}

The proof of Theorem \ref{thme3.6.1}  follows the Dobrushin
Shlosman approach: we first introduce and verify a finite
size condition and then prove that this implies uniqueness
and exponential decay. In this subsection we describe the
former step. Let $\La$ be a $\mathcal
D^{(\ell_{+,\ga})}$-measurable, connected region contained
in $\La^*$ where $\La^*$ is obtained by taking a cube $C\in
\mathcal D^{(\ell_{+,\ga})}$, then considering $A:=C\cup
\delta_{\rm out}^{\ell_{+,\ga}} [C]$ and finally $\La^* = A
\cup \delta_{\rm out}^{\ell_{+,\ga}} [A]$.  All the bounds
we will write must be uniform in such a class. Notice that
the diameter of $\La$ is $>\ell_{+,\ga}$ which for $\ga$
small is much larger than the interaction range, in this
sense $\La$ is ``large'' and we are away from the
Dobrushin's high temperatures uniqueness scenario.

Our finite size condition involves only   Gibbs measures without
polymers: namely the probability on $\mathcal X_\La^{(k)}$ defined
for any given $\bar q_\La\in \mathcal X_{\La^c}^{(k)}$ as follows
        \begin{equation}
    \label{e3.7.1}
dG^0_\La(q_\La|\bar q_{\La^c}):=\frac { e^{-\beta H_{\La,t} (q_\La|
\bar q_{\La^c}) }}{Z^0_\La(\bar q_{\La^c})}  d \nu_\La(q_\La)
     \end{equation}
We want to compare two such measures with different boundary
conditions ${\bar q}'_{\La^c}$ and ${\bar q}''_{\La^c}$, thus
introducing the product space $\mathcal X_\La^{(k)} \times
\mathcal X_\La^{(k)}$ whose elements are denoted by
$(q'_\La,q''_\La)$. The finite size condition requires that there
is a coupling $dQ$ of $dG^0_\La(q_\La|{\bar q}'_{\La^c})$ and
$dG^0_\La(q_\La|{\bar q}''_{\La^c})$ with the property that the
event we define below has  a ``large probability''.

\vskip.5cm

 \centerline{\it Notation}
 \nopagebreak
Let  $\bar m= 2^{d}+2$ and  $c_{\rm acc}=2c^*$ with $c^*$ as in Theorem
\ref{thmee5.0} below. Call $\zeta_n := c_{\rm acc}^{-n}\zeta$ and
define a partition of $\mathbb R_+$ into the intervals
$[0,\zeta_{\bar m})$, $[\zeta_{\bar m},\zeta_{\bar m
-1})$,...,$[\zeta_{3},\zeta_{2})$, $[\zeta_{2},\infty)$.

\vskip.5cm

 \begin{defin}{\bf The function $K_\La(\cdot)$ and the set $\Theta_\La(\cdot)$.}
 \label{kappa}
 \nopagebreak

\bigskip

We denote by
            \begin{equation}
            \label{3.22}
A_x:=B_x(10^{-10}\ell_{+,\ga}) \cap \La^c,\qquad {\text {$B_x(R)$
the ball of center $x$ and radius $R$}}
    \end{equation}
Given $\bar q'_{\La^c}$ and $\bar q''_{\La^c}$, we define the
function $K_\La({\bar q}'_{\La^c},{\bar q}''_{\La^c};x)$, $x\in
\ell_{-,\ga}\mathbb Z^d \cap \La$ as follows.

If $A_x=\emptyset$ then $K_\La({\bar q}'_{\La^c},{\bar
q}''_{\La^c};x)=\bar m+1$.

If $A_x\ne \emptyset$ and ${\bar q}'_{\La^c}\cap A_x\ne {\bar
q}''_{\La^c}\cap A_x$, then $K_\La(q'_{\La^c},q''_{\La^c};x)=0$.

If $A_x\ne \emptyset$ and $q'_{\La^c}\cap A_x= q''_{\La^c}\cap A_x$,
call $b:=\dis{\max_{r\in A_x, s\in \{1,..,S\}}|
\rho^{(\ell_{-,\ga})}({\bar q}'_{\La^c};r,s)-\rho^{(k)}_s|}$, then
if $b\in [\zeta_{m+1},\zeta_m)$ for some $m \geq 2$, we set $K_\La({\bar
q}'_{\La^c},{\bar q}''_{\La^c};x)=m$, otherwise we set $K_\La({\bar q}'_{\La^c},{\bar q}''_{\La^c};x)=0$.

\bigskip

The set $\Theta_\La(x)=\Theta_\La({\bar q}'_{\La^c}{\bar
q}''_{\La^c};x)$, $x\in \ell_{-,\ga}\mathbb Z^d\cap \La$,
is defined as the whole space {$\{q'_\La,q''_\La\}$}
if $K(\cdot;x)=K_\La({\bar q}'_{\La^c},{\bar q}''_{\La^c};x)=0$ and
otherwise by
    \begin{eqnarray}
      \label{e3.7.1.0}
&&\hskip-1cm \Theta_\La(x)=\Big\{q'_\La,q''_\La:\;\;
q'_\La\cap C_x^{(\ell_{-,\ga})}= q''_\La \cap
C_x^{(\ell_{-,\ga})}, \nn\\&&\hskip2cm
 \max_{s\in\{1,..,S\}}|\rho^{(\ell_{-,\ga})}({
q}'_{\La};x,s)-\rho^{(k)}_s|
 \le \zeta_{K(\cdot;x)-1}\Big\}
     \end{eqnarray}
 \end{defin}

 \vskip1cm
In section \ref{subsec:7.5}, we will use Theorem \ref{thme3.7.1}
below with $n = 5^d-1$ and $\La \subset \La^\ast$. Recalling
the definition of $N_\La$ in \eqref{ea3.1.1b}, we state:

\label{t3}
\label{y2}
    \begin{thm}
    \label{thme3.7.1}
 For any integer $n>0$ there exist $\ga_n>0$ and $\eps_n< 1$ such that for all
$\ga < \ga_n$ and for any $\La$ with $N_{\La}\le n$,
for any ${\bar q}'_{\La^c}$
and ${\bar q}''_{\La^c}$ as above, there is a coupling $dQ_\La$ of
$dG^0_\La(q_\La|{\bar q}'_{\La^c})$ and $dG^0_\La(q_\La|{\bar
q}''_{\La^c})$ such that with $K(\cdot;x)=K_\La({\bar
q}'_{\La^c},{\bar q}''_{\La^c};x)$ and
$\Theta_\La(x)=\Theta_\La({\bar q}'_{\La^c}{\bar q}''_{\La^c};x)$
defined above,

    \begin{eqnarray}
      \label{e3.7.1a}
&&\hskip-1cm
Q_\La\Big(\bigcap_{x\in \ell_{-\ga}\mathbb Z^d\cap
\La }
\Theta_\La(x) \Big) \ge 1 - \eps_n
     \end{eqnarray}
%

    \end{thm}

\vskip1cm

The proof of Theorem \ref{thme3.7.1} is given in Part II of
this paper.  It consists of three parts, in the first one
we use a step of the renormalization group to describe the
marginal of $dG^0_\La$ over the variables
$\{\rho^{(\ell_{-,\ga})}(x,s), x\in \ell_{-,\ga}\mathbb Z^d
\cap \La, s\in \{1,..,S\}\}$.  Their distribution is proved
to be Gibbsian with an  effective Hamiltonian at the
inverse effective temperature $\beta \ell_{-,\ga}$.  In a
second part we study the ground states of the effective
Hamiltonians, proving exponential decay from the boundary
conditions.  In a third and final part we bound  the
Wasserstein distance between the Gibbs measures by
approximating the latter to Gaussian distributions
describing fluctuations around the ground states
characterized in the previous step.

\vskip2cm

\subsection{Disagreement percolation}
    \label{sec:e3.8}
The finite size condition established in Theorem
\ref{thme3.7.1} is used to construct the coupling $Q$ of
Theorem \ref{thme3.6.1}. The proof uses the ideas
introduced by van der Berg and Maes in their disagreement
percolation paper, \cite{vm}. The proof given in Part III
of this paper consists of two steps. In the first one we
introduce set-valued stopping times, called stopping sets,
and prove that monotone sequences of stopping sets define
couplings of the Gibbs measures and that if the sequence
stops, then in the last set there is agreement.  In the
second and last step we prove that the probability that the
sequence stops late is related to a percolation event which
is then shown to have exponentially small probability.

\vskip3cm

\part { The finite size condition}


%

\vskip2cm

\section{{\bf Effective Hamiltonians}}
        \label{sec:e4}

 We will use the following notations.

\subsection{General notation for Part II}
 \label{subsec:e4.1}

$\bullet$\; By default in this section  $\La$ is  a
connected, $\mathcal D^{(\ell_{+,\ga})}$-measurable region
contained in $\La^*$, see  Subsection \ref{subsec:e3.7},
and regions in $\mathbb R^d$ are all $\mathcal
D^{(\ell_{-,\ga})}$-measurable. To discretize $\mathbb R^d$
we will use the lattice $\ell_{-,\ga} \mathbb Z^d$. Thus in
the sequel $\ell_{-,\ga}$ is the basic mesh. We  define
  \begin{equation}
      \label{e4.1.1}
J_\ga^{(\ell)}(x,y)=
\mintone{C^{(\ell)}_x}\mintone{C^{(\ell)}_y}
J_\ga(r,r'),\quad x,y \in \ell \mathbb Z^d,
\;\;\ell=\ell_{-,\ga}
     \end{equation}

$\bullet$\; The basic variables are the densities
$\rho_\Delta=\{\rho_\Delta(x,s)\ge 0, x\in \ell_{-,\ga} \mathbb
Z^d \cap \Delta,\,s\in \{1,..,S\} \}$, $\Delta \subset \mathbb
R^d$, (by default variables denoted by $\rho$ are non negative
densities). Call $X_\Delta^{(k)}$ the set of all $\rho_\Delta$
such that $n_\Delta:=\ell_{-,\ga}^d\rho_\Delta$ has integer
values, so that $X_\Delta^{(k)}$ is the range of values of the
densities $\rho^{(\ell_{-,\ga})}_\Delta(q_\Delta;x,s)$ when
$q_\Delta \in \mathcal X^{(k)}_\Delta$, $x\in \ell_{-,\ga}\mathbb
Z^d\cap \Delta$, $s\in \{1,..,S\}$; $\rho^{(\ell)}_\Delta$ being
defined in   \eqref{e3.2.2}.

$\bullet$\; To have lighter notation we will use the label
$i$  for a pair $(x,s)$, $x\in\ell_{-,\ga} \mathbb Z^d, s
\in\{1,..,S\}$, writing $x(i)=x$, $s(i)=s$ if $i=(x,s)$ and
sometimes shorthand $|i-j|$ for $|x(i)-x(j)|$ and $i\in
\La$ for $x(i)\in \ell_{-,\ga} \mathbb Z^d\cap\La$.

$\bullet$\; $\mathcal H$ denotes the Euclidean space of
vectors $u=\big(u(i), i \in \La\big)$ with the usual scalar
product $\dis{(u,v)= \sum_{i}  u(i) v(i)}$.  By an abuse of
notation we also denote by $\mathcal H$ the Hilbert space
with $\La$ above replaced by $\mathbb R^d$.

\vskip1cm


\subsection{The effective Hamiltonian}
        \label{subsec:e4.2}


The effective Hamiltonian $ H^{{\rm eff}}_\La(\rho_\La|\bar
q_{\La^c})$, $\rho_\La \in X_\La^{(k)}$, $\bar q_{\La^c}\in
\mathcal X_{\La^c}^{(k)}$, is defined by the equality
   \begin{equation}
     \label{ee4.2.1}
e^{-\beta \ell_{-,\ga}^d H^{{\rm eff}}_\La(\rho_\La|\bar
q_{\La^c})} := \int_{\{\rho( q_{\La};\cdot)=\rho_\La \}} e^{-\beta
H_{\La,t}(q_\La|\bar q_{\La^c})} \nu_\La(dq_\La)
     \end{equation}
$ H_{\La,t}$ as in  \eqref{e3.4.2}, so that $\beta
\ell_{-,\ga}^d$ is the effective inverse temperature. The
Gibbs measure with Hamiltonian $H^{{\rm
eff}}_\La(\rho_\La|\bar q)$, inverse temperature $\beta
\ell_{-,\ga}^d$ and free measure the counting measure on
$X_\La^{(k)}$ is then the marginal over the variables
$\{\rho_\La\in X^{(k)}_\La\}$ of  the Gibbs measure
$dG^0_\La(q_\La|\bar q_{\La^c})$ defined in \eqref{e3.7.1}.

Since $\ell_{-,\ga}=\ga^{-1+\alpha_-}$ and $\alpha_-$ is
small, the effective temperature vanishes as $\ga\to 0$,
and the analysis of the Gibbs measure becomes intimately
related to the study of the ground states of $ H^{{\rm
eff}}_\La$. This will be the argument of the next section,
in this one we determine  $ H^{{\rm eff}}_\La$. In this
subsection we describe its main terms and state the main
theorem; in the successive ones we give the proof.

%


%

\vskip.5cm

 \centerline{\it The LP term.}
  \nopagebreak
The main contribution to the effective Hamiltonian will be
the Lebowitz-Penrose free energy functional, the LP term in
the title of the paragraph.  This is
    \begin{equation}
     \label{ee4.2.2}
F_\La(\rho_\La|\bar \rho_{\La^c})= t\Big\{\frac 12
(\rho_\La,\bar V_\ga \rho_\La) + (\rho_\La,\bar V_\ga\bar
\rho_{\La^c})\Big\} -\frac{1}{\beta} (1_\La, \mathcal
I(\rho_\La)) +(1-t) (\rho^{(k)}1_\La,  \rho_\La)
     \end{equation}
where we employ the usual vector notation: if $A(i,j)$ is a
matrix, $u(i)$ a vector in $\mathcal H$,
         \begin{equation}
     \label{ee4.2.3}
 \big(u, v\big)= \sum_{i}
 u(i) v(i),\quad Au(i)= \sum_{j}A(i,j)u(j)
     \end{equation}
calling $1_\La$ the  vector $ 1_\La(i)= 1$ if $i\in \La$
and $=0$ otherwise.  In \eqref{ee4.2.2}
         \begin{equation}
     \label{ee4.2.3a}
     \bar V_\ga(i,j)= \ell_{-,\ga}^d
\sum_{y\in \ell_{-,\ga}\mathbb Z^d}
J_\ga^{(\ell_{-,\ga})}(x(i),y) \ell_{-,\ga}^d
J_\ga^{(\ell_{-,\ga})}(y,x(j))  \text{\bf 1}_{s(i)\ne s(j)}
     \end{equation}
The normalization is such that $\bar V_\ga$ is a
probability kernel.  The term $(1_\La, \mathcal
I(\rho_\La))$ in \eqref{ee4.2.2} is ``the entropy minus the
chemical potential energy'': 
         \begin{equation}
     \label{ee4.2.4}
\mathcal I(\rho_\La)(i)= \mathcal I^*(\rho_\La(i)),\quad
\mathcal I^* (b):=  -b (\log b -1) +\beta \la_\beta b
    \end{equation}
When $t=1$, $F_\La$ is just the usual LP free energy and for this
reason we    call   $F_\La$ the LP term.  Notice that if
$\rho_\La(i)=\rho^{(k)}_{s(i)} 1_\La(i)$, then the bulk terms of
$F_\La$ which are proportional to $t$ cancel, this will play an
important role in the study of the ground states.

\vskip.5cm

 \centerline{\it The one body effective potential.}
  \nopagebreak
This term is due to second order terms in the Stirling formula
when computing the entropy contribution. It has the
form:
              \begin{equation}
     \label{ee4.2.5}
H^{(1)}_\La(\rho_\La) = \frac{\ell_{-,\ga}^{-d}}\beta \big(  1_{\La}, \log
\sqrt{2\pi\ell_{-,\ga}^{d}\rho_\La} +t [\la_\beta-\la]\rho_\La\big)
     \end{equation}

\vskip.5cm

 \centerline{\it The many-body effective potential.}
  \nopagebreak
This term denoted by $H^{(2)}_\La(\rho_\La|\bar
q_{\La^c})$, takes into account variations of the potential
energy inside the elementary cells $C_x^{(\ell_{-,\ga})}\in
\mathcal D^{(\ell_{-,\ga})}$ which have been neglected in
the LP term.  The dependence of $H^{(2)}_\La$ on $\rho_\La$
is very simple, it is in fact a polynomial of order $<N$,
$N$ a suitable positive integer.  The coefficients of the
polynomial are described next, they have a simpler form
once we use Poisson polynomials. We denote by $\pi_k(n)=
n(n-1)\cdots (n-k+1)$, $k\in \mathbb N_+$, $n\in \mathbb
N_+$, 
the Poisson polynomial of order $k$
and, by an abuse of notation we write
   \begin{equation}
      \label{ee4.2.6}
\pi^*_k(\rho) = \ell_{-,\ga}^{-dk}\pi_k(n),\;\;
\rho=\frac{n}{\ell_{-,\ga}^d}
     \end{equation}

We shorthand $\und i=(i_1,..,i_n)$, $n< N$, and call
$n=n(\und i)$; $\und i \cap \La \ne  \emptyset$ meaning
that there is $i_h\in \und i$ such that $i_h\in \La$. Given
$\und i$ we denote by $k(\und i) =(k(i_1),..,k(i_n))$, with
$k(i_h)$ positive integers, calling $\dis{|k(\und i)|
=\sum_{h=1}^{n(\und i)}k(i_h)}$. We finally call
$\bar\rho_{\La^c}(i):= \rho^{(\ell_{-,\ga})}(\bar
q_{\La^c};i)$ and denote by $\rho(i)$ the function equal to
$\rho_\La(i)$ and to $\bar \rho_{\La^c}(i)$ when $i\in
\La$, respectively $i\in \La^c$; $a_0$ below is a positive
number $<1$. Then $H^{(2)}_\La$ has the form:
         \begin{equation}
     \label{ee4.2.7}
H^{(2)}_\La(\rho_\La|\bar q_{\La^c}) = \sum_{\und i \cap
\La\ne \emptyset} \;\; \sum_{k(\und i):2\le |k(\und i)|< N}
(\ga\ell_{-,\ga})^{a_0 |k(\und i)|}\Phi(\und i,k(\und i),
\bar q_{\La^c, \und i}) \prod_{h=1}^{n(\und i)} \pi^*_{k(
i_h)}(\rho(i_h))
     \end{equation}
$\Phi$ are coefficients which may depend on $\bar q_{\La^c}$ but
only if $\und i \cap \La^c\ne \emptyset$, in such a case they only
depend on $\dis{\bar q_{\La^c, \und i}:= \bigcup_{i\in \und i:
x(i)\in \La^c}\{\bar q_{\La^c}\cap C_{x(i)}^{(\ell_{-\ga})}\}}$.
The main features of the coefficients $\Phi$ (whose dependence on
$t$ is not made explicit) is that:
         \begin{equation}
     \label{ee4.2.8}
\Phi(\und i,k(\und i),\bar q_{\La^c, \und i})=0\;\;\text{
if diam$(x(i_1),...,x(i_n)) \ge 2N\ga^{-1}$}
     \end{equation}
and
   \begin{equation}
      \label{ee4.2.9}
\sum_{\und i\ni i_0} \;\;\sum_{k(\und i):2\le |k(\und i)|<
N} \Phi(\und i,k(\und i),\bar q_{\La^c, \und i}) \le
c,\quad \text{for any $i_0$}
     \end{equation}
where   $c>0$ is a constant independent of $\bar q_{\La^c}$
and $t$.

\vskip1cm

          \begin{thm}

    \label{thme4.2.1}

For any $a_0<1$ there are $c$, $N$ and coefficients $\Phi$
as above such that for all $\ga$ small enough
    \begin{equation}
     \label{ee4.2.10}
H^{{\rm eff}}_\La(\rho_\La|\bar q_{\La^c}) =
F_\La(\rho_\La|\bar \rho_{\La^c})
+H^{(1)}_\La(\rho_\La)+H^{(2)}_\La(\rho_\La|\bar
q_{\La^c})+R_\La(\rho_\La|\bar q_{\La^c})
     \end{equation}
with the remainder $ R_\La(\rho_\La|\bar q_{\La^c}) =
 R^{(1)}+R^{(2)}$
    \begin{equation}
      \label{ee4.2.11}
      |R^{(i)}| \le c\ga^{\tau},\; i=1,2
     \end{equation}

  \end{thm}
with $\tau = (3-5\alpha_--2\alpha_+)\frac{d}{2}>0$ (see \eqref{R1} and \eqref{Ae.2.13}).

\vskip1cm

Recall that in this section $\La$ is a subset of $\La^*$
thus $|\La| \le c \ell_{+,\ga}^d$, $c$ a constant, if we
wanted larger volumes we would have to increase $N$, namely
to include more body-potentials and longer interaction
range, the expansion in Theorem \ref{thme4.2.1} being
highly non uniform in $\La$. The proof which follows
closely the one in \cite{LMP} of a similar result, is given
in the remaining subsections.

\vskip1cm


\subsection{Derivation of the LP term}
        \label{subsec:ee4.3}


 We fix  arbitrarily
$\rho_\La \in X_\La^{(k)}$, call $n_\La(i)=\ell_{-,\ga}^d
\rho_\La(i)$, introduce a set of labels $\mathcal L$ 
whose elements are denoted by $\xi=(i,\ell)$, where
$i=(x,s)\in \La$, $\ell\in\{1,..,n_\La(\xi)\}$;  the
coordinate functions on $\mathcal L$ are  $x(\xi)$,
$s(\xi)$ and $\ell(\xi)$ respectively equal to the first,
second and third entry in $\xi$. We then define for $\xi
\in \La$ (meaning $x(\xi)\in \La$) the probability measures
on $\La \times\{1,..,S\}$ as $\dis{dp_{\xi}(r,s)= \text{\bf
1}_{r\in C_{x(\xi)}^{(\ell_{-,\ga})}}\text{\bf
1}_{s}\frac{dr }{\ell_{-,\ga}^d}}$ and call $\dis{d p_\La =
\prod_{\xi\in \mathcal L} dp_{\xi}}$, remembering that this
measure as well as the index set $\mathcal L$ depend on the
initial choice of $\rho_\La$, as this is momentarily fixed
we are not making it explicit. We obviously have:
   \begin{equation}
     \label{ee4.3.1}
e^{-\beta \ell_{-,\ga}^d H^{{\rm eff}}_\La(\rho_\La|\bar
q)} = \Big( \prod_{i\in \La} \frac{\ell_{-,\ga}^{d
n_\La(i)}}{n_\La(i)!}\Big)\; \int e^{-\beta
H_{\La,t}(q_\La|\bar q_{\La^c})} dp_\La
     \end{equation}
where  $q_\La$ on the r.h.s.\ should be thought of as a
$\xi$-labeled configuration of particles (the label
specifying also the cube where the particle is) which is
identified to the integration variable relative to the
measure $dp_\La$: thus the dependence on $\rho_\La$ is
 hidden in the structure of the
probability $dp_\La$. The bracket on the r.h.s.\ is equal
to
   \begin{equation}
     \label{ee4.3.2}
 \prod_{i\in  \La} \frac{\ell_{-,\ga}^{d n_\La(i)}}{n_\La(i)!}
=e^{
 \ell_{-,\ga}^d
\big(1_{\La}, S(\rho_\La)\big)},\quad S(\rho_\La)(i)=
\ell_{-,\ga}^{-d} \big( n \log \ell_{-,\ga}^{d} - \log
n!\big),\;\; n=n_\La(i)=\ell_{-,\ga}^d\rho_\La(i)
      \end{equation}
Then, recalling the Stirling formula:
      \begin{equation}
    \label{ee4.3.3}
n! = n^{n+1/2} e^{-n} \sqrt{2\pi}\Big( 1 +
0\left(\frac{1}{\sqrt n}\right)\Big)
     \end{equation}
we can estimate $\big(1_{\La}, S(\rho_\La)\big)$ as follows
     \begin{equation}
       \label{ee4.3.4}
      \big( 1_{\La}, S(\rho_\La)\big) = \big(1_{\La}, S^{\rm
         app}(\rho_\La)\big)  - \beta H^{(1,0)}  - \beta R^{(1)}
     \end{equation}
where $S^{\rm app}(\rho) = -\rho (\log\rho-1)$ and $H^{(1,0)}$ is equal to the r.h.s. of \eqref{ee4.2.5} with  $t=0$.

\vskip.5cm

{\bf Proof of  \eqref{ee4.2.11} for $R^{(1)}$.}

We now show that $R^{(1)}$ defined in \eqref{ee4.3.4} above satisfies the bound \eqref{ee4.2.11}:
\begin{align*}
\ell_{-,\ga}^{d}\left[S^{\rm app}(\rho_\La)(i) - S(\rho_\La)(i)\right] & =  -n_{\La}(i)\left( \log  n_{\La}(i) -1 \right)+ \log n_{\La}(i)!\\
& = \frac{1}{2}\log n_{\La}(i) + \log \sqrt{ 2\pi } + 0\left(\frac{1}{\sqrt{n_{\La}(i)}}\right),\\
\\
\big(1_{\La}, S^{\rm app}(\rho_\La)\big) - \big( 1_{\La}, S(\rho_\La)\big)& =  \beta H^{(1,0)} + \sum_{i\in \La} 0\left(\ell_{-,\ga}^{-d/2}  \right)
\end{align*}
where  we used the fact that $ n_{\La}(i) \geq c \ell_{-,\ga}^{d}$, since $\rho_{\La} \in  X_\La^{(k)}$. From this, we get
\begin{align}
\nonumber | \beta R^{(1)} | & \leq  \# \left\{i \in \La \right\} \cdot \ell_{-,\ga}^{-3d/2} \\
\label{R1}               & \leq SN_{\La}  \left( \frac{ \ell_{+,\ga}}{ \ell_{-,\ga}}\right)^{d}   \ell_{-,\ga}^{-3d/2} \qed
\end{align}

Call $ \bar H_\La(q_\La|\bar q_{\La^c})$ the energy
$H_{\La,t}(q_\La|\bar q_{\La^c})$ defined with $J_\ga$
replaced by $ J^{(\ell_{-,\ga})}_\ga$, then $\bar
H_\La(q_\La|\bar q_{\La^c})$ depends only on the densities
$\rho^{(\ell_{-,\ga})}(q_\La;i)$ and
$\rho^{(\ell_{-,\ga})}(\bar q_{\La^c};i)$ which in
\eqref{ee4.3.1} are fixed equal to $\rho_\La(i)$ and $\bar
\rho_{\La^c}(i)$, hence
   \begin{equation}
     \label{ee4.3.5}
 \bar
H_\La(q_\La|\bar q_{\La^c})= \ell_{-,\ga}^d\Big\{t\Big(
\frac 12 \big(\rho_\La, \bar V_\ga \rho_\La\big) +
\big(\rho_\La, \bar V_\ga
\bar\rho_{\La^c}\big)-\la(1_\La,\rho_\La)\Big) + (1-t)
(1_\La[\rho^{(k)}-\la_\beta],\rho_\La)\Big\}
      \end{equation}
Collecting all the above terms we thus identify in
\eqref{ee4.3.1}
        \begin{equation}
     \label{ee4.3.6}
e^{-\beta \ell_{-,\ga}^d \{H^{(2)}_\La(\rho_\La|\bar
q_{\La^c}) +R^{(2)}\}} =   \int e^{-\beta
\{H_{\La,t}(q_\La|\bar q_{\La^c})- \bar H_\La(q_\La|\bar
q_{\La^c})\}} dp_\La
     \end{equation}

\vskip1cm


\subsection{Cluster expansion}
        \label{subsec:e4.3}


To  estimate the r.h.s.\  of \eqref{ee4.3.6} we  use
cluster expansion. Call $\mathcal E$ a set of unordered
pairs $(\xi,\xi')$, $\xi\ne \xi'$, then $\mathcal E$
defines a graph structure $(\mathcal L,\mathcal E)$ with
vertices $\xi\in \mathcal L$ and edges $(\xi,\xi')\in
\mathcal E$. We call diagrams the connected sets $\theta$
in $(\mathcal L,\mathcal E)$, $\und
\theta=(\theta_1,..,\theta_n)$ their collection. Call
$\Theta$ and $\Theta_{\rm dsc}$ the spaces of all possible
diagrams and of all possible $\und \theta$ which appear
when varying $\mathcal E$. Let
        \begin{equation}
     \label{Ae.2.5}
w(\theta)=    \int \Big(\prod_{(\xi,\xi')\in \theta,
s(\xi)\ne s(\xi')} \{ e^{-\beta t
\{V_\ga(x(\xi),x(\xi'))-\ell_{-,\ga}^{-d} \bar
V_\ga(x(\xi),x(\xi'))\}}-1\}\Big) dp_\La
     \end{equation}
then, since $dp_\La$ is a product measure,
        \begin{equation}
     \label{Ae.2.6}
\int e^{-\beta \{H_{\La,t}(q_\La|\bar q_{\La^c})- \bar
H_\La(q_\La|\bar q_{\La^c})\}} dp_\La = \sum_{\und \theta
\in \Theta_{\rm dsc}}\prod_{\theta \in \und \theta}
w(\theta)
     \end{equation}
\eqref{Ae.2.6} is derived from \eqref{ee4.3.6} by writing
        \begin{equation*}
e^{-\beta \{H_{\La,t}(q_\La|\bar q_{\La^c})- \bar
H_\La(q_\La|\bar q_{\La^c})\}}=\prod_{(\xi,\xi'): s(\xi)\ne
s(\xi')} \{ e^{-\beta t
\{V_\ga(r_\xi,r_{\xi'})-\ell_{-,\ga}^{-d}\bar
V_\ga(x(\xi),x(\xi'))\}}-1 +1\}
     \end{equation*}
where the labels $\xi$ include both the particles in $\La$
and those of $\bar q_{\La^c}$ outside $\La$. After
expanding the product we then get \eqref{Ae.2.6}, details
are omitted.

The basic condition for cluster expansion which we have in
the present context, involves the elementary diagrams
namely $\theta=(\xi,\xi')$ and states that  given any
$\dis{a>0}$
             \begin{equation}
     \label{Ae.2.7}
 \sum_{\xi'} |w\big((\xi,\xi')\big)| \ga^{-\alpha_-+a} <1,
 \quad \text{for   any $\ga$ small enough}
     \end{equation}
\eqref{Ae.2.7} is proved by observing that the densities
$\rho_\La(i)$ are bounded and that \eqref{Ae.2.5} yields
for $\theta=(\xi,\xi')$
             \begin{equation}
     \label{Ae.2.7.1}
 |w\big((\xi,\xi')\big)|\le c
 \ga^{d}(\ga\ell_{-,\ga})\text{\bf 1}_{{\rm dist}(C_{x(\xi)}^{(\ell_{-,\ga})},
 C_{x(\xi')}^{(\ell_{-,\ga})})\le  \ga^{-1}}
     \end{equation}
``Cluster expansion'' then applies for any $\ga$ small
enough and the following holds (for any $\rho_\La\in
X_\La^{(k)}$).

\vskip.5cm

{\it Notation.} We give $\Theta$ a  graph structure by
calling vertices  the diagrams $\theta\in \Theta$ and
edges   the pairs $\theta$ and $\theta'$ which have non
empty intersection, as sets in $\mathcal L$.

Denote by $m(\theta)$, $\theta\in\Theta$, positive, integer
valued  functions, calling $m(\theta)$ ``the multiplicity''
of $\theta$.  We restrict to $m\in \mathcal M$ where
    \begin{equation}
     \label{Ae.2.8}
 \text{$m\in \mathcal M$ if and only if}\;\;{\rm sp}(m):=
  \{\xi:\xi\in\theta, m(\theta)>0\}\;\; \text{is a connected set}
     \end{equation}
and shorthand  $\xi\in m$ when $\xi\in $ sp$(m)$.

\vskip.5cm

Cluster expansion tells us that given any $a_0<1$ for all
$\ga$ small enough there are coefficients $\om (m )$, $m\in
\mathcal M$, such that
     \begin{equation}
     \label{Ae.2.9}
\log Z (\{w(\cdot)\}):=\log\{ \sum_{\und \theta \in
\Theta_{\rm dsc}}\prod_{\theta \in \und \theta}
w(\theta)\}= \sum_{m\in \mathcal M} \om (m )
     \end{equation}
and, for any $\xi\in \mathcal L$,
     \begin{equation}
       \label{Ae.2.10}
\sum_{m\in \mathcal M:m\ni \xi}\;\;  |\om (m ) |
\{\prod_{\theta: m(\theta)>0}
(\ga\ell_{-,\ga})^{a_0|\theta|_{\rm edg}m(\theta)|}\} < 1
    \end{equation}
where $|\theta|_{\rm edg}$ is the number of edges in
$\theta$. The coefficients  $\om(m)$ have the following
explicit expression:
     \begin{equation}
       \label{Ae.2.11}
 \om(m) = C_m \prod_{\theta: m(\theta)>0}
  w(\theta)
     \end{equation}
where thinking of $Z(\{w(\cdot)\})$ in \eqref{Ae.2.9} as a
function of the weights $\{w(\theta), \theta\in \Theta\}$,
          \begin{equation}
       \label{Ae.2.12}
C_m=   \prod_{\theta: m(\theta)>0} \frac 1{m(\theta)!} \;
\{ \prod_{\theta: m(\theta)>0}\frac{\dis{\partial ^{
m(\theta)}}}{\dis{
\partial
 w(\theta)^{m(\theta)}}}\} \;\; \log Z_\La(w(\cdot))\Big |_{w(\theta)=0}
    \end{equation}
($C_m$ being bounded coefficients independent of $\La$). As
said, all the above follows from the general theory (of
cluster expansion) using the condition \eqref{Ae.2.7}, see
for instance \cite{cluster-expansion}.

\vskip1cm


\subsection{Identification of the many body potential}
        \label{subsec:e4.4}


We will next use \eqref{Ae.2.10} to truncate the sum in
\eqref{Ae.2.9} identifying the remainder with the term
$R^{(2)}$ and recognizing in the finite sum the Hamiltonian
$H^{(2)}_\La(\rho_\La|\bar q_{\La^c})$, for this we will
use the explicit representation of the terms  of the
expansion provided by \eqref{Ae.2.11}--\eqref{Ae.2.12}.

Calling $\dis{|m|= \sum_{\theta\in \Theta}|\theta|_{\rm
edg} m(\theta)}$, by \eqref{Ae.2.10},
 for any $N>0$,
     \begin{eqnarray*}
&&\hskip-1cm \sum_{m\in \mathcal M:|m|\ge N} |\om (m )| \le
\sum_{\xi\in \mathcal L} \;\;\sum_{m\in \mathcal M:m\ni
\xi, |m|\ge N}\;\;  |\om (m ) |\nn\\&& \hskip1cm \le
|\mathcal L| (\ga\ell_{-,\ga})^{a_0 N} \sum_{m\in \mathcal
M:m\ni \xi} |\om (m )| \{\prod_{\theta: m(\theta)>0}
(\ga\ell_{-,\ga}))^{-a_0|\theta|_{\rm edg}m(\theta)| }\}
\le |\mathcal L| (\ga\ell_{-,\ga})^{a_0 N}
     \end{eqnarray*}
Since  $\La\subset \La^*$, there is $c>0$ such that
$|\mathcal L| \le c  \ell_{+,\ga}^{d}$ and we can then
choose $N$ so large that
     \begin{eqnarray}
     \label{Ae.2.13}
&& -\beta \ell_{-}^d R^{(2)}:= \sum_{m\in \mathcal M:|m|\ge
N}  \om (m ), \qquad |\sum_{m\in \mathcal M:|m|\ge N}  \om
(m )|
\le \ell_{-,\ga}^{-d/2} 
     \end{eqnarray}
thus \eqref{ee4.2.11} is satisfied and
\begin{eqnarray}
     \label{Ae.2.14}
&& -\beta \ell_{-,\ga}^dH^{(2)}(\rho_\La|\bar q_{\La^c}):=
\sum_{m\in \mathcal M:|m|< N} \om (m )
     \end{eqnarray}

The dependence on $\rho_\La$ is hidden in the space
$\Theta$, on which the functions $m$ are defined. Theorem
\ref{thme4.2.1} will be proved once we show that the
r.h.s.\ of \eqref{Ae.2.14} can be written as the r.h.s.\ of
\eqref{ee4.2.7}.

We  rewrite the r.h.s.\ of \eqref{Ae.2.14} by first summing
over all $m$ in ``the same equivalence class'' and then
summing over all equivalence classes. Before defining the
equivalence $m\sim m'$ we observe that if $\psi$ is a one
to one map of $\mathcal L$ onto itself, then $\psi$ extends
naturally to a map of $\Theta$ onto itself by letting
$\psi(\theta)$ be the diagram with vertices $\psi(\xi)$,
$\xi\in \theta$, and edges $(\psi(\xi),\psi(\xi'))$,
$(\xi,\xi')$ the edges of
$\theta$.  
We then call  $m\sim m'$ if there is a one to one map
$\phi$ from $\mathcal L$ onto $\mathcal L$ such that
$\bullet$\; $x(\phi(\xi))=x(\xi)$, $s(\phi(\xi))=s(\xi)$
for
all $\xi$; 
$\bullet$\; $m'(\phi(\theta)) =m(\theta)$ for all
$\theta\in \Theta$.


Calling $[m]$ the equivalence class of $m$, i.e.\ the set
of all $m': m'\sim m$, we define the average weight
     \begin{eqnarray}
     \label{Ae.2.15.0}
&&  \om^*(m):= \frac{1}{\text{\rm card}([m])}\sum_{m' \in
[m]} \om (m')
     \end{eqnarray}
Notice that if sp$(m)$ consists only of $\xi$ such that
$x(\xi)\in \La$ then $\om(m)=\om(m')=\om^*(m)$ for all
$m'\in [m]$. If instead there are labels $\xi$ in sp$(m)$
such that $x(\xi)\in \La^c$ then $\om^*(m)$ is a non
trivial average. Actually the averages involve the labels
$\ell$ in each triple $(x,s,\ell)$, $x\in \La^c$, with
$m(x,s,\ell)>0$.  Calling $K(i;m)$ the number of $\xi\in m$
such that $i(\xi)= i$,
     \begin{eqnarray}
     \label{Ae.2.15}
&& \text{\rm card}([m])= \prod_{i} \pi_{K(i;m)}(n(i))
     \end{eqnarray}
where $\pi_k(n)$ is the Poisson polynomial and
$n(i)=\rho(i) \ell_{-,\ga}^d$.
We then have
     \begin{eqnarray}
     \label{Ae.2.16}
&&\hskip-1cm  -\beta \ell_{-}^dH^{(2)}(\rho_\La|\bar
q_{\La^c}):= \sum_{[m], |m|< N} \om^*(m ) \{\prod_{i}
\pi_{K(i;m)}(n (i))\}
     \end{eqnarray}
We next interchange the sums: for any sequence $K(i)\in
\mathbb N_+$, $\dis{\sum_{i} K(i) <N}$,  let
          \begin{eqnarray}
     \label{Ae.2.17}
&& \Psi(K(\cdot)):= \ell_{-,\ga}^{-d}\sum_{[m], m:
K(\cdot;m)=K(\cdot)} \om^*(m ) \prod_{i}\ell_{-,\ga}^{d
K(i)}
     \end{eqnarray}
then
     \begin{eqnarray}
     \label{Ae.2.18}
&&\hskip-1cm -\beta H^{(2)}(\rho_\La|\bar q_{\La^c}):=
\sum_{K(\cdot)} \Psi(K(\cdot)) \{\prod_{i} \ell_{-,\ga}^{-d
K(i)} \pi_{K(i)}(n(i))\}
     \end{eqnarray}
thus identifying $\Phi$ in   Theorem \ref{thme4.2.1} in
terms of $\Psi$:
     \begin{eqnarray}
     \label{Ae.2.18.1}
     \Psi(K(\cdot))=(\ga\ell_{-,\ga})^{a_0 |K(\und i)|}
\Phi(\und i,   K(\und i),\bar q_{\La^c, \und i})
  \end{eqnarray}
recalling the remark before \eqref{Ae.2.15}, indeed the
l.h.s.\ depends on  $\bar q_{\La^c}$ only via $\bar
q_{\La^c, \und i}$.

Of course we still need to prove that the function $\Phi$
defined via \eqref{Ae.2.18.1} satisfies the bounds stated
in   \eqref{ee4.2.8}--\eqref{ee4.2.9}. Since the
coefficients $C_m$ in \eqref{Ae.2.11}, are bounded, say
     \begin{eqnarray}
     \label{Ae.2.18.111}
 \max_{m:|m|<N} |C_m|\le c_{N}
  \end{eqnarray}
we just need to bound $|w(\theta)|$. The definition of
$w(\theta)$ involves product of terms
$w\big((\xi,\xi')\big)$ for each edge of the diagram which
we bound using \eqref{Ae.2.7.1}. The bound obtained in this
way is the same for all $m'\in [m]$ so that the bound for
$\om^*(m)$ is the same as for $\om(m)$. To fix up the
combinatorics, we proceed as follows. For any $m$ we define
a graph structure $G(m)$ on sp$( m)$ introducing a node for
each element $\xi$ of sp$( m)$ which is then given the
label $i=(x(\xi),s(\xi))$, thus different nodes may have
the same label. Edges in $G(m)$ are the union of all the
edges present in all the diagrams $\theta$ such that
$m(\theta)>0$. Each edge is then given a multiplicity equal
to the sum of all $m(\theta)$ over the diagrams $\theta$
which contain the given edge. With this definition any
$m'\in [m]$ gives rise to the same $G(m)$ as we are only
recording the coordinates  $x(\xi)$ and $s(\xi)$ of $\xi$.

To proceed with the bound we assign a ``weight''
$\ell_{-,\ga}^{d}$ to any node in $G(m)$.  Having \eqref{Ae.2.7.1}in mind, we assign to each edge a weight  $\dis{\Big(c
\ga^{d}(\ga \ell_{-,\ga}) \text{\bf 1}_{{\rm
dist}(C_{x(\xi)}^{(\ell_{-,\ga})},  C_{x(\xi')}^{(\ell_{-,\ga})})\le  \ga^{-1}} \Big)^p}$,
where $p$ the multiplicity of the edge. We
have thus assigned a weight $W(G(m))$ to $G(m)$
equal to the product of the weights of its nodes and of its
edges and, with reference to \eqref{Ae.2.17} and recalling
\eqref{Ae.2.18.111}
          \begin{eqnarray}
     \label{Ae.2.19}
&& |\Psi(K(\cdot))|\le   c_{N} \ell_{-,\ga}^{-d}\sum_{[m],
m: K(\cdot;m)=K(\cdot)} W(G(m) )
     \end{eqnarray}
Recalling that $K(i;m)$ is the number of $\xi\in m$ such
that $i(\xi)=i$, $K(i;m)$ is also the number of nodes in
$G(m)$ with label $i$.  Thus, calling $K(i, G)$ the number
of nodes in $G$ with label $i$, $\und i=\{i, i \in
G\}$, and $K(\und i,G)=\{K(i, G), i \in \und i\}$,
          \begin{eqnarray}
     \label{Ae.2.20}
&& |\Psi(K(\und  i))|\le   c_{N} \ell_{-,\ga}^{-d}\sum_{G:
K(\und i ;G )=K(\und i)} W(G)
     \end{eqnarray}
\eqref{Ae.2.18.1} then yields
     \begin{eqnarray}
     \label{Ae.2.18.1.0}
 |\Phi(\und i,   K(\und i),\bar q_{\La^c, \und i})|\le c_N
 \ell_{-,\ga}^{-d}
(\ga \ell_{-,\ga})^{-a_0|K(\und i)|} \sum_{G: K(i ;G
)=K(i)} W(G)
  \end{eqnarray}
By \eqref{Ae.2.17} the terms  to consider have $\und i$
such that $\dis{\sum_{i\in \und i} K(i) <N}$. Then
$\Phi(\und i,K(\und i),\bar q_{\La^c, \und i})=0$ if
diam$(\und x) \ge 2 \ga^{-1}N$, $\und x$ being the sites
appearing in $\und i$, because the weight of the edges in
$G$ are proportional to $\text{\bf 1}_{{\rm
dist}(C_{x(\xi)}^{(\ell_{-,\ga})},
C_{x(\xi')}^{(\ell_{-,\ga})})\le  \ga^{-1}}$.  

To prove \eqref{ee4.2.9} we fix $i_0$ and restrict the sum
in \eqref{Ae.2.18.1.0} to $G: K(i_0;G)>0$.  For each such
$G$ we can then define a tree structure $T_{i_0}(m)$ in
$G(m)$ with root $i_0$, a first generation made by all
nodes connected to the root, second generation made by the
nodes connected to those of the first generation and so
forth. To recover the original graph we may also have to
add edges connecting individuals of the same generation and
also attribute to each edge its multiplicity, as explained
earlier.  We then have
   \begin{equation}
      \label{Ae.2.21}
\text{l.h.s.\ of \eqref{ee4.2.9}} \;\; \le  \;\; \sum_{\und
i\ni i_0} \;\;\sum_{K(\und i):|K(\und i)|< N}
  \ell_{-.\ga}^{-d}(\ga\ell_{-,\ga})^{-a_0|K(\und i)|} \sum_{T_{i_0}:
K(\und  i ;T_{i_0} )=K(\und i)} W(T_{i_0})
     \end{equation}
Define a new weight $W^*(T)$ by changing the weights of the
edges into
    $$
    \Big( c  (\ga\ell_{-,\ga})^{1-a_0}
\ga^{d}\text{\bf 1}_{|x -x ' |\le 2\ga^{-1}} \Big)^p,\quad
\text{$p$ the multiplicity of the edge}
       $$
 while the weights of the node are unchanged.  Then
    \begin{equation}
      \label{Ae.2.22}
\text{l.h.s.\ of \eqref{ee4.2.9}} \;\;\le  \;\;
\ell_{-,\ga}^{-d} \sum_{\und  i\ni  i_0} \;\;\sum_{K(\und
 i):|K(\und i)|< N}
  \sum_{T_{ i_0}:
K(\und  i ;T_{ i_0} )=K(\und  i)} W^*(T_{ i_0})
     \end{equation}
The weight of the root of the tree cancels with the
prefactor $ \ell_{-,\ga}^{-d}$.  We upper bound the sum on
the r.h.s.\ if we regard a multiple edge with multiplicity
$k$ as $k$ distinct edges originating from a same node and
also regard edges between nodes in the same generation as
edges into the next generation (thus dropping the
constraint that the arrival node is the same as the arrival
node of another edge), each node added in this way getting
an extra weight $\ell_{-,\ga}^d$. In this way we have an
independent branching and since
      $$
 \lim_{\ga\to 0}  \sum_{x'}  (\ga\ell_{-,\ga})^{a_0}
\ga^{d}\text{\bf 1}_{|x'|\le 2\ga^{-1}} \ell_{-}^d =0
       $$
we then get \eqref{ee4.2.9}, details are omitted.  Theorem
\ref{thme4.2.1} is proved. \qed

\vskip2cm

\section{Ground states of the effective Hamiltonian }
        \label{sec:ee5}

In this section we study the ground states of the main term
in the effective Hamiltonian $H^{{\rm
eff}}_\La(\rho_\La|\bar q_{\La^c})$,  which, with reference
to \eqref{ee4.2.10}, is
    \begin{equation}
     \label{ee5.0.1}
f(\rho_\La;\bar q_{\La^c}):= H^{{\rm
eff}}_\La(\rho_\La|\bar q_{\La^c}) - R_\La(\rho_\La|\bar
q_{\La^c})
     \end{equation}
While originally $\rho_\La=\big(\rho_\La(i), i=(x,s), x\in
\ell_{-,\ga}\mathbb Z^d\cap \La, s\in \{1,..,S\} \big)\in
X^{(k)}_\La$ defined in Subsection \ref{subsec:e4.1}, it is
convenient here to extend  the range of values of $\rho_\La(i)$ to
an interval of the real line. We thus call
            $$
Y^{(k)}_\La=\Big\{\rho_\La:
\rho_\La(x,s)\in[\rho^{(k)}_s-\zeta,\rho^{(k)}_s+\zeta], \forall
x\in \ell_{-,\ga}\mathbb Z^d\cap \La, \forall s\in \{1,..,S\}\Big\}
                $$
The ground states in the title are then the minimizers of
$f(\rho_\La;\bar q_{\La^c})$ as a function on $Y^{(k)}_\La$ with
$\bar q_{\La^c}$ regarded as a parameter.

\vskip1cm

Let  $\hat K_\La(x)\equiv \hat K_\La(\bar q'_{\La^c},\bar
q''_{\La^c};x)$ be the function defined as $K_\La(x)$ in Definition
\ref{kappa} but with the set $A_x$ in \eqref{3.22} replaced with the
set
    \begin{equation}
    \label{kappa2}
\hat A_x=B_x(10^{-30}\ell_{+,\ga}) \cap \La^c
    \end{equation}

Our main result is the following theorem:

\vskip1cm

\begin{thm}
    \label{thmee5.0}
    There are $c^*$ and $\hat \om$ positive such that for any $a_0<1$ and
    for all $\ga$ small enough the following holds. For any $\bar
    q_{\La^c}\in \mathcal X^{(k)}_{\La^c}$ there is a unique minimizer
    $\hat \rho_\La$ of $\{f(\rho_\La;\bar q_{\La^c}), \rho_\La\in
    Y^{(k)}_\La\}$.  Let $\hat K(x)$
    $x\in \ell_{-,\ga}\mathbb Z^d\cap \La$,  be as above and  ${\hat
      \rho}'_\La$ and  ${\hat \rho}''_\La$ the minimizers with ${\bar
      q}'_{\La^c}$ and ${\bar q}''_{\La^c}$, then for any
    $s\in\{1,..,S\}$:

    \begin{itemize}

    \item (i) If $\hat K(x)>0$,  $|{\hat \rho}'_\La(x,s)-
      {\hat \rho}''_\La(x,s)| \le c e^{-10^{-30} (\ga\ell_{+,\ga})\hat \om}$.

    \item (ii) If $ \hat K(x)=m>0$,
      $|{\hat \rho}'_\La(x,s)- \rho^{(k)}_s|
      \le c^*(\zeta_m+ (\ga\ell_{-,\ga})^{a_0}+ e^{-10^{-30}
        (\ga\ell_{+,\ga})\hat\om})$,
      with same bound for $
      {\hat \rho}''_\La(x,s)$.

    \end{itemize}

  \end{thm}

\vskip1cm

Existence of a minimizer follows from $f$
being a smooth function on a compact set of the Euclidean
space.  Uniqueness and exponential decay are more difficult
and the proof will take the whole section. The basic
ingredient  is that $D^2f$ (the Hessian matrix
of the derivatives w.r.t.\ the variables $\rho_\La(i)$)
computed on the minimizer in the constraint space
$Y^{(k)}_\La$  is positive and ``quasi diagonal'', which
would then give the required uniqueness and exponential
decay if we had $Df=0$. This is however not necessarily the
case  because  the minimum could be reached on the
boundaries of the domain of definition, which, on the other
hand, is necessary to ensure convexity. We will solve the
problem by relaxing the constraint and then studying the
limit when the cutoff is reconstructed.

\vskip1cm

\vskip1cm


\subsection{Extra notation and definitions}
        \label{subsec:e5.1}

The basic notation are those established in Subsection
\ref{subsec:e4.1}, here we add a few new ones specific to
this section:

\vskip.3cm

$\bullet$\; We will write $f(\rho_\La;\bar q_{\La^c})=
F(\rho_\La;\bar \rho_{\La^c})+g(\rho_\La;\bar q_{\La^c})$
where, recalling \eqref{ee4.2.10},
    \begin{equation}
     \label{ee5.1.1}
g(\rho_\La;\bar q_{\La^c}) =
H^{(1)}_\La(\rho_\La)+H^{(2)}_\La(\rho_\La|\bar
q_{\La^c})
     \end{equation}

$\bullet$\;    To evidentiate some of the variables in
$\rho_\La$, say those in $\Delta\subset \La$, we write
$\rho_\La= (\rho_\Delta,\rho_{\La \setminus \Delta})$,
where $\rho_\Delta$ and $\rho_{\La \setminus \Delta}$ are
the restrictions of $\rho_\La$ to $\Delta$ and respectively
to $\La \setminus \Delta$.

$\bullet$\; It will be convenient to relax the constraint
$\rho_\La\in Y^{(k)}_\La$ by enlarging $Y^{(k)}_\La$ into
$W^{(k)}_{\La}$
       \begin{equation}
       \label{def:Wk}
W^{(k)}_{\La}=\Big\{\rho_\La:
\rho_\La(x,s)\in[\rho^{(k)}_s-b,\rho^{(k)}_s+b], \forall x\in
\ell_{-,\ga}\mathbb Z^d\cap \La, \forall s\in \{1,..,S\}\Big\}
\end{equation}
where $\dis{b:=\min_{k_1\neq k_2} \frac{\| \rho^{(k_1)} -
\rho^{(k_2)}\|_\infty}{2}}$ has been chosen such that
\begin{equation*}
W^{(k)}_{\La} \cap \{\rho^{(1)},..,\rho^{(S+1)}\} = \{ \rho^{(k)} \}
\end{equation*}

We then introduce a cutoff parameter $\eps\in (0,1)$ (which will
eventually vanish), call $(a)_+=a \,\text{\bf 1}_{a>0}$,
$(a)_-=a\,\text{\bf 1}_{a<0}$ and define for any $\eps>0$, the
function $f_\eps$ on $W^{(k)}_{\La}$ as
    \begin{eqnarray}
     \label{ee5.1.2}
&&f_\eps(\rho_\La;\bar q_{\La^c}):=f(\rho_\La;\bar
q_{\La^c})+ \frac {\eps^{-1}}4 \sum_{i\in
 \La}\Big(
\{(\rho_\La(i)-[\rho^{(k)}_{s(i)}+\zeta])_+\}^4\nn\\&&\hskip3cm
+\{(\rho_\La(i)- [\rho^{(k)}_{s(i)}-\zeta])_-\}^4 \Big)
     \end{eqnarray}

$\bullet$\; Since $f$ [$f_\eps$] is a continuous function of
$\rho_\La$ which varies on a compact set, it has a minimizer denoted
by $\hat \rho_\La$  [$\hat \rho_{\La,\eps}$], and we will later see that
this minimizer is unique. We call $\hat \rho$ its extension  to the
whole $\ell_{-,\ga}\mathbb Z^d \times \{1,..,S\}$, by setting $\hat
\rho= \bar \rho_{\La^c}$ on $\La^c$. Here $\bar \rho_{\La^c}$ is the
density associated to $\bar q_{\La^c}$ via \eqref{e3.2.2} with
$\ell=\ell_{-,\ga}$, thus $\hat \rho$ of course depends on $\bar
q_{\La^c}$.

$\bullet$\; For any $\mathcal D^{(\ell_{-,\ga})}$-measurable set $B$
we write for any differentiable and $\mathcal
D^{(\ell_{-,\ga})}$-measurable function $\psi(\rho)$
     \begin{equation}
      \label{ee5.1.3}
D_{B}\psi\;=\;\Big\{\;\frac{\partial\psi}{\partial \rho(i)},\;
   x(i)\in \ell_{-,\ga}\mathbb Z^d \cap B\Big\}
     \end{equation}

\vskip1cm


\subsection{A-priori estimates}

       \label{subsec:e5.2}


In this subsection we prove some a-priori bounds on $\hat
\rho_{\La,\eps}(i)$.  When $\eps>0$ we loose the bound
$|\hat\rho_\La(i) - \rho^{(k)}_{s(i)}|\le \zeta$ valid at
$\eps=0$  but, as we will see, we have the great
simplification to know that for $\eps$ small enough,
minimizers are critical points, thus satisfying $D_\La
f_\eps=0$, and $|\hat\rho_{\La,\eps}(i) -
\rho^{(k)}_{s(i)}|\le 2\zeta$.

\vskip1cm

    \begin{lemma}
     \label{lemmaee.5.2.1}
There is a constant $c>0$ such that for all $\eps>0$ and for  any
minimizer $\hat \rho_{\La,\eps}\in W^{(k)}_{\La}$ of $f_\eps$ the following holds: for all $x\in \ell_{-,\ga}\mathbb Z^d\cap \La$ and all $s\in \{1,..,S\}$,
    \begin{eqnarray}
     \label{ee5.2.1}
&& \big|\hat\rho_{\La,\eps}(x,s)-\rho^{(k)}_{s}\big| \le \zeta +c
(\frac{\ell_{+,\ga}}{\ell_{-,\ga}})^{d/4}\eps^{1/4}
     \end{eqnarray}
In particular, if $\zeta< b/2$ then for all $\eps>0$ small enough, any
minimizer $\hat \rho_{\La,\eps}\in W^{(k)}_{\La}$ of $f_\eps$ is also a critical point.
   \end{lemma}

\vskip.5cm

{\bf Proof.}  We denote by
    $$
\psi(\rho_\La)=\sum_{i\in \La}
\{(\hat\rho_{\La,\eps}(i)-[\rho^{(k)}_{s(i)}+\zeta])_+\}^4
+\{(\hat\rho_{\La,\eps}(i)- [\rho^{(k)}_{s(i)}-\zeta])_-\}^4
    $$
Then for all $\rho_{\La}\in W^{(k)}_\La$,
   $$
\frac 1{4\eps}\psi(\hat \rho_{\La,\eps})\le f(\rho_\La;\bar
q_{\La^c})-f(\hat \rho_{\La,\eps};\bar q_{\La^c})+\frac
1{4\eps}\psi(\rho_\La)
    $$
and since $\Psi$ vanishes on $Y^{(k)}_\La$:
   $$
\frac 1{4\eps}\psi(\hat \rho_{\La,\eps})\le \inf_{\rho_{\La}\in
Y^{(k)}_\La}  f(\rho_\La;\bar q_{\La^c}) - f(\hat
\rho_{\La,\eps};\bar q_{\La^c})
    $$
and, calling $\dis{\phi'=\min_{\rho_{\La}\in
Y^{(k)}_\La}f(\rho_\La;\bar q_{\La^c})}$,
$\dis{\phi''=\min_{\rho_{\La}\in W^{(k)}_\La}f(\rho_\La;\bar
q_{\La^c})}$
    $$
\frac 1{4\eps}\psi(\hat \rho_{\La,\eps})\le \phi'-\phi''
    $$
and in conclusion
    \begin{equation}
     \label{ee5.2.1.0}
|\hat\rho_{\La,\eps}(x,s)-\rho^{(k)}_{s}|\le \Big(4\eps
(\phi'-\phi'')\Big)^{1/4}\,+\,\zeta
    \end{equation}
and \eqref{ee5.2.1} follows because $\phi'$ and $\phi''$ are bounded
proportionally to the cardinality of  $\{x: x\in \ell_{-,\ga}\mathbb
Z^d\cap\La\}$.

By choosing $\eps$ so small that $\dis{\zeta +c\ga^{-(\alpha_++\alpha_-)d/4}\eps^{1/4}< 2\zeta< b}$, we conclude that $\hat\rho_{\La,\eps}$ is in the interior of $W^{(k)}_\La$ and is thus a critical point.
\qed

\vskip1cm

    \begin{lemma}
     \label{lemmaee.5.2.2}

$\hat \rho_{\La,\eps}$ converges by subsequences and any
limit point  $\hat \rho_{\La}$ is a minimizer of $f$.

   \end{lemma}

\vskip.5cm

{\bf Proof.} Convergence by subsequences follows from
compactness and by \eqref{ee5.2.1} any limit point  $\hat
\rho_{\La}$ is in $Y^{(k)}_\La$. Now for any $\rho_{\La} \in Y^{(k)}$, we get
$f(\rho_\La) = f_\eps(\rho_\La) \ge f_\eps(\hat
\rho_{\La,\eps})\ge f(\hat \rho_{\La,\eps})$ and by taking
$\eps\to 0$ along a convergent subsequence $ f(\rho_\La)
\ge f(\hat \rho_{\La})$.
\qed

\vskip1cm

A minimizer $\hat \rho_\La$ of $f$ is not necessarily a
critical point, i.e.\ $D_\La f=0$, the equality may fail if
the minimizer is on the boundary of the constraint. In such
a case however, the gradient if different from zero ``must
be directed along the normal pointing toward the
interior''.

\vskip.5cm

\begin{lemma}
 \label{lemmaee.5.2.3}
Any minimizer $\hat\rho_\La$  of $\{f(\rho_\La,\bar
q_{\La^c})$, $ \rho_{\La}\in Y_{\La}^{(k)}\}$ is ``a
critical point'' in the following sense:

$\bullet$\; If for some $i\in \La$,
$|\hat\rho_\La(i)-\rho^{(k)}_{s(i)}| <\zeta$ (strictly!),
then
   \begin{equation}
      \label{ee5.2.2}
\frac{\partial}{\partial \rho_\La(i)} f (\hat \rho_\La,\bar
q_{\La^c}) =0
     \end{equation}

$\bullet$\; If instead
$\hat\rho_\La(i)=\rho^{(k)}_{s(i)}\pm \zeta$, then
   \begin{equation}
      \label{ee5.2.3}
\frac{\partial}{\partial \rho_\La(i)}  f(\hat \rho_\La,\bar
q_{\La^c}) \le 0,\;\;\text{respectively $\ge 0$}
     \end{equation}

  \end{lemma}

\vskip2cm


\subsection{Convexity and uniqueness}

       \label{subsec:e5.3}


Convexity is a key ingredient in our analysis:

 \vskip1cm

\begin{thm}
 \label{thmIe5.3.1}

Given any  $\kappa \in (0,\kappa^*)$ ($\kappa^*$ as in \eqref{2.4}),
for all $\ga$ small enough the following holds. Let $\rho_\La \in
W^{(k)}_\La$ be such that $|\rho_{\La}(i)-\rho^{(k)}_{s(i)}\big| \le
4\zeta$, then the matrix $A:=D^2_\La f_\eps(\rho_\La,\bar
q_{\La^c})$  is strictly positive, as an operator on $\mathcal H$,
namely (recall the definitions in Subsection \ref{subsec:e4.1})
   \begin{equation}
      \label{e5.3.1}
\big(u,A u\big) \ge \kappa (u,u),\;\;\text{for
 all $u \in \mathcal H$}
     \end{equation}
Same inequality holds when $\eps=0$.


  \end{thm}

\vskip.5cm

{\bf Proof.} Recalling \eqref{ee5.1.1} and denoting by
$\rho_\La^{-1}$ below the diagonal matrix with entries
$\rho_\La(i)^{-1}$
    \begin{equation*}
(u,A u) =    t (u, \bar V_\ga u)   +\frac{1}{\beta}
(u,\rho_\La^{-1}   u) + (u, [ D^2_\La g]u) +  (u, [ D^2_\La
(f_\eps-f)]u)
     \end{equation*}
and get a lower bound by dropping the last term thus
reducing the proof to the case $\eps=0$. Extend $u$ and $A$
as equal to 0 outside $\La$ and set
   $$
   U(x,s)=  \ell_{-,\ga}^d \sum_{y\in \ell_{-,\ga} \mathbb Z^d}
   J^{(\ell_{-,\ga})}_\ga(x,y)u(y,s),
   \;\;x\in \ell_{-,\ga}\mathbb Z^d
   $$
where $J^{(\ell)}_\ga$ is defined in  \eqref{e4.1.1}. Then,
   \begin{eqnarray*}
&& (u,A u) \ge    t  \sum_{s\ne s'} \sum_{x\in \ell_{-,\ga}\mathbb
Z^d}U(x,s)U(x,s') +\frac{1}{\beta} (u,\rho_\La^{-1}  u) + (u, [
D^2_\La g]u)\nn\\&& \hskip1.5cm =  \Big\{ t \sum_{s\ne s'}
\sum_{x\in \ell_{-,\ga}\mathbb Z^d}U(x,s)U(x,s') + \sum_{x\in
\ell_{-,\ga}\mathbb Z^d,s} [\frac{1}{\beta\rho^{(k)}(s)} -\kappa^*]
U(x,s)^2\Big\} \nn\\&& \hskip2cm -  \sum_{x\in \ell_{-,\ga}\mathbb
Z^d,s}[\frac{1}{\beta\rho^{(k)}(s)} -\kappa^*] U(x,s)^2
+\frac{1}{\beta} (u,\rho_\La^{-1}  u) +(u, [ D^2_\La g]u)
     \end{eqnarray*}
recalling\eqref{e2.8}, by \eqref{2.4} the curly bracket is non
negative as well as $\dis{\frac 1{\beta\rho^{(k)}_s}-\kappa^*}$.

Since for each $s$
      $$
\sum_{x\in \ell_{-,\ga}\mathbb Z^d}  U(x,s)^2 \le
\sum_{x\in \ell_{-,\ga}\mathbb Z^d} u(x,s)^2
       $$
then
      $$
\sum_{x\in\ell_{-,\ga}\mathbb
Z^d,s}[\frac{1}{\beta\rho^{(k)}_s} -\kappa^*] U(x,s)^2 \le
( u,[\frac{1}{\beta\rho^{(k)}}-\kappa^*]u)
       $$
Thus
   \begin{eqnarray*}
&& \big(u,A  u\big) \ge   \Big(u, [\kappa^*+\frac 1 {\beta \rho_\La}
-\frac 1 {\beta \rho^{(k)}}] u\Big)+ \big(u, [D^2_\La g] u\big)
     \end{eqnarray*}
Recalling \eqref{7e.0}, \eqref{7e.0.0} and using
\eqref{ee4.2.7}--\eqref{ee4.2.9} we get
    \begin{eqnarray*}
&& \| D^2_\La g\|\le \sup_i \sum _j |\frac {\partial ^2g}{\partial
\rho_\La(i)\partial\rho_\La(j)}|\le(\ga\ell_{-,\ga})^{a_0}
     \end{eqnarray*}
Thus
    $$
\big(u, [D^2_\La g] u\big)\le [\ga\ell_{-,\ga}]^{2a_0}(u,u)
    $$
 \eqref{e5.3.1} is then proved recalling the assumption
$|\rho_{\La}(i)-\rho^{(k)}_{s(i)}\big| \le 4\zeta$.

 \qed

\vskip1cm

\begin{thm}
 \label{thmIe5.3.2}
Given any  $\kappa \in (0,\kappa^*)$ ($\kappa^*$ as in
\eqref{2.4}), for all $\ga$ small enough the following
holds.  Let $\hat\rho_{\La,\eps}$ be a minimizer of
$f_\eps$ and for $\eps=0$ of $f$, then for both $\eps>0$
small enough and  $\eps=0$
   \begin{equation}
      \label{e5.3.2}
f_\eps(\rho_\La, \bar q_{\La^c})\ge
 f_\eps(\hat\rho_{\La,\eps},
\bar q_{\La^c})
 + \frac{\kappa}{2}\; \big(\rho_\La-\hat\rho_{\La,\eps},\rho_\La-\hat\rho_{\La,\eps}\big)
     \end{equation}
for all $\rho_\La$ such that
$|\rho_\La(i)-\rho^{(k)}_{s(i)} |\le 2\zeta$ for all $i\in
\La$.  \eqref{e5.3.2} remains valid  if
$\hat\rho_{\La,\eps}$ is a critical point, $D_\La
f_\eps=0$, and $|\hat\rho_{\La,\eps} -\rho^{(k)} \big| \le
2\zeta$ as well as when  $\eps=0$ and $\hat\rho_{\La,0}$ a
``critical point'' of $f$ in the sense of Lemma
\ref{lemmaee.5.2.3}.

\end{thm}

\vskip.5cm

{\bf Proof.}  We interpolate by setting $\rho_\La(\theta)=
\theta\rho_\La+ (1-\theta) \hat \rho_{\La,\eps}$, $\theta\in [0,1]$,
then  calling  $\psi_\eps(\theta):=f_\eps (\rho_\La(\theta),\bar
q_{\La^c})$ we have
   \begin{eqnarray*}
 \psi_\eps(1) -  \psi_\eps(0)
&=& \int_0^1 \big(D_\La \psi_\eps(\theta),\rho_\La- \hat\rho_{\La,\eps}\big)\\
&=& \int_0^1 \int_0^\theta \big(D^2_\La
\psi_\eps(\theta')\{\rho_\La- \hat\rho_{\La,\eps}\},\rho_\La-
\hat\rho_{\La,\eps}\big) + \big(D_\La \psi_\eps(0),\rho_\La-
\hat\rho_{\La,\eps}\big)
     \end{eqnarray*}

By \eqref{ee5.2.1} for $\eps>0$ small enough and for
$\eps=0$ as well, $|\rho_{\La}(\theta)-\rho^{(k)}\big| \le
4\zeta$ so that by \eqref{e5.3.1}
   \begin{eqnarray*}
 \int_0^1 \int_0^\theta \big(D^2_\La
\psi_\eps(\theta')\{\rho_\La- \hat\rho_{\La,\eps}\},\rho_\La-
\hat\rho_{\La,\eps}\big)  \ge \frac{\kappa}{2}\;
\big(\rho_\La-\hat\rho_{\La,\eps},\rho_\La-\hat\rho_{\La,\eps}\big)
     \end{eqnarray*}
Moreover  $\big(D_\La \psi_\eps(0),\rho_\La-
\hat\rho_{\La,\eps}\big)\ge 0$. In fact, if $\eps>0$ and
$\hat\rho_{\La,\eps}$ is a minimizer of $f_\eps$, by Lemma
\ref{lemmaee.5.2.1} (for $\eps>0$ small enough)
$\hat\rho_{\La,\eps}$ is also a critical point and $D_\La
\psi_\eps(0)=0$. If $\eps=0$ and $\hat\rho_{\La}$ a minimizer of $f$
then by Lemma \ref{lemmaee.5.2.3}, $\big(D_\La \psi_0(0),\rho_\La-
\hat\rho_{\La}\big)\ge 0$ which, for the same reason,  holds if
$\hat\rho_{\La}$ is a critical point of $f$ in the sense of Lemma
\ref{lemmaee.5.2.3}.

\qed

\vskip.5cm

 \begin{coro}
 \label{coroe5.3.1}
For any $\ga$ and $\eps>0$ small enough the minimizer of
$f_\eps$ is unique, same holds at $\eps=0$ for $f$. For
$\eps>0$ (and small enough) there is a unique critical
point in the space $\{|\rho_{\La} -\rho^{(k)} \big| \le
2\zeta\}$; such a critical point minimizes $f_\eps$. Analogously, when $\eps=0$ there is a unique critical point
in the sense of Lemma \ref{lemmaee.5.2.3}. Such a critical
point minimizes  $f$. The minimizer of $f_\eps$, $\eps>0$,
converges as $\eps\to 0$ to the minimizer of $f$.

 \end{coro}

\vskip.5cm

{\bf Proof.} From Lemma \ref{lemmaee.5.2.1} it follows that any
minimizer $\hat \rho_{\La,\eps}$ of $f_\eps$ is also a critical
point and verifies \eqref{ee5.2.1}, so for all $x\in
\ell_{-,\ga}\mathbb Z^d\cap \La$ and all $s\in \{1,..,S\}$,
$\big|\hat\rho_{\La,\eps}(x,s)-\rho^{(k)}_{s}\big| \le 2\zeta $ and
we can apply
 Theorem   \ref{thmIe5.3.2} to the matrix $D^2_\La f_\eps(\hat \rho_{\La,\eps};\bar
q_{\La^c})$. If we assume that there are two minimizers, then
\eqref{e5.3.2} gives a contradiction. The proofs in the case
$\eps=0$ follows by using
 Lemma \ref{lemmaee.5.2.2}.\qed

\vskip2cm


\subsection{Perfect boundary conditions}

       \label{subsec:e5.4}


In this subsection we restrict to ``perfect boundary conditions'',
by this meaning that we study
    \begin{equation}
     \label{e5.4.1}
f^{\rm pf}(\rho_\La;\bar q_{\La^c}) =
F_\La(\rho_\La|\rho^{(k)}1_{\La^c}) + g( \rho_\La;\bar q_{\La^c})
     \end{equation}
namely we replace in the LP term of the effective Hamiltonian, see
\eqref{ee4.2.10}, $\bar\rho_{\La^c}$ by the mean field equilibrium
value.  $f^{\rm pf}_\eps$ is then defined by adding to $f^{\rm pf}$
the term $f_\eps-f$ given by \eqref{ee5.1.2}.  All the previous
considerations obviously apply to $f^{\rm pf}$ and $f^{\rm
pf}_\eps$.

\vskip1cm

\begin{thm}
 \label{thme5.4.1}
For any $\ga$ small enough and for all $\eps>0$ small
enough, the minimizer ${\hat\rho}^{\rm pf}_{\La,\eps}$ of
$f^{\rm pf}_\eps$ minimizes $f^{\rm pf}$ as well and it is
such that
   \begin{equation}
      \label{e5.4.2}
| {\hat \rho}^{\rm pf}_{\La,\eps}(i)-\rho^{(k)}_{s(i)}| \le
c (\ga\ell_{-,\ga})^{a_0},\quad \text{ for all $i\in \La$}
     \end{equation}
$c>0$ a constant.

\end{thm}

\vskip.5cm

{\bf Proof.} Since ${\hat\rho}^{\rm pf}_{\La,\eps}$ is a minimizer
of $f^{\rm pf}_\eps$, $D_\La  f^{\rm pf}_\eps({\hat\rho}^{\rm
pf}_{\La,\eps})=0$.  Then if \eqref{e5.4.2} holds, $D_\La f^{\rm
pf}({\hat\rho}^{\rm pf}_{\La,\eps})=D_\La f^{\rm
pf}_\eps({\hat\rho}^{\rm pf}_{\La,\eps})=0$ and by Corollary
\ref{coroe5.3.1} ${\hat\rho}^{\rm pf}_{\La,\eps}$ is a minimizer
of $f^{\rm pf}$.     We thus have only to prove \eqref{e5.4.2} for
all $\eps>0$ small enough.  Consider first the simplified problem
with $g=0$.

  \centerline{\it Case $g=0$}
 \nopagebreak
Recalling \eqref{ee4.2.2},  if $D_\La
F_\La(\rho_\La|\rho^{(k)}1_{\La^c})=0$, then by an explicit
computation, for all $i \in \La$,
   \begin{equation}
      \label{e5.4.3}
  \rho_\La(i) = \exp\Big\{-\beta[ \sum_{j\in \ell_{-,\ga}
  \mathbb Z^d} t\bar V_ \ga (i,j) \rho(j)
  +(1-t)\rho^{(k)}_{s(i)}
  -\la_\beta]\Big\}
     \end{equation}
where $\rho(j) =\rho_\La(j)$ if  $j\in \La$ and
$=\rho^{(k)}_{s(j)}$ if  $j\in \La^c$.
$\rho_\La(i)=\rho^{(k)}_{s(i)}$ is a solution of
\eqref{e5.4.3} and therefore also a solution of $D_\La
f^{\rm pf}_\eps=0$ (with $g=0$). By Corollary
\ref{coroe5.3.1} it is then the unique minimizer of $f^{\rm
pf}_\eps$ and \eqref{e5.4.2} is proved (for $g=0$).

\vskip.5cm

  \centerline{\it Proof of \eqref{e5.4.2}.}
 \nopagebreak

Call
        $$
f_{\eps,\theta}(\rho_\La)=F_\La (\rho_\La|\rho^{(k)}1_{\La^c})
+\theta g(\rho_\La;\bar q_{\La^c},t)+(f_\eps-f)
        $$
$\theta\in [0,1]$; for all $\eps>0$
small enough denote by $\hat \rho_{\La,\eps,\theta}$ the minimizer
of $f_{\eps,\theta}$, so that $D_\La f_{\eps,\theta}(\hat
\rho_{\La,\eps,\theta})=0$. Suppose that
  \begin{equation}
      \label{e5.4.4}
 \frac{d\hat
\rho_{\La,\eps,\theta}}{d\theta}\;\; \text{exists for all
$\theta\in [0,1]$ and depends continuously on $\theta$}
     \end{equation}
Obviously $\hat \rho_{\La,\eps,1}=\hat \rho_{\La,\eps}$ while $\hat
\rho_{\La,\eps,0}= \rho^{(k)}1_\La$ because of the above analysis
with $g=0$.  Then
   \begin{equation}
      \label{I.3.2.19.0}
\hat \rho_{\La,\eps,\theta}= \rho^{(k)}1_\La+\int_0^1 \frac
{d\hat \rho_{\La,\eps,\theta}}{d\theta}
     \end{equation}
On the other hand by differentiating $D_\La f _{\eps,\theta}(\hat
\rho_{\La,\eps,\theta})=0$ we get
   \begin{equation}
      \label{I.3.2.21.0}
 D^2_\La f
_{\eps,\theta}(\hat \rho_{\La,\eps,\theta})  \frac
{d\hat\rho_{\La,\eps,\theta}}{d\theta}
 = -
    D_{\La}g(\hat\rho_{\La,\eps,\theta})
      \end{equation}
By Lemma \ref{lemmaee.5.2.1} and Theorem \ref{thmIe5.3.1} for all
$\eps>0$ small enough, $ D^2_\La f_{\eps,\theta}(\hat
\rho_{\La,\eps,\theta}) $ is symmetric and positive definite, then
by Theorem \ref{thmappB.1.1} the inverse $(D^2_\La
f_{\eps,\theta}(\hat \rho_{\La,\theta}))^{-1}$ is well defined and
bounded as an operator on $L^\infty$, and we thus get  from
\eqref{I.3.2.21.0}
   \begin{equation}
      \label{I.3.2.21.1}
 | \frac
 {d\hat\rho_{\La,\eps,\theta}}{d\theta}| \le c \|
    D_{\La}g(\hat\rho_{\La,\eps,\theta})\|_\infty \le c'
    (\ga\ell_{-,\ga})^{a_0}
      \end{equation}
which by \eqref{I.3.2.19.0} yields \eqref{e5.4.2}.
\eqref{I.3.2.21.1} also implies that
$|\hat\rho_{\La,\eps,\theta} -\rho^{(k)}1_\La|\le c'
(\ga\ell_{-,\ga})^{a_0}$.  Notice that \eqref{I.3.2.21.1}
implies  \eqref{e5.4.4}, but unfortunately the argument is
circular as it started by supposing the validity of
\eqref{e5.4.4}. To avoid the impasse we start from the
equation in the unknown $u_\La$
   \begin{equation}
      \label{I.3.2.21.11}
D^2_\La f _{\eps,\theta}( \rho_{\La}) u_{\La}= -
    D_{\La}g(\rho_{\La})
      \end{equation}
where $\rho_\La$ is considered as a ``known term'' such
that $|\rho_\La(i)- \rho^{(k)}_{s(i)}|\le 2\zeta$ for all
$i\in \La$.  From what said before, \eqref{I.3.2.21.11} has
a unique solution called $\dot \rho_\La(i|\rho_\La)$ and
   \begin{equation}
      \label{I.3.2.22}
|\dot \rho_\La(i|\rho_\La)| \le c
(\ga\ell_{-,\ga})^{a_0},\quad \text{ for all $i\in  \La$}
     \end{equation}
Since $\dot \rho_\La(\cdot|\rho_\La)$ is Lipschitz in
$\rho_\La$ (we omit the details)  the ordinary differential
equation
   \begin{equation}
      \label{I.3.2.23}
\frac{ d \rho_\La(\theta)}{d\theta}= \dot
\rho_\La(\cdot|\rho_\La(\theta)),\quad
\rho_\La(0)=\rho^{(k)} 1_\La
     \end{equation}
has  a unique solution $\tilde \rho_\La(\theta)$.  Then, by
\eqref{I.3.2.21.11},
   \begin{equation}
      \label{I.3.2.24}
\frac {d}{d\theta} D_\La f _{\eps,\theta}(\tilde \rho_\La(\theta))
=0,\;\;\text{and hence}\;\; D_\La f_{\eps,\theta} (\tilde
\rho_\La(\cdot;\theta)) = D_\La f_{\eps,0}(\rho^{(k)} 1_\La)=0
      \end{equation}
Since $|\tilde\rho_{\La}(\theta) -\rho^{(k)}1_\La|\le c'
    (\ga\ell_{-,\ga})^{a_0}$, $D_\La f_{0,\theta} (\tilde
\rho_\La(\cdot;\theta))=0$ as well, hence by Corollary
\ref{coroe5.3.1}, $\tilde \rho_\La(\cdot;\theta)=\hat
\rho_{\La,\eps,\theta}(\cdot)$ and by \eqref{I.3.2.23} it is
differentiable with continuous derivative. \eqref{e5.4.4} thus holds
and the theorem proved.  \qed

\vskip2cm


\subsection{Exponential decay}

       \label{subsec:e5.5}


This subsection concludes our analysis with the following main
theorem, Theorem \ref{thmee5.0} will be proved in the Subsection
\ref{subsec:e5.6} as a corollary, taking $\La_{1}^{c}$ as a neighborhood of $x$ in $\La^c$ and $\La_{2}^c = \La^{c}\setminus \La_{1}^c$.

\vskip1cm

\begin{thm}
 \label{thmIe.2.5}
There are $\hat \om$ and $c$ positive such that the following holds.
Let ${\hat \rho}'_\La$ and  ${\hat \rho}''_\La$ be the minimizers of
$f(\rho_\La,{\bar q}'_{\La^c})$, respectively $f(\rho_\La,{\bar
q}''_{\La^c})$, with ${\bar q}'_{\La^c},{\bar q}''_{\La^c}\in
\mathcal X_{\La^c}$. Then for any partition of $\La^c$ into two
$\mathcal D^{(\ell_{-,\ga})}$-measurable sets $\La^c_1$ and $
\La^c_2$,
   \begin{eqnarray}
      \label{Ie.3.2.18}
&&\hskip-.5cm|{\hat \rho}''_\La(i)-{\hat \rho}'_\La(i)| \le
c \Big (\min\big\{
\text{\bf 1}_{{\bar q}''_{\La^c_1}\ne {\bar q}'_{\La^c_1}};
 \max_{j \in \La_1^c}\big((\ga\ell_{-,\ga})^{a_0}
+ \;|\rho^{(\ell_{-,\ga})}({\bar
q}''_{\La^c};j)-\rho^{(\ell_{-,\ga})}({\bar
q}'_{\La^c};j)|\big)\big\}\nn\\&&\hskip2cm + \sum_{j \in \La_2^c}
e^{-\om \ga |x(i)-x(j)|} \;\text{\bf 1}_{{\bar
q}''_{C^{(\ell_{-,\ga})}_j } \ne {\bar q}'_{C^{(\ell_{-,\ga})}_j}}
\Big),\qquad \,\,\,\,\,\,\forall i\in \La
     \end{eqnarray}

\end{thm}

{\bf Proof.}   We follow the interpolation strategy used in the
proof of Theorem \ref{thme5.4.1}.  To this end we separate the
``interaction part''  in $f_\eps$ writing $f_\eps=
f^0_\eps+f_\eps^1$ where $f^0_\eps=f^0_\eps(\rho_\La)$ is
independent of the boundary conditions while
    \begin{equation}
     \label{Ie.3.2.21.0}
f^1_\eps (\rho_\La,\bar q_{\La^c})=  
t (\rho_\La, \bar V_\ga\bar \rho_{\La^c}) + g_1 (\rho_\La,\bar
q_{\La^c})
     \end{equation}
where $g_1$ is given by the r.h.s of \eqref{ee4.2.7} with the sum
over $\und i$ restricted to the set $\und i \cap \La^c\ne
\emptyset$.

 We then interpolate between the
two boundary conditions
   \begin{equation}
      \label{Ie.3.2.21}
f_{\theta,\eps}(\rho_\La):= f^0_\eps(\rho_\La) + \theta f_\eps^1
(\rho_\La,{\bar q''}_{\La^c}) +(1-\theta) f_\eps^1 (\rho_\La,{\bar
q'}_{\La^c}),\quad \theta\in [0,1]
     \end{equation}
The analysis done in the previous subsections, applies to
$f_{\theta,\eps}(\rho_\La)$ as well. Thus the minimizer $\hat
\rho_{\La.\eps,\theta}$ of $f_{\theta,\eps}$ is unique, is a
critical point, namely  $D_\La f_{\theta,\eps}(\hat
\rho_{\La.\eps,\theta})=0$ and satisfies for all $x\in
\ell_{-,\ga}\mathbb Z^d\cap \La$ and all $s\in \{1,..,S\}$,
$\big|\hat\rho_{\La,\eps,\theta}(x,s)-\rho^{(k)}_{s}\big| \le
2\zeta $.

We can apply the same proof as the one given in Theorem
\ref{thme5.4.1}. In fact by Theorem \ref{thmIe5.3.1}, for all
$\eps>0$ small enough, and for all $\rho_\La$ such that
$\big|\rho_{\La}(x,s)-\rho^{(k)}_{s}\big| \le 2\zeta $, we have that
$ D^2_\La f_{\theta,\eps}( \rho_{\La})$ is symmetric and positive
definite, then by Theorem \ref{thmappB.1.1} the inverse $(D^2_\La
f_{\theta,\eps}(\rho_{\La}))^{-1}$ is well defined and bounded as an
operator on $L^\infty$. Thus the equation
     \begin{equation}
      \label{a5.25}
(D_\La^2f_{\theta,\eps} (\rho_{\La}))u_\La=- \frac {\partial D_\La
f_{\theta,\eps}(\rho_{\La})}{\partial\theta}
     \end{equation}
has a unique solution that we call $u_\La(\cdot,\rho_{\La})$ that
is Lipschitz in $\rho_{\La}$. This implies that the equation
    \begin{equation*}
\frac {d
\rho_{\La,\eps,\theta}}{d\theta}=u_\La(\cdot,\rho_{\La,\eps,\theta}),\qquad
 \rho_{\La,\eps,0} =\hat\rho_{\La,\eps,0}
     \end{equation*}
has a unique solution that coincides with the minimizer
$\hat\rho_{\La,\eps,\theta}$. Thus $\hat \rho_{\La.\eps,\theta}$ is
differentiable in $\theta$ and $d\hat
\rho_{\La.\eps,\theta}/d\theta$ satisfies
   \begin{equation}
      \label{Ie.3.2.23}
(D_\La^2f_{\theta,\eps}(\hat \rho_{\La,\eps,\theta}) )\frac {d\hat
\rho_{\La,\eps,\theta}}{d\theta}=- \frac {\partial D_\La
f_{\theta,\eps}(\hat \rho_{\La,\eps,\theta})}{\partial\theta}
     \end{equation}
By  Corollary \ref{coroe5.3.1},  $\hat
\rho_{\La.\eps,\theta}$ converges by subsequences as
$\eps\to 0$ to a limit $\tilde \rho_{\La,\theta}$ which
minimizes $f_\theta$, so that
   \begin{eqnarray}
      \label{e5.5.5}
&& |{\hat \rho}''_\La(i)-{\hat \rho}'_\La(i)| \le
\lim_{\eps\to 0} \int_0^1  |\frac {d\hat
\rho_{\La.\eps,\theta}(i)}{d\theta}|
     \end{eqnarray}
We now estimate  $\dis{|\frac {d\hat
\rho_{\La.\eps,\theta}(i)}{d\theta}|}$ uniformly in $\eps$ and
$\theta$ to prove \eqref{Ie.3.2.18} as a consequence of
\eqref{e5.5.5}.

\vskip1cm

  \centerline{\it Equations for ${d\hat
\rho_{\La.\eps,\theta}/d\theta}$.}
 \nopagebreak
we let
   \begin{equation}
      \label{Ie.3.2.28.2}
u:=\frac{d}{d\theta}\hat \rho_{\La.\eps,\theta},\qquad
 v=-  \frac {d}{d\theta'}
D_\La f_{\theta',\eps}(\hat \rho_{\La.\eps,\theta})
\Big|_{\theta'=\theta},\qquad A:=D^2_\La f_{\theta,\eps}(\hat
\rho_{\La.\eps,\theta})
     \end{equation}
so that  \eqref{Ie.3.2.23} becomes $$Au=v$$ We also define:
        \begin{equation}
      \label{a5.29}
A_0:=D^2_\La f_{\theta,0}(\hat
\rho_{\La.\eps,\theta}),\;\;\;\alpha:= A-A_0
     \end{equation}
$\alpha$ is a diagonal matrix whose diagonal elements are
   \begin{equation}
      \label{Ie.3.2.30.0}
\alpha(i): = 3 \eps^{-1} \Big( \{(\hat
\rho_{\La.\eps,\theta}(i)-[\rho^{(k)}_{s(i)}+\zeta])_+\}^2
+\{(\hat
\rho_{\La.\eps,\theta}(i)-[\rho^{(k)}_{s(i)}-\zeta])_-\}^2\Big)
     \end{equation}
To distinguish among large and non large (called small)
values of $\alpha(i)$, we introduce a large positive number
$b$ which will be specified later and, calling $\mathcal H$
the Hilbert space of vectors $u=\big(u(i), i\in \La\big)$,
   \begin{equation}
      \label{Ie.3.2.30}
G=\big\{(i): \alpha(i) \ge b\big\},\quad \mathcal H_G=
\big\{u\in \mathcal H: u(i)=0,\; \text{for all}\;i\in
G^c\big\}
     \end{equation}
Let $Q$ be the orthogonal projection on $ \mathcal H_G$ and
$P=1-Q$, thus $Q$ selects the sites where $\alpha$ is large
and $P$ those where it is small.

Our strategy will be the following: rewrite $Pu,Qu$ as linear expressions of $Pv,Qv$ to get bounds on $Pu,Qu$ (and therefore on $u$) using knowledge on $v$.

\vskip.5cm

  \centerline{\it Rewriting $Pu,Qu$ in terms of $Pv,Qv$.}
 \nopagebreak
Since the matrices $\alpha,P,Q$ are diagonal they commute, giving for instance $Q\alpha P = \alpha PQ = 0$, i.e.:
\begin{equation}
\label{QAA0P}
QAP = QA_0P,
\end{equation}
and symetrically:
\begin{equation}
\label{PAA0Q}
PAQ = PA_0Q.
\end{equation}

Using $Q^2=Q$ together with \eqref{QAA0P} we get:
   \begin{align}
      \nonumber QAQ Qu & = QAQu = QA(u - Pu)\\
      \nonumber QAQ Qu & = Qv - QAPu \\
        \label{Ie.3.2.31} Qu & = (QAQ)^{-1}\{Qv- QA_0Pu\}
     \end{align}
where $QAQ$ is invertible on the range of $Q$ since $A$ is a positive matrix.

Using $P^{2}=P$ together with \eqref{Ie.3.2.31} and \eqref{PAA0Q}we get:
   \begin{align*}
PAPu + PAQu & = PAu = Pv \\
PAPu + PA_0(QAQ)^{-1}\{Qv- QA_0Pu\} & = Pv \\
 \Big( PAP - PA_0(QAQ)^{-1}QA_0\Big)Pu & = Pv- PA_0(QAQ)^{-1}Qv
     \end{align*}
Let
   \begin{equation}
      \label{Ie.3.2.32.1}
B= PAP - PA_0(QAQ)^{-1}QA_0
     \end{equation}
so that if $B$ is invertible on the range of $P$ (as we
will prove), then
   \begin{equation}
      \label{Ie.3.2.33}
Pu=B^{-1}\{ Pv-PA_0(QAQ)^{-1}Qv\}
     \end{equation}

\vskip.5cm

  \centerline{\it A decomposition of $v$.}
 \nopagebreak
Recalling \eqref{Ie.3.2.28.2} and \eqref{Ie.3.2.21}, after
expanding the Poisson polynomials in  \eqref{ee4.2.7} we
get,
       \begin{eqnarray}
     \label{Ie.3.2.33.0}
&& v(i) =  -t \sum_{j\in \La^c}\bar V_\ga(i,j)
\Big({\bar\rho}''_{\La^c}(j)-{\bar\rho}'_{\La^c}(j)\Big)\nn\\&&\hskip1cm
-\sum_{n} (\ga\ell_{-,\ga})^{a_0n}\sum_{i_1,k_{i_1},..i_n,k_{i_n}:
i_1=i} k_{i_1}\Big( d_n\big(i_1,k_{i_1},..,i_n,k_{i_n};\bar
q''_{\La^c};t\big)\rho''(i_1)^{k_{i_1}-1}\cdots
\rho''(i_n)^{k_{i_n}}\nn\\&&\hskip2cm
-d_n\big(i_1,k_{i_1},..,i_n,k_{i_n};\bar q'_{\La^c};t\big)
\rho'(i_1)^{k_{i_1}-1}\cdots \rho'(i_n)^{k_{i_n}}\Big)
     \end{eqnarray}
where $\rho''(i)=\rho'(i)=\hat \rho_{\La,\eps,\theta}(i)$ if
$x(i)\in \La$ and $\rho''(i)={\bar\rho}''_{\La^c}(i)$,
$\rho'(i)={\bar\rho}'_{\La^c}(i)$ when $x(i)\in \La^c$. The
coefficients $d_n$ satisfy the same bounds as the coefficients
$\Phi$ of \eqref{ee4.2.7} (with maybe a different constant).

Shorthand by $\{x_j\}$ the sites in $\{x(i_1),..,x(i_n)\}$ which
are in $\La^c$, noticing that by definition of $g_1$ there are not
terms with $\{x_j\}= \emptyset$.

We then call $v^{(1)}$ the sum of $\dis{- t\sum_{j\in \La_1^c}
\bar V_\ga(i,j)
\Big({\bar\rho}''_{\La^c}(j)-{\bar\rho}'_{\La^c}(j)\Big)}$ minus
the second sum on the r.h.s.\ of \eqref{Ie.3.2.33.0} restricted to
sets $(i_1,...,i_n)$ such that: $\{x_j\} \ne \emptyset$ and any
$x_j\in \{x_j\}$ is either in $\La_1^c$ or ${\bar q
}''_{C^-_{x_j}}={\bar q}'_{C^-_{x_j}}$,
$C^-_x=C^{\ell_{-,\ga})}_x$, (or both). $v^{(2)}:= v- v^{(1)}$.

By linearity $u=u^{(1)}+u^{(2)}$ where $u^{(1)}$ and $
u^{(2)}$ are defined with $v$ replaced by $v^{(1)}$ and
$v^{(2)}$ and we will bound differently  $u^{(1)}$ and
$u^{(2)}$ using $\|\cdot\|_\infty$ norms for the former and
 $\|\cdot\|$ norms for the latter.

\vskip.5cm

  \centerline{\it Bounds on $u^{(1)}$.}
 \nopagebreak
By Theorem \ref{thmappB.2} if $b$ is large enough and
$c\ge\|A_0\|$,
   \begin{equation}
      \label{Ie.3.2.32.11}
\| PA_0(QAQ)^{-1}QA_0\|\le  \frac {2c^2}{b}=:\delta,\quad
\| PA_0(QAQ)^{-1}QA_0\|_\infty\le  \frac {2c^2}{b} e^{2c'}
     \end{equation}
Moreover by \eqref{appB.6}
   \begin{equation}
      \label{Ie.3.2.32.12}
 \sup_{i} \sum_{j}|B(i,j)| e^{\ga |i-j|}
\le \sup_{i\in G^c} \sum_{j}|A(i,j)| e^{\ga |i-j|} + \frac
{2c^2 e^{2c'}}b\;\le c''' b =:a
     \end{equation}
Then applying Theorem \ref{thmappB.1},\ref{thmappB.1.1} with $B$ as in
\eqref{Ie.3.2.32.1} and $R_1=PA_0(QAQ)^{-1}QA_0$, $B$ is invertible and there is
a constant $c>0$ such that $\|B^{-1}\|_\infty \le  c$. Therefore there is a new constant $c$
such that
   \begin{equation}
      \label{Ie.3.2.32.12.1}
 |Pu^{(1)}(i)| \le  c \max_j |v^{(1)}(j)|
     \end{equation}
If ${\bar q}''_{\La_1^c}={\bar q}'_{\La_1^c}$, $v^{(1)}=0$
and $u^{(1)}=0$ as well, let us then suppose ${\bar
q}''_{\La_1^c}\ne {\bar q}'_{\La_1^c}$.  Then
\eqref{Ie.3.2.33.0} yields
       \begin{eqnarray}
     \label{Ie.3.2.32.12.2}
&&|Pu^{(1)}(i)| \le c  \Big(\max_{j\in \La_1^c}
|{\bar\rho}''_{\La^c}(j)-{\bar\rho}'_{\La^c}(j)| +
(\ga\ell_{-,\ga})^{a_0}\Big)
     \end{eqnarray}
To bound $|Qu^{(1)}(i)|$ we go back to \eqref{Ie.3.2.31},
the same arguments used before prove that
$\|(QAQ)^{-1}\|_\infty \le c$ as well, so that
$|Qu^{(1)}(i)|$  is bounded as on the r.h.s.\ of
\eqref{Ie.3.2.32.12.2} (with a new constant $c$) and
$|u^{(1)}(i)|$ is therefore bounded as the first term on
the r.h.s.\  \eqref{Ie.3.2.18}, we will prove next that
$|u^{(2)}(i)|$ is bounded as the second term on the r.h.s.\
\eqref{Ie.3.2.18} which will then be proved.

\vskip.5cm

  \centerline{\it Bounds on $u^{(2)}$.}
 \nopagebreak
Recalling the definition of $v^{(2)}$
   \begin{equation}
      \label{Ie.3.2.32.14}
|v^{(2)}(i)|\le \sum_{j\in \La_2^c} K_\ga(i,j) \text{\bf
1}_{{\bar q}''_{C_j^-}\ne {\bar q}'_{C_j^-}}
     \end{equation}
where $\dis{ \sum_{i} K_\ga(i,j)\le c_K}$ and
$K_\ga(i,j)=0$ if $|x(i)-x(j)|\ge c'\ga^{-1}$, $c$ and $c'$
suitable constants.

By Theorem \ref{thmappB.1}
   \begin{equation}
      \label{Ie.3.2.32.13}
|B^{-1}(i,j)| \le  (\frac 1a+\frac 1
{\kappa'})\exp\Big\{-\frac{\kappa' \ga|i-j|}
{a+\kappa'}\Big\},\quad \kappa'=\kappa-\delta,
\;\;\text{$\delta$ as in \eqref{Ie.3.2.32.11}}
     \end{equation}
 By \eqref{Ie.3.2.32.13} and
\eqref{Ie.3.2.32.14}, calling $c''=1/a+  1/ {\kappa'}$ and
$\om= \kappa'/(a+\kappa')$,
   \begin{equation}
      \label{Ie.3.2.32.15}
|B^{-1} Pv^{(2)}(i)|\le  \sum_{j\in \La_2^c} \text{\bf
1}_{{\bar q}''_{C_j^-}\ne {\bar
q}'_{C_j^-}}\{c_Kc''e^{c'\om}\}e^{-\om \ga|x(i)-x(j)|}
     \end{equation}
By \eqref{appB.6}
   \begin{equation}
      \label{Ie.3.2.32.16}
 \sum_i |(QAQ)^{-1}(i,j)|
e^{\ga |i-j|} \le \frac {c_Q}b
    \end{equation}
and since $A_0(i,j)=0$ if $|i-j|\ge c'\ga^{-1}$ and
$\dis{\sum_i|A_0(i,j)|\le c_{A_0}}$,
   \begin{eqnarray*}
&&|B^{-1} PA_0(QAQ)^{-1}Qv^{(2)} (i)|\le  \sum_{j'}
\sum_{j''} \sum_{j'''\in \La_2^c}\text{\bf 1}_{{\bar
q}''_{C_{j'''}^-}\ne {\bar q}'_{C_{j'''}^-}} \{c''e^{-\om
\ga|x(i)-x(j')|} c_{A_0}e^{c'\om}\}\\&& \hskip1cm \times
e^{-\ga|x(j'')-x(j')|}|(QAQ)^{-1}(j',j'')|e^{\ga|x(j'')-x(j')|}
  K_\ga(j'',j''') \\&&
 \hskip1cm  \le  \{c'' c_{A_0}e^{c'\om}\}
\sum_{j'''\in \La_2^c}\text{\bf 1}_{{\bar
q}''_{C_{j'''}^-}\ne {\bar q}'_{C_{j'''}^-}}  e^{-\om
\ga|x(i)-x(j''')|} e^{\om c'} \; (\frac {c_Q}b)\; c_K
     \end{eqnarray*}
Thus supposing $\om \le 1$, we get from \eqref{Ie.3.2.33}
   \begin{equation}
      \label{Ie.3.2.32.17}
|Pu^{(2)}(i)|\le \sum_{j\in \La_2^c}\text{\bf 1}_{{\bar
q}''_{C_{j}^-}\ne {\bar q}'_{C_{j}^-}} c e^{-\om
\ga|x(i)-x(j)|}
     \end{equation}
To bound $Qu^{(2)}$ (recall \eqref{Ie.3.2.31}) we use
\eqref{Ie.3.2.32.16} to get
    \begin{equation}
      \label{Ie.3.2.32.18}
|(QAQ)^{-1} Qv^{(2)}(i)| \le  \sum_{j\in \La_2^c}\text{\bf
1}_{{\bar q}''_{C_{j}^-}\ne {\bar q}'_{C_{j}^-}} \frac
{c_Ke^{c'}c_Q}be^{-\ga|x(i)-x(j)|}
     \end{equation}
while, using \eqref{Ie.3.2.32.17} and \eqref{Ie.3.2.32.16},
   \begin{eqnarray*}
&&|(QAQ)^{-1} QA_0Pu^{(2)}(i)| \le \sum_{j'' }  \sum_{j'
}\sum_{j'''\in \La_2^c}\text{\bf 1}_{{\bar
q}''_{C_{j'''}^-}\ne {\bar q}'_{C_{j'''}^-}}   c e^{-\om
\ga|(x(j'')-x(j''')y|}\\&&\hskip1cm \times
 |A_0(j',j'')|
e^{-\ga|x(j'')-x(i)|} e^{c'} |(QAQ)^{-1}(i,j')|
e^{\ga|(x(j')-x(i)|}\\&&\hskip1cm\le ce^{c'}\sum_{j'''\in
\La_2^c}\text{\bf 1}_{{\bar q}''_{C_{j'''}^-}\ne {\bar
q}'_{C_{j'''}^-}}  e^{-\om \ga|x(i)-x(j''')|} c_{A_0}\;
(\frac {c_Q}b)
     \end{eqnarray*}
hence
    \begin{equation}
      \label{Ie.3.2.32.19}
| Qu^{(2)}(i)| \le \frac cb \sum_{j\in \La_2^c}\text{\bf
1}_{{\bar q}''_{C_{j}^-}\ne {\bar q}'_{C_{j}^-}}
\;e^{-\ga|x(i)-x(j)|}
     \end{equation}

     \qed

\vskip2cm


\subsection{Proof of Theorem \ref{thmee5.0}}

       \label{subsec:e5.6}

The proof is a corollary of Theorem \ref{thmIe.2.5}. Indeed given
any $x\in \ell_{-,\ga}\mathbb Z^d\cap \La$, call $\La_1^c$ the
union of all $C^{( \ell_{-,\ga})}_y$, $y\in \ell_{-,\ga}\mathbb
Z^d\cap
 \big(\La^c \cap B_x(10^{-30}\ell_{+,\ga})\big)$.  Then if $\hat K(x)>0$,
same notation as in  Theorem \ref{thmIe.2.5}, ${\bar q}''_{\La^c_1}
= {\bar q}'_{\La^c_1}$ and by \eqref{Ie.3.2.18} we are reduced to a sum over
$j\in \La_2^c$.  We split the exponent $-\ga\om |x(i)-x(j)|$ into two
equal terms and get
   \begin{eqnarray}
      \label{Ie.3.2.18.00}
&& |{\hat \rho}''_\La(x,s)-{\hat \rho}'_\La(x,s)| \le c  \{
e^{-(\om/2) \ga [10^{-30}\ell_{+,\ga}-\ell_{-,\ga}] }\}\{
\sum_{j\notin\La_1^c } e^{-(\om/2) \ga |x-x(j)|}\} \nn\\&&\hskip2cm
\le c' e^{-(\om/2) \ga[10^{-30}\ell_{+,\ga}-\ell_{-,\ga}] }
     \end{eqnarray}
The exponent $\hat \om$ in  Theorem \ref{thmee5.0} is thus going to
be half the $\om$ of Theorem \ref{thmIe.2.5}. Using  Theorem
\ref{thmIe.2.5} with ${\bar\rho}''_{\La^c}$ replaced by $\rho^{(k)}
1_{\La^c}$, and calling ${\hat \rho}^{\rm pf}_{\La}$  the
corresponding minimizer,
   \begin{eqnarray*}
&& |{\hat \rho}^{\rm pf}_{\La}(x,s)-{\hat \rho}'_\La(x,s)|  \le
c'e^{-(\om/2) \ga [10^{-30}\ell_{+,\ga}-\ell_{-,\ga}] } 
+ \big(c_1(\ga\ell_{-,\ga})^{a_0} +\zeta_m \big)
     \end{eqnarray*}
and using \eqref{e5.4.2}
   \begin{eqnarray*}
&& |{\hat \rho}'_\La(x,s)-\rho^{(k)}_{s}|  \le
c'e^{-(\om/2) \ga [10^{-30}\ell_{+,\ga}-\ell_{-,\ga}] } 
+ \big([c_1+c](\ga\ell_{-,\ga})^{a_0} +\zeta_m \big)
     \end{eqnarray*}

\newpage

\vskip1cm

\newpage

\section{\bf  Local Couplings}
\label{sec:a6}

 \vskip1cm

In this section we prove Theorem \ref{thme3.7.1}, thus we fix a
region $\La$, union of a finite number  $N_\La$, of cubes of
$\mathcal D^{(\ell_+)}$ and two boundary conditions $\bar
q_{i,\La^c}\in\mathcal X^{(k)}_{\La^c}$, $i=1,2$. We also fix a
$t\in(0,1]$ and we consider the two  Gibbs measures
$dG_\La^0(q_\La|\bar q_{i,\La^c})$ $i=1,2$ defined in \eqref{e3.7.1}
and with state space $\mathcal X^{(k)}_\La$. The aim is to construct
a coupling $Q_\La$ of these two probabilities such that
\eqref{e3.7.1a} holds. $Q_\La$, being a joint distribution, is
defined on the product space $\mathcal X^{(k)}_\La\times\mathcal
X^{(k)}_\La$ whose elements are denoted by $(q'_\La,q''_\La)$.

\bigskip

\subsection{Definitions and main results}
        \label{sec:a6.1}

Recalling that $K_\La(\cdot;x):=K_\La(\bar q_{1,\La^c},\bar
q_{2,\La^c};x)$ is defined in Definition \ref{kappa} we denote by
    \begin{equation}
    \label{a6.1}
\Delta_0\equiv\Delta_0(\bar q_{1,\La^c},\bar q_{2,\La^c}):=
\big\{x\in \ell_{-\ga}\mathbb Z^d\cap \La :K_\La(\bar
q_{1,\La^c},\bar q_{2,\La^c};x)>0\big\}
    \end{equation}
    In order to prove Theorem \ref{thme3.7.1} we have to find a coupling
$Q_\La$ so that there is $\eps_g$ such that
    \begin{equation}
    \label{aa6.1}
\sum_{x\in\Delta_0}Q_\La(\Theta_\La(x)^c)\le\eps_g
    \end{equation}
\medskip

We  define (recall that $B_x(R)$ is the ball of center $x$ and
radius $R$),
    \begin{equation}
    \label{aaa6.1}
\Delta_1=\bigcup_{x\in\Delta_0} B_x(10^{-20}\ell_{+,\ga})\cap\La
    \end{equation}
and we observe that $\Delta_1\supset \Delta_0$,
dist$\dis{({\Delta_0,\Delta_1^c})> 10^{-20}\ell_+,\ga}$.

\medskip

We denote by
    \begin{equation}
    \label{enn}
\und n\equiv n_{\La}=\Big\{n(x,s)\in \mathbb N, x\in
\ell_{-,\ga}\mathbb Z^d\cap\La, s\in \{1,..,S\}\Big\}
    \end{equation}
and in the sequel we will consider only those $\und n$ such that for
all $x\in \ell_{-,\ga}\mathbb Z^d\cap\La$ and $s\in \{1,..,S\}$,
$$\Big|\frac{n(x,s)}{\ell_-^d}-\rho^{(k)}(s)\Big|\le \zeta$$
Given $\und n$ and any subset $\Delta\subset \La$ we will call
$n_\Delta$ the restriction to $\Delta$ of $\und n$.

 \vskip1cm

%
%

\medskip

Given a subset $\Delta\subset \La$, we call $d_\Delta$ the following
metric on
 $\mathcal X^{(k)}_\La\times\mathcal X^{(k)}_\La$:
    \begin{eqnarray}
    \label{a6.2}
&&d_\Delta(q'_\La,q''_\La)=\sum_{x\in\ell_{-,\ga}\mathbb{Z}^d \cap
\Delta}d_x(q'_\La,q''_\La)
\\&&d_x(q'_\La,q''_\La)=\begin{cases}
    0 &{\text {if }} q'_\La\cap C^{(\ell_{-,\ga})}_x=
    q''_\La\cap C^{(\ell_{-,\ga})}_x
    \\ 1  &{\text {otherwise }}
\end{cases}
    \label{aa6.2a}
    \end{eqnarray}
We call $R_\Delta(\mu,\mu')$  the corresponding Wasserstein distance
between two measures $\mu$ and $\mu'$ in $\mathcal
X^{(k)}_\La\times\mathcal X^{(k)}_\La$:
    \begin{eqnarray}
    \nn
\hskip-1cm R_\Delta(\mu,\mu')&&=\inf_Q\int d_\Delta(q'_\La,q''_\La)
dQ(q'_\La,q''_\La)
\\&&=\inf_Q \sum_{x\in\ell_{-,\ga}\mathbb{Z}^d \cap
\Delta}Q\big(q'_\La\cap C^{(\ell_{-,\ga})}_x\ne
    q''_\La\cap C^{(\ell_{-,\ga})}_x\big)
    \label{a6.3}
    \end{eqnarray}
where the inf runs over all possible joint distributions (couplings)
of $\mu$ and $\mu'$.

\vskip1cm

 In Subsection
\ref{sec:II.2} we prove the following Theorem.
\bigskip

    \begin{thm}
    \label{thm:6.1}
Given $\La$ union of $N_\La$ cubes of $\mathcal D^{(\ell_+)}$ there
is $\eps_0=\eps_0(N_\La)$ such that for all $\bar
q_{i,\La^c}\in\mathcal X^{(k)}_{\La^c}$, $i=1,2$, the following
holds.

Given any $\und n'$, $\und n''$ such that $ n'_{\Delta_1}
=n''_{\Delta_1}=:n_{\Delta_1}$ ($\Delta_1$ defined in
\eqref{aaa6.1}), the following holds.

Calling $\mathring{\Delta}_{1}= \Delta_1\setminus\delta_{\rm
in}^{\ga^{-1}}[\Delta_1]$, for any two configurations $\bar
q_{i,\La\setminus\mathring{\Delta}_{1}}$, $i=1,2$ on $\mathcal
X^{(k)}_{\La\setminus\mathring{\Delta}_{1}}$, we denote by $\bar q_{i,\mathring{\Delta}_{1}^c}=\bar q_{i,\La\setminus\mathring{\Delta}_{1}}\cup \bar
q_{i,\La^c}$, $i=1,2$.

Let $dG^0_\La(q_{\mathring{\Delta}_{1}}|q_{i,\mathring{\Delta}_{1}^c},
n_{\Delta_1})$, $i=1,2$ be the probabilities  $dG^0_\La(\cdot | \bar
q_{i,\La^c})$, $i=1,2$ conditioned to have the configuration in
$\mathring{\Delta}_{1}^c$ equal to $\bar q_{i,\mathring{\Delta}_{1}^c}$ and
occupation numbers in $\Delta_1$ given by $ n_{\Delta_1}$.

Then for $\Delta_0$ defined in \eqref{a6.1}
    \begin{equation}
    \label{a6.4delta0}
{R_{\Delta_0}\big(dG^0_\La(\cdot|q_{1,\mathring{\Delta}_{1}^c},
n_{\Delta_1}),dG^0_\La(\cdot|q_{2,\mathring{\Delta}_{1}^c},
n_{\Delta_1})\big)\le \eps_0}
    \end{equation}
    \end{thm}

\bigskip
The next result, proved at the end of Subsection \ref{sec:II.6.5},
deals with the Wasserstein distance $R_{\Delta_1}$ of the
distributions of the occupation numbers $\und n$ that in Theorem
\ref{thm:6.1} have been set equal to each other inside $\Delta_1$.
For these variables the metric $d_x$ defined in \eqref{aa6.2a} is
replaced by
    $$
d_x( \und n',\und n'')=\begin{cases}
    0 &{\text {if }} n'(x,s)=
    n''(x,s), \forall s
    \\ 1  &{\text {otherwise }}
\end{cases}
    $$

    \begin{thm}
    \label{thm:6.2}
Given $\La$ union of $N_\La$ cubes of $\mathcal D^{(\ell_+)}$ there
is $\eps_1=\eps_1(N_\La)$ such that the following holds.
Let $G^0_\La(n_{\La}|q_{i,\La^c})$, $i=1,2$ be the marginals of
$dG^0_\La(q_\La | \bar q_{i,\La^c})$, $i=1,2$ on the variables
$n_{\La}$ defined in \eqref{enn}.

Then
    \begin{equation}
    \label{a6.4}
R_{\Delta_1}\big(dG^0_\La(n_{\La}|\bar q_{1,\La^c}),
dG^0_\La(n_{\La}|\bar q_{2,\La^c})\big)\le \eps_1
    \end{equation}
    \end{thm}

\vskip1cm

In Subsection \ref{sec:II.8} we show that Theorem \ref{thme3.7.1} is
a consequence of Theorems \ref{thm:6.1} and \ref{thm:6.2}.

 \vskip1cm
\subsection{Two properties of the Wasserstein distance in an abstract setting }
        \label{sec:a6.2}

\bigskip

Let $\Om$ be a complete, separable metric space with distance
$d(\om,\om')$ and let $R(\mu_1,\mu_0)$ be the corresponding
Wasserstein distance between two measures $\mu_1$ and $\mu_0$. Thus
    \begin{equation}
    \label{vase}
R(\mu_1,\mu_0)=\inf_Q\int d(\om,\om') Q(d\om,d\om')
    \end{equation}
where the inf runs over all possible joint distributions  of $\mu_1$
and $\mu_0$.

\vskip1cm

  \begin{thm}

  \label{thmII.2.1}
Let $\nu$ be a given positive measure on $\Om$. Let $h$ and $v$ be
such that for all $t\in [0,1]$,
    \begin{equation}
       \label{II.2.8}
  Z_t =\int
 e^{-[h(\om)+tv(\om)]}\nu(d\om) <\infty,
     \end{equation}
Set
         \begin{equation}
       \label{II.2.9}
 m_t(\om) = Z_t^{-1}
 e^{-[h(\om)+tv(\om)]},\quad  \mu_t(d\om)=m_t(\om)\nu(d\om)
     \end{equation}
Then
    \begin{equation}
       \label{II.2.10}
R(\mu_1,\mu_0) \le  \sup_{0\le t \le 1}\Big( \mu_t(|\om|\, |v| )+
\mu_t( |\om| )\mu_t(|v|)\Big)
     \end{equation}
where, after fixing arbitrarily an element $\om_0\in \Om$, we have
called $|\om|= d(\om,\om_0)$.

 In particular,
    \begin{equation}
       \label{II.2.11}
R(\mu_1,\mu_0) \le  2\big(\sup |\om|\big)\;\big(\sup |v(\om)|\big)
     \end{equation}

  \end{thm}

\vskip.5cm

{\bf Proof.}  Let
    \begin{equation*}
m(\om)= \min\{m_1(\om),m_0(\om)\},\qquad C= 1-\int m(\om)\nu(d\om)
     \end{equation*}
    \begin{equation*}
P(d\om d\om')= \{m(\om) \delta_{\om-\om'} +
 \frac 1C  [m_1(\om)- m(\om)][m_0(\om')-m(\om')]\}\nu(d\om)\nu(d\om')
     \end{equation*}
$P$ is a coupling of $\mu_1$ and $\mu_0$ and therefore
    \begin{eqnarray*}
R(\mu_1,\mu_0) &\le & \int_{ \Om\times \Om }d(\om,\om')P(d\om d\om')
\le \int_{ \Om} |\om| \big([m_1(\om)-m(\om)] +
 [m_0(\om)-m(\om)]\big)\nu(d\om)
  \\&=& \int_{ \Om}  |\om| \, |m_1(\om)-m_0(\om)|\nu(d\om)
     \end{eqnarray*}
having bounded $d(\om,\om')\le  |\om| + |\om'| $  and integrated
over the missing variable.

\eqref{II.2.10} is then obtained by writing $\dis{m_1(\om)-m_0(\om)
= \int_0^1 \frac{d}{dt} m_t(\om) }$.

\qed

\vskip.5cm

The following estimate is taken from \cite{leipzig}:

\vskip.5cm

  \begin{thm}
  \label{thmII.6.1}

Let $A\subset \Om$ be a measurable set,  $\mu$ a probability on
$\Om$ and $\mu_A$ the probability $\mu$ conditioned to $A$.  Then
    \begin{equation}
       \label{8z.4.6.5}
R(\mu,\mu_A) \le 2\sup_{\om\in\Om}|\om|\,\,\mu(A^c)
     \end{equation}

  \end{thm}
{\bf Proof.} Let
    \begin{eqnarray*}
Q(d\om,d\om')= \text{\bf 1}_{\om \in A}
\mu(d\om)\delta_{\om}(d\om') +\text{\bf 1}_{\om \in
A^c}\mu(d\om)\mu_A(d\om')
     \end{eqnarray*}
where $\delta_{\om}(d\om')$ is the probability supported by $\om$.
Let  $f$ be any bounded, measurable function on $\Om$, then
    \begin{eqnarray*}
\int f(\om)Q(d\om,d\om')= \int_{ A} f(\om) \mu(d\om) +\int_{
A^c}f(\om)\mu(d\om)\int \mu_A(d\om') = \mu(f)
     \end{eqnarray*}
    \begin{eqnarray*}
&&\int f(\om')Q(d\om,d\om')= \int_{ A} f(\om) \mu(d\om) +\mu(
A^c)\int f(\om')\mu_A(d\om') \\&&\hskip3cm = \mu_A(f)
\mu(A)+\mu_A(f) \mu(A^c)=\mu_A(f)
     \end{eqnarray*}
Hence $Q$ is a coupling and
    \begin{equation*}
R(\mu,\mu_A) \le \int d(\om,\om')Q(d\om,d\om') \le \int \text{\bf
1}_{\om \in A^c}(|\om|+|\om'|)\mu(d\om)\mu_A(d\om')
     \end{equation*}
which proves \eqref{8z.4.6.5}.  \qed

\vskip1cm

Eventually, we mention the following elementary property:

\begin{prop}Assume that the distance $d$ satisfies $m(d):=\dis{\inf_{\omega\neq \omega' \in \Omega}d(\omega,\omega')} >0 $. Then for all probability measures $\mu,\nu$ and for all $A \subset \Omega$
\begin{equation}
m(d) \cdot | \mu(A)-\nu(A) | \quad \leq \quad R(\mu,\nu)
\end{equation}
\end{prop}
\begin{proof}
Without loss of generality we assume $\mu(A) \geq \nu(A)$. Remarking that $\textbf{1}_{\omega \neq \omega'} \geq \textbf{1}_{\omega \in A}-\textbf{1}_{\omega' \in A}$, we get for any coupling $Q$ of $\mu,\nu$
\begin{equation}
m(d)(\mu(A)-\nu(A)) \leq \int d(\om,\om')\textbf{1}_{\omega \neq \omega'} G(d\om,d\om')
\end{equation}
and the proposition is proved by taking the infimum over all possible couplings $Q$.
\end{proof}

\begin{remark1}
The proposition above states that Wasserstein distances associated to very particular distances $d$ are finer than the \emph{total variation} distance $\dis{d_\textrm{TV}(\mu,\nu) := \sup_{A\subset \Omega} | \mu(A)-\nu(A) |}$. In the following, we will use this property for $R_{\Delta}$, remarking that $m(d_{\Delta}) = 1$.
\end{remark1}

\vskip1cm

\subsection{Couplings of multi-canonical measures}
        \label{sec:II.2}

        \vskip1cm

Here we prove Theorem \ref{thm:6.1}. Recalling that $\mathring{\Delta}_{1}=
\Delta_1\setminus\delta_{\rm in}^{\ga^{-1}}[\Delta_1]$, we fix two
boundary conditions $\bar q_{i,\mathring{\Delta}_{1}^c}=\bar
q_{i,\La\setminus\Delta_1}\cup \bar q_{i,\La^c}$, $i=1,2$. We have
to compare the marginal distributions of
 $dG^0_\La(q_{\bar\Delta_1}|q_{i,\mathring{\Delta}_{1}^c}, n_{\Delta_1})$,
$i=1,2$
%
%
over the configurations in $\Delta_0$ (i.e.\ well inside $\mathring{\Delta}_{1}$). Since the probabilities
$dG^0_\La(q_{\bar\Delta_1}|q_{i,\mathring{\Delta}_{1}^c}, n_{\Delta_1})$,
$i=1,2$ depend only on the restrictions of $q_{i,\mathring{\Delta}_{1}^c}$
to $\delta_{\rm out}^{\ga^{-1}}[\Delta_1]$ where $n'(x,s)=n''(x,s)$
the corresponding occupation numbers in the two measures are all equal to
each other. We will thus study couplings of multi-canonical
measures, hence the title of the Subsection.

\vskip1cm

 It is now convenient to label the particles.  To this
purpose we use  a multi-index $p=(C_x,s,j)$, where $C_x$ is the cube
of $\mathcal D^{(\ell_-)}$ where the particle  is; $s$ is its spin
and $j\in \{1,..,n(x,s)\}$ distinguishes among the particles in the
same cube with same spin.  We call $\mathcal L_{\mathring{\Delta}_{1}}$ the
set of labels
    $$\mathcal L_{\mathring{\Delta}_{1}}=\{p=(C_x,s,j), x\in \mathring{\Delta}_{1},
    s=1,\dots,S,j\in
\{1,..,n(x,s)\}\}
    $$
Observe that $\mathcal L_{\mathring{\Delta}_{1}}$ is determined by $\und
n_{\Delta_1}$ and we thus have the same labels for the two measures.
Given $p=(C_x,s,j)\in \mathcal L_{\mathring{\Delta}_{1}}$ we denote by $r_p$
a vector configuration $r_p=(r_j,s)$ with $r_j\in C_x$. We then
denote by $r_{\mathcal L_{\mathring{\Delta}_{1}}}=\{r_p, p\in \mathcal
L_{\mathring{\Delta}_{1}}\}$ a vector configuration in $\mathring{\Delta}_{1}$.
Analogously we define $r_{\mathcal L_{\mathring{\Delta}_{1}^c}}$. We then
call $ H_{\mathcal L_{\mathring{\Delta}_{1}}}(r_{\mathcal L_{\mathring{\Delta}_{1}}}|r_{\mathcal L_{\mathring{\Delta}_{1}^c}})$ the energy $H_{\mathring{\Delta}_{1},t}$ defined  in \eqref{e3.4.2} and with $ n_{\Delta_1}$
fixed as above.

Calling
   \begin{equation}
      \label{II.2.2}
d\nu_p(r)= \text{\bf 1}_{r\in C_x} dr
     \end{equation}
we define
   \begin{equation}
      \label{II.2.3}
P_{\mathcal L_{\mathring{\Delta}_{1}}}(dr_{\mathcal L_{\mathring{\Delta}_{1}}}|r_{\mathcal L_{\mathring{\Delta}_{1}^c}}) = Z(r_{\mathcal L_{\mathring{\Delta}_{1}^c}})^{-1} e^{-\beta H_{\mathcal L_{\mathring{\Delta}_{1}}}(r_{\mathcal L_{\mathring{\Delta}_{1}}}|r_{\mathcal L_{\mathring{\Delta}_{1}^c}})}\prod_{p\in \mathcal L_{\mathring{\Delta}_{1}}} \nu_{p}(dr)
     \end{equation}
\begin{remark1}
\label{rem_unlabel}
If  $A$ is  a $\mathcal D^{(\ell_-)}$ measurable subset of $\mathring{\Delta}_{1}$, then $\mathcal L_A$ denotes all labels $(C,s,j)$ with
$C\subset A$ and the marginal of $P_{\mathcal L_A}(dr_{\mathcal
L_A}|r_{\mathcal L_{A^c}})$ over the unlabeled configurations is the
original multi-canonical measure in $A$.

We will thus prove Theorem \ref{thm:6.1} if we can compare
\begin{equation}
\label{PprimePsecond}
P'=P_{\mathcal L_{\mathring{\Delta}_{1}}}(\cdot|r'_{\mathcal L_{\mathring{\Delta}_{1}^c}}) \text{ and } P''=P_{\mathcal L_{\mathring{\Delta}_{1}}}(\cdot|r''_{\mathcal L_{\bar \Delta_1^c}})
\end{equation}
by evaluating the Wasserstein distance $R_{\Delta_0}(P',P'')$.
\end{remark1}

We will use the Dobrushin high-temperature techniques which allow
to reduce to a comparison of the conditional probabilities of a
single variable  $r_p$.
  \begin{prop}[Dobrushin high-temperature theorem]

  \label{propII.2.1}
There is $c$  such that the following holds. For all
$p_0=(C_{x_0},s_0,j_0)$, $C_{x_0}\subset \mathring{\Delta}_{1}$, all
$p_1=(C_{x_1},s_1,j_1)$ and all $r'_{p_1}$ and $r''_{p_1}$
   \begin{equation}
      \label{II.2.7}
\sup_{\und r}R_{\Delta_0}\Big(P_{\mathcal L_{p_0}}\big(\cdot|\und
r,r'_{p_1}\big),P_{\mathcal L_{p_0}}\big(\cdot|\und
r,r''_{p_1}\big)\Big) \le c \ga^{d+\alpha_-} \text{\bf 1}_{{\rm
dist}(C_{x_0},C_{x_1}) \le \ga^{-1}}
     \end{equation}
where $\und r=(r_p)_{ p \ne p_0,p_1}$

  \end{prop}

\vskip.5cm

{\bf Proof.}  The probabilities to compare have  the form
   \begin{equation*}
 P_{\mathcal L_{p_0}}\big(dr|\und r,r'_{p_1}\big)=
 \frac 1{ Z(\und r,r'_{p_1})} e^{W_\ga(r)}\text{\bf
1}_{r\in C_{x_0}}\; dr
     \end{equation*}
while
   \begin{equation*}
 P_{\mathcal L_{p_0}}\big(dr|\und r,r''_{p_1}\big)=
 \frac 1{ Z(\und r,r''_{p_1})} e^{W_\ga(r)+
W_\ga'(r)}\text{\bf 1}_{r\in C_{x_0}}\; dr
     \end{equation*}
where $W_\ga(r)=-\beta V_\ga(r,r_{j_1}')$ and $W_\ga'(r)=-\beta\{
V_\ga(r,r''_{j_1})-V_\ga(r,r'_{j_1})\}$ hence
    \begin{equation*}
| W_\ga'(r)|\le \beta \sup_{r'\in C_{x_1}}|\nabla V_\ga(r,r')|
\ell_- \le c'\ga^{d+\alpha_-} \text{\bf 1}_{{\rm
dist}(C_{x_0},C_{x_1}) \le \ga^{-1}}
     \end{equation*}
 Proposition \ref{propII.2.1} then follows from
 Theorem \ref{thmII.2.1}.  \qed

\begin{remark1}
 From the proof above, we see that the r.h.s of \eqref{II.2.7} is actually proportionnal to $\beta \gamma^{d+\alpha_{-}}$. In other terms, the \emph{effective temperature} of the system is of order $\gamma^{-d-\alpha_{-}}$ and thus \emph{very high} indeed.
\end{remark1}

\vskip1cm

  \begin{coro}

  \label{coroII.2.1}
With $P',P''$ defined by \eqref{PprimePsecond}, there is $\eps_0$ such that for all $\ga$ small enough the following holds:
    \begin{equation}
    \label{a6.16}
  R_{\Delta_0}(P',P'') \le \eps_0
  \end{equation}
  \end{coro}

\vskip.5cm

{\bf Proof.} For $p_0$ and $p_1$ as in Proposition \ref{propII.2.1} we call $\delta(p_0,p_1)= c \ga^{d+\alpha_-}\text{\bf 1}_{{\rm dist}(C_{x_0},C_{x_1}) \le \ga^{-1}}$ (which is
the r.h.s.\ of \eqref{II.2.7}). Then there is $\varsigma>0$ such that for all $\ga$ small
enough the following holds:
    \begin{eqnarray*}
  R_{\Delta_0}(P',P'') &\le& \sum_{p_0\in \mathcal
L_{\La_0}} \sum_{n} \sum_{p_1,..,p_n \in \mathcal L_\Delta} \sum_{p
\notin \mathcal L_\Delta} \delta (p_0,p_1)\cdots \delta(p_n,p)\nn
\\&\le& e^{- \varsigma{\rm dist}(\Delta_0, \bar\Delta_1^c)}
     \end{eqnarray*}
The first inequality follows from the Dobrushin high-temperature theorem (Proposition \ref{propII.2.1}) while the second one is obvious once  $\dis{\sum_{p'\ne p} \delta(p,p') \le c
\ga^{\alpha_-}<1}$ (which is satisfied for all $\ga$ small enough).
\qed

\vskip.5cm

In view of Remark \ref{rem_unlabel}, the Theorem \ref{thm:6.1} is a straightforward consequence of \ref{coroII.2.1}. \qed

 \vskip1cm

\subsection{Taylor expansion}
        \label{sec:II.4}

\bigskip

In this subsection we consider the marginal of $dG_\La^0(q_\La|\bar
q_{\La^c})$ on the variables $\rho_\La=\ell_-^{-d}\,n_\La$,
$n_\La=\{n(x,s),x\in \ell_{-,\ga}\mathbb Z^d\cap\La$, $s\in
\{1,..,S\}$. By an abuse of notation we denote also the marginal
with $G_\La^0(\rho_\La|\bar q_{\La^c})$. 
%

 Recalling \eqref{ee4.2.1}
we get
    \begin{equation}
     \label{a6.17}
G_\La^0(\rho_\La|\bar q_{\La^c}) =\frac 1{Z^{\rm eff}(\bar
q_{\La^c})}e^{-\beta \ell_{-,\ga}^d H^{{\rm eff}}_\La(\rho_\La|\bar
q_{\La^c})}
     \end{equation}
Recalling \eqref{ee5.0.1} we also define
    \begin{equation}
     \label{a6.18}
G_\La^\star(\rho_\La|\bar q_{\La^c}) =\frac 1{Z^*(\bar
q_{\La^c})}e^{-\beta \ell_{-,\ga}^d f(\rho_\La;\bar q_{\La^c})}
     \end{equation}

\bigskip

The following holds:

        \begin{prop}
    \label{prop1}
For all $\bar q_{1,\La^c},\bar q_{2,\La^c}\in \mathcal
X^{(k)}_{\La^c}$,
    \begin{equation}
      \label{a6.20}
R_{\Delta_1}\big(G^0(\cdot|\bar q_{1,\La^c}),G^0(\cdot|\bar
q_{2,\La^c})\big) \le R_{\Delta_1}\big(G^\star(\cdot|\bar
q_{1,\La^c}),G^\star(\cdot|\bar q_{2,\La^c})\big) +
2c\ga^{\tau}
     \end{equation}
with $\tau$ given in \eqref{ee4.2.11}.
        \end{prop}

\vskip.6cm

{\bf Proof.} By \eqref{ee4.2.11} there is $c=c(N_\La)$ such that
     \begin{equation}
\label{a6.20bis}
 |H^{{\rm eff}}_\La(\rho_\La|\bar
q_{\La^c}) -f(\rho_\La;\bar q_{\La^c})|\le c \ga^{\tau}
    \end{equation}

By \eqref{a6.20bis} and Theorem \ref{thmII.2.1}, there is a (different)
constant $c>0$ such that
   \begin{equation}
    \label{a6.19}
R_{\Delta_1}\big(G^0(\cdot|\bar q_{\La^c}),G^\star(\cdot|\bar
q_{\La^c})\big) \le c \ga^{\tau}
     \end{equation}
Hence the triangular inequality implies \eqref{a6.20}.\qed

\bigskip

We will bound $R_{\Delta_1}\big(G^\star(\cdot|\bar
q_{1,\La^c}),G^\star(\cdot|\bar q_{2,\La^c})$ by using the
triangular inequality to replace the two measures by their Taylor
approximants.

\medskip

We first prove the following result true for any $\mathcal
D^{(\ell_+)}$-measurable region $\La$.

\medskip
        \begin{thm}
 \label{thmII.5.1}
For any $\bar q_{\La^c}\in\mathcal X^{(k)}_{\La^c}$, calling
$\mu=G^\star_\La(\cdot|\bar q_{\La^c})$, the following holds.

There are $c>0$ and $\delta<1/2$ that verifies \eqref{delta} below,
so that, calling $\hat\rho_\La$ the minimizer of $f(\rho_\La;\bar
q_{\La^c})$
   \begin{equation}
      \label{II.5.1}
\mu \Big( \{\exists x\in \La,\exists
s:|\rho_\La(x,s)-\hat\rho_\La(x,s)| \ge \ell_-^{-d/2+\delta}\}\Big)
\le e^{-c \ell_-^{2\delta}}
     \end{equation}

 \end{thm}

 \vskip.5cm

{\bf Proof.}  Denoting simply $A:= \{\exists x\in \La,\exists
s:|\rho_\La(x,s)-\hat\rho_\La(x,s)| \ge \ell_-^{-d/2+\delta}\} $ we have
    \begin{equation*}
\mu \Big( \{\exists x\in \La,\exists
s:|\rho_\La(x,s)-\hat\rho_\La(x,s)| \ge \ell_-^{-d/2+\delta}\} \Big) = \frac 1{Z^*(\bar
q_{\La^c})}\sum_{\rho_\La\in \mathcal X^{(k)}_\La}e^{-\beta \ell_-^d
f(\rho_\La;\bar q_{\La^c})} \textbf{1}_{A}(\rho_\La)
    \end{equation*}
By Theorem \ref{thmIe5.3.2} we have that
   \begin{equation*}
f(\rho_\La, \bar q_{\La^c})\ge
 f(\hat\rho_\La,
\bar q_{\La^c})
 + \frac{\kappa}{2}\; \big(\rho_\La-\hat\rho_\La,\rho_\La-
 \hat\rho_\La\big)
     \end{equation*}
Thus calling $\dis{C=\big(\sum_{n=0}^\infty e^{-\beta \frac
{\kappa}{2} n^2}\big)^{SN_{\La}}}$ we get
\begin{align*}
  \sum_{\rho_\La\in \mathcal X^{(k)}_\La}e^{-\beta
    \ell_-^d f(\rho_\La;\bar q_{\La^c})}& \textbf{1}_{A}(\rho_\La) \\
& \le
  e^{-\beta  \ell_-^d f(\hat \rho_\La;\bar q_{\La^c})} \sum_{\rho_\La\in \mathcal X^{(k)}_\La}\exp
  \big\{-\beta \ell_-^d \frac {\kappa}{2}\;
  \sum_{y,s}[\rho_\La(y,s)-\hat\rho_\La(y,s)]^2
  -\beta \frac {\kappa}{2}\ell_-^{2\delta}
 \big\}
 \\
&\le e^{-\beta  \ell_-^d f(\hat \rho_\La;\bar q_{\La^c})} e^{-\beta \frac
 {\kappa}{2}\ell_-^{2\delta}}\,\,\Big[\big(\sum_{n=0}^\infty
 e^{-\beta \frac
 {\kappa}{2} n^2}\big)^S\Big]^{|\La|/\ell_-^d}
\\
& \le e^{-\beta  \ell_-^d f(\hat \rho_\La;\bar q_{\La^c})} e^{-\beta \frac{\kappa}{2}\ell_-^{2\delta}}\,C^{(\ell_+/\ell_-)^d}
\end{align*}

We bound the partition function as follows, with $0<\eps$ a small constant to be chosen later:
     \begin{align*}
Z^*(\bar
q_{\La^c}) & \ge \sum_{\rho_\La\in \mathcal X^{(k)}_\La}
e^{- \beta\ell_-^d f(\rho_\La;\bar q_{\La^c})}\text{\bf
1}_{\{|\rho_\La(x,s)-\hat\rho_\La(x,s)| \le \eps\ell_-^{-d/2+\delta}
\forall x,\forall s\}} \\
 & \ge e^{-\beta  \ell_-^d f(\hat \rho_\La;\bar q_{\La^c})}  e^{-\beta\frac{C'\eps^{2}}{2}\ell_-^{2\delta}} (\eps\ell_-^{-d/2+\delta})^{S(\ell_+/\ell_-)^d},
    \end{align*}
so that
\begin{align*}
\mu \Big(\{|\rho_\La(x,s)-\hat\rho_\La(x,s)| \ge
\ell_-^{-d/2+\delta}\}\Big) &
 \le \exp\{-\left[\beta \frac{\kappa-C'\eps^{2}}{2}- \,\, \ell_-^{-2\delta}\left(\frac{\ell_+}{\ell_-}\right)^{d}\log (C \eps^{-1}\ell_{-}^{d/2-\delta})\right]  \ell_-^{2\delta}\}.
\end{align*}
Remark now that
\begin{align*}
\ell_-^{-2\delta} \left(\frac{\ell_+}{\ell_-}\right)^{d}\log (C \eps^{-1}\ell_{-}^{d/2-\delta}) & =  a \ga^{b}\left(\log \gamma\right)^{c}
\end{align*}
with $ a = (d/2-\delta)(1-\alpha_{-})>0$, $b = (1-\alpha_-)2\delta
 -(\alpha_++\alpha_-)d$ and $c = C\eps^{-1}>0$. Choosing $\delta$ such that $b>0$, i.e.
        \begin{equation}
        \label{delta}
\dis{\delta> \frac {(\alpha_++\alpha_-)d}{2(1-\alpha_-)}}
    \end{equation}
 which is always possible (see \eqref{ea3.1.1}), we get $\gamma^{b}(\log\gamma)^{c} \to 0$ as $\gamma \to 0$. The Theorem is now proved with $0<c < \beta\frac{\kappa-C'\eps^{2}}{2}$, which is always possible for $\eps$ small enough.\qed

\vskip1cm We  call $\hat\rho_{\La,i}$ the minimizer of $f(\cdot;\bar
q_{i,\La^c})$, $i=1,2$. We then let
    \begin{equation}
      \label{a4.55}
    A_{\le,i}=\big\{\rho_\La\in \mathcal X^{(k)}:
    |\rho_\La(x,s)-\hat\rho_{\La,i}(x,s)| \le
\ell_-^{-d/2+\delta}, \forall x,\forall s\big\},\qquad i=1,2
        \end{equation}

\bigskip

    \begin{prop}
    \label{prop2}
For all
 $\bar q_{i,\La^c}\in \mathcal
X^{(k)}_{\La^c}$, $i=1,2$,
    \begin{equation}
      \label{II.5.2}
R_{\Delta_1}\big(G^\star_\La(\rho_\La| \bar
q_{1,\La^c}),G^\star_\La(\rho_\La| \bar q_{2,\La^c}) \big)\le
R_{\Delta_1}\big(G^\star_\La(\rho_\La| \bar q_{1,\La^c},
A_{\le,1}),G^\star_\La(\rho_\La| \bar q_{2,\La^c}, A_{\le,2}) \big)
+
 2c e^{-c \ell_-^{2\delta}}
     \end{equation}

where $G^\star_\La(\rho_\La| \bar q_{i,\La^c}, A_{\le,i})$ $i=1,2$
are
 the probabilities $G^\star_\La(\cdot| \bar q_{i,\La^c})$
 conditioned to $ A_{\le,i}$, $i=1,2$.
 \end{prop}
\medskip

{\bf Proof.} \eqref{II.5.2} follows from Theorem \ref{thmII.5.1} and
Theorem \ref{thmII.6.1}. \qed

\vskip1cm

Analogously to \eqref{aaa6.1} we define the following subset of
$\La$.
    \begin{equation}
    \label{aa6.28}
\Delta_2=\bigcup_{x\in\Delta_1} B_x(10^{-30}\ell_{+,\ga})\cap\La
    \end{equation}
and we observe that $\Delta_2\supset \Delta_1$,
dist$\dis{({\Delta_1,\Delta_2^c})> 10^{-30}\ell_+}$. We also have

    \begin{lemma}
    \label{lemma6.10}
Let $\hat K$ be as in Theorem \ref{thmee5.0}. Then $\hat K(x)>0$ for
all $x\in \Delta_2$.
    \end{lemma}
\medskip

{\bf Proof.} Let $x\in \Delta_2$, by definition of $\hat K(x)$, if
$\hat A_x=B_x(10^{-30}\ell_{+,\ga})\cap\La^c=\emptyset$ then $\hat
K(x)=\bar m + 1>0$. Assume then that $\hat A_x\ne\emptyset$. By
\eqref{aa6.28} and \eqref{aaa6.1} there is $x_0\in\Delta_0$ such
that $|x-x_0|\le (1+10^{-10})10^{-20}\ell_{+,\ga}$, thus  $\hat
A_x\subset A_{x_0}=B_x(10^{-10}\ell_{+,\ga})\cap\La^c$ and
therefore $A_{x_0}\ne\emptyset$. By definition of $\Delta_0$ we
then have that $q'_{\La^c}\cap \hat A_x= q''_{\La^c}\cap \hat A_x$
and also that $K(x_0)=m+1>0$ with  $m\geq 2$ where $m$ is given by
$\dis{\max_{r\in A_{x_0}, s\in \{1,..,S\}}|
\rho^{(\ell_{-,\ga})}({\bar q}'_{\La^c};r,s)-\rho^{(k)}_s|}\in
[\zeta_{m+1},\zeta_m)$. Then $\dis{\max_{r\in \hat A_x, s\in
\{1,..,S\}}| \rho^{(\ell_{-,\ga})}({\bar
q}'_{\La^c};r,s)-\rho^{(k)}_s|}<\zeta_m$, that implies that $\hat
K(x)>0$. \qed

\bigskip

Recalling that $\hat\rho_{i,\La}$ is the minimizer of $f(\cdot; \bar
q_{i,\La^c})$, $i=1,2$, we observe that in general the gradient of
$D_{\La}f$ (see \eqref{ee5.1.3} for notation), evaluated at
$\hat\rho_{i,\La}$ does not vanishes in all $\La$. However, by
Theorem \ref{thmee5.0} and Lemmas \ref{lemmaee.5.2.3}, \ref{lemma6.10} it follows that
$D_{\Delta_2}f(\hat \rho_{i,\La}; \bar q_{i,\La^c})=0$.

$N$ being defined by Theorem \ref{thme4.2.1}, we set $\bar \Delta_2 = \Delta_2\cup \delta_{\rm out}^{\ga^{-1}N}[\Delta_2]$ and define
\begin{equation}
  \label{6.32}
  \rho_i^*(x,s)=
  \begin{cases} \hat\rho_{1,\La}(x,s) &\text{if
      $x\in\ell_{-,\ga}\mathbb{Z}^d \cap \bar \Delta_2$ }
    \\\hat\rho_{i,\La}(x,s)&\text{if $x\in\ell_{-,\ga}\mathbb{Z}^d
      \cap(\La\setminus \bar\Delta_2)$}
  \end{cases}
\end{equation}
Thus $\rho_2^*=\hat\rho_{1,\La}$ in $\bar\Delta_2$ while $ \rho_1^*(x,s)=\hat\rho_{1,\La}(x,s)$ for all $x\in\ell_-\mathbb
Z^d\cap \La$ and $\forall s$. We denote by $\rho^*$ the common
value, thus
    \begin{equation}
    \label{rostar}
\rho^*(x,s)=\rho^*_1(x,s)=\rho^*_2(x,s),   \qquad \forall
x\in\ell_{-,\ga}\mathbb{Z}^d\cap \bar\Delta_2, \forall s
    \end{equation}

 We also define  the matrix $B_{i,\La}$ with entries:
    \begin{equation}
      \label{6.33}
      B_{i,\La}(x,s,x',s')=
      \begin{cases}
        D^2_\La f(\hat \rho_{1,\La};\bar q_{1,\La^c})(x,s,x',s')  &\text{if
$x,x'\in\ell_{-,\ga}\mathbb{Z}^d\cap \bar\Delta_2$ } \\
        D^2_\La f(\hat \rho_{i,\La};\bar q_{i,\La^c})(x,s,x',s')  & \text{otherwise}
      \end{cases}
    \end{equation}
Observe that  $B_{1,\La}=D^2_\La f(\hat \rho_{1,\La}; \bar
q_{1,\La^c})$. We denote by $B$ the two matrices restricted to
$\Delta_2\cup \delta_{\rm out}^{\ga^{-1}N}[\Delta_2]$ which are then
equal; their common entries are then
    \begin{equation}
    \label{bmatrix}
B(x,s,x',s')=B_{1,\La}(x,s,x',s') = B_{2,\La}(x,s,x',s')  \qquad
\forall x,x'\in\ell_{-,\ga}\mathbb{Z}^d\cap(\bar\Delta_2), \forall s
    \end{equation}

\medskip

We define for $i=1,2$
    \begin{equation}
    \label{6.27}
\varphi_i(\rho_\La; \bar q_{i,\La^c})= \Big(D_{\La} f(\hat
\rho_{i,\La}; \bar q_{i,\La^c}),[\rho_\La-\rho^*_i]\Big)+\frac 12
\Big([\rho_\La-\rho^*_i], B_{i,\La}[\rho_\La-\rho^*_i]\Big)
    \end{equation}
and the probabilities
    \begin{equation}
    \label{6.28}
\mu_i(\rho_\La):=\frac
1{Z_{i,\La}}e^{-\beta\ell_-^d\,\varphi_i(\rho_\La; \bar
q_{i,\La^c})}\chi_{A_{\le,i}}(\rho_\La),\qquad
Z_{i,\La}=\sum_{\rho_\La}e^{-\beta\ell_-^d\,\varphi_i(\rho_\La; \bar
q_{i,\La^c})}\chi_{A_{\le,i}}(\rho_\La)
    \end{equation}
where $\chi_A$ is the characteristic function of the set $A$:
\medskip

The following holds:
 \begin{prop}
    \label{prop3}
For all
 $\bar q_{i,\La^c}\in \mathcal
X^{(k)}_{\La^c}$, $i=1,2$, and for all $\eps_2>0$ if $\ga$ is small
enough the following holds:
    \begin{equation}
      \label{6.29tay}
R_{\Delta_1}(G^\star_\La(\rho_\La| \bar q_{1,\La^c},
A_{\le,1}),G^\star_\La(\rho_\La| \bar q_{2,\La^c}, A_{\le,2})  \le
R_{\Delta_1}(\mu_1,\mu_2)+  2c \ga^{d/4}+\eps_2
     \end{equation}
 \end{prop}

\medskip

{\bf Proof.} We Taylor expand $f( \rho_{\La}; \bar q_{i,\La^c})$ and
we call $\mathcal{R}_i$ the third order.
    \begin{eqnarray}
    \nn
&&\hskip-3cm\mathcal{R}_i:=f( \rho_{\La}; \bar q_{i,\La^c})-
f(\hat\rho_{i,\La}; \bar q_{i,\La^c})
 -\Big(D_{\La} f(\hat \rho_{i,\La};
\bar q_{i,\La^c}),[\rho_\La-\hat\rho_{i,\La}]\Big)\\&& - \frac 12
\Big([\rho_\La-\hat \rho_{i,\La}], D^2_\La f(\hat \rho_{i,\La}; \bar
q_{i,\La^c})[\rho_\La-\hat\rho_{i,\La}]\Big)
     \label{a13.8}
     \end{eqnarray}
Observe that in $A_{\le,i}$ and for a suitable constant $c_1$
    \begin{equation*}
\beta\ell_-^d|\mathcal{R}_i|\le c_1\beta\ell_-^d \sum_{x,s}
|\rho_{\La}(x,s)-\hat\rho_{i,\La}(x,s)|^3\le c_1 \ell_-^d
\big(\frac{\ell_+}{\ell_-}\big)^d \ell_-^{3\delta-3d/2}
     \end{equation*}
and conclude that the right hand side of the above inequality is
estimated by $c \ga^{d/4}$ as soon as $\delta$ satisfies
\begin{equation}
\label{delta2}
\delta < \frac{d}{6}\left[\frac{1}{2}-3\alpha_--2\alpha_+\right]
\end{equation}
which is compatible with \eqref{delta}, see \eqref{ea3.1.1}.

Since $B_{1,\La}=D^2_\La f(\hat \rho_{1,\La}; \bar q_{1,\La^c})$ and
$\rho^*_1=\hat\rho_{1,\La}$, by applying Theorem \ref{thmII.2.1}
with $v=\beta\ell_-^d\mathcal{R}_1$ and $h= \beta\ell_-^d(f(
\rho_{\La}; \bar q_{i,\La^c})
 -\mathcal{R}_1)$  we get  that
  \begin{equation}
    \label{6.25}
R_{\Delta_1}\Big(G^\star_\La(\rho_\La| \bar q_{1,\La^c},
A_{\le,1}),\mu_1 \Big)\le c \ga^{d/4}
    \end{equation}
From Lemma \ref{lemma6.10} and (i) of Theorem \ref{thmee5.0}
we get that given any $\eps_2$ for $\ga$ small enough.
    \begin{eqnarray}
    \nn
&& \Big|\frac {\beta\ell_-^d}2 \Big([\rho_\La-\hat
\rho_{2,\La}], D^2_\La f(\hat \rho_{2,\La}; \bar
q_{2,\La^c})[\rho_\La-\hat\rho_{2,\La}]\Big)- \frac {\beta\ell_-^d}2
\Big([\rho_\La-\rho^*_2],
B_{2,\La}[\rho_\La-\rho^*_2]\Big)\Big|
\\ &&\hskip1cm \nn
\le \Big|\frac {\beta\ell_-^d}2 \Big([\rho_\La-\hat
\rho_{2,\La}], \left(D^2_\La f(\hat \rho_{2,\La}; \bar
q_{2,\La^c})-B_{2,\La}\right) [\rho_\La-\hat\rho_{2,\La}]\Big)_{\bar \Delta_{2}}\Big|
\\ && \hskip2cm \nn
+ \Big| \frac {\beta\ell_-^d}2\Big([\hat \rho_{1,\La} - \hat \rho_{2,\La}],B_{2,\La}[\hat \rho_{1,\La} - \hat \rho_{2,\La}]\Big)_{\bar\Delta_{2}}\Big|
\\ &&\hskip1cm \nn
\le \Big|\frac {\beta\ell_-^d}2 \Big([\rho_\La-\hat
\rho_{2,\La}], \left(D^2_\La f(\hat \rho_{2,\La}; \bar
q_{2,\La^c})-D^2_\La f(\hat \rho_{1,\La}; \bar
q_{2,\La^c})\right) [\rho_\La-\hat\rho_{2,\La}]\Big)_{\bar \Delta_{2}}\Big|
\\ && \hskip2cm \nn
+ \Big| \frac {\beta\ell_-^d}2\Big([\hat \rho_{1,\La} - \hat \rho_{2,\La}],B_{2,\La}[\hat \rho_{1,\La} - \hat \rho_{2,\La}]\Big)_{\bar\Delta_{2}}\Big|
\\&&\hskip1cm \nn
\le \frac {\beta\ell_-^d}2 (\ell_-^{2\delta}+1  )\sum_{x\in\ell_{-,\ga}\mathbb{Z}^d\cap \bar\Delta_2}
c e^{-10^{-30}(\ga\ell_+)\hat\om}\le \eps_2
     \label{6.34}
    \end{eqnarray}

By applying Theorem \ref{thmII.2.1} with
$v=\beta\ell_-^d[\mathcal{R}_2- \frac 12 ([\rho_\La-\rho^*_2],
B_{2,\La}[\rho_\La-\rho^*_2])]$ and $h= \beta\ell_-^d(f( \rho_{\La};
\bar q_{2,\La^c})
 -v)$  we get  that
    \begin{equation}
    \label{6.35}
R_{\Delta_1}\Big(G^\star_\La(\rho_\La| \bar q_{2,\La^c},
A_{\le,2}),\mu_2 \Big)\le c \ga^{d/4}+\eps_2
    \end{equation}
By using the triangular inequality we then get
\eqref{6.29tay}.\qed

 \vskip1cm
\subsection{Quadratic approximation in continuous variables }
        \label{sec:II.5}

\bigskip

In this subsection we consider the conditional probabilities
$\mu_i(\cdot|\bar \rho_{i,\La\setminus\Delta_2})$, $\bar
\rho_{i,\La\setminus\Delta_2}\in A_{\le,i}$, $i=1,2$. Since
$D_{\Delta_2} f(\hat \rho_{i,\La}; \bar q_{i,\La^c})=0$, and recalling
\eqref{rostar} and \eqref{bmatrix}, we have that
    \begin{equation}
\mu_i(\rho_{\Delta_2}|\bar \rho_{i,\La\setminus\Delta_2}):=\frac{e^{-\beta\ell_-^d\Big[\frac 12 \big([\rho_{\Delta_2}-\rho^*],
B_{\Delta_2}[\rho_{\Delta_2}-\rho^*]\big)+\big([\rho_{\Delta_2}-\rho^*],
B[ \bar \rho_{i,\La\setminus\Delta_2}-\rho^*]\big)
\Big]}\chi_{A_{\le,i}}(\rho_{\Delta_2})}{Z_{i,\Delta_2}(\bar
\rho_{i,\La\setminus\Delta_2})}
    \label{6.36}
    \end{equation}
where $B_{\Delta_2}$ is the matrix $B$ restricted to $\Delta_2$ and
where, as usual, $Z_{i,\Delta_2}(\bar
\rho_{i,\La\setminus\Delta_2})$ is the sum over $\rho_{\Delta_2}$ of
the numerator on the right hand side of \eqref{6.36}.

\medskip

We compare the probabilities  $ \mu_i(\cdot|\bar
\rho_{i,\La\setminus\Delta_2})$ with measures $p_i$ with the same
energy but with continuous state space. To define these measures we
start by setting some notations.

%

\smallskip

By convenience we consider the variables $n_{\Delta_2}
=\ell_-^{d}\rho_{\Delta_2}$, thus $n_{\Delta_2}=(n(x,s), x\in
\ell_-\mathbb{Z}^d\cap \Delta_2, s\in \{1,..,S\})$.  Since $\mu_i$,
$i=1,2$ defined in \eqref{6.36} have support on $A_{\le,i}$, the
variables $n_{\Delta_2}$ are such that
   \begin{equation}
      \label{6.37}
[n(x,s)-a^*(x,s)]\in \Big\{-M,
 -M+1,\dots, M\Big\},\qquad
 a^*(x,s)=\ell_-^{d}\rho^*(x,s)
     \end{equation}
where  $M$ is the integer part of $\ell_-^{d/2 +\delta}$ ($\delta$
as in Theorem \ref{thmII.5.1}).

We call $ \xi=(\xi(x,s), x\in \ell_-\mathbb{Z}^d\cap \Delta_2, s\in
\{1,..,S\})$ with
   \begin{equation}
      \label{6.38}
 \xi(x,s) =  \ell_-^{-d/2} [n(x,s)-a^*(x,s)]
     \end{equation}
and we denote by $X_M=\Big\{ \xi: \xi(x,s)\in\{-M,
 -M+1,\dots, M\}\Big\}$. In this new variables the boundary
 conditions become
    \begin{equation}
      \label{a6.39}
\xi^*_i= \ell_-^{-d/2} B[ \bar n_{i,\La\setminus\Delta_2}-a^*],
\qquad \bar n_{i,\La\setminus\Delta_2}=\ell_-^d \bar
\rho_{i,\La\setminus\Delta_2}
     \end{equation}

By an abuse of notation we call $\mu_i( \xi| \xi^*_i)$ the
distribution of the variables $ \xi$ under the probabilities
$\mu_i(\cdot|\bar \rho_{i,\La\setminus\Delta_2})$ defined in
\eqref{6.36}, thus
            \begin{equation}
\mu_i( \xi|\xi^*_i)=\frac 1{Z( \xi^*_i)}e^{-\beta\big[\frac 12 (
\xi, B_{\Delta_2} \xi)+(\xi,\xi^*_i)\big]}
    \label{6.39}
   \end{equation}
where $Z( \xi^*_i)$ is the sum over $ \xi\in X_M$ of the numerator.

We next introduce  variables $\und r=(r(x,s),x\in \Delta_2, s\in
\{1,..,S\})$ which take values in the interval of the real line:
   \begin{equation}
      \label{II.6.5}
r(x,s) \in  \ell_-^{-d/2} [-M,M+1]
     \end{equation}
and we call
\begin{equation}
      \label{6.39ym}
Y_M=\Big\{\und r: r(x,s)\in\ell_-^{-d/2}[-M,M+1], \forall x\in
\Delta_2, s\in \{1,..,S\})\Big\}
     \end{equation}
We next define the probabilities measures on $Y_M$ as
   \begin{equation}
      \label{6.40}
dp_i(\und r| \xi^*_i)=  \frac 1{Z_M( \xi^*_i)}e^{-\beta\big[\frac 12
(\und r, B_{\Delta_2}\und r)+(\und r,  \xi^*_i)\big]}\chi_{Y_M}(\und
r) d\und r,\qquad i=1,2
     \end{equation}
where $\dis{d\und r=\prod_{x,s}dr(x,s)}$ and $Z_M( \xi^*_i)$ is the
integral of the numerator.

\smallskip

 \begin{prop}
    \label{prop10}
For all
 $\bar
\rho_{i,\La\setminus\Delta_2}\in A_{\le,i}$, recalling \eqref{a6.39}
 the following holds:
    \begin{equation}
      \label{6.29}
R_{\Delta_1}\Big(\mu_1(\cdot| \xi^*_1)),\mu_2(\cdot| \xi^*_2))\Big)
\le R_{\Delta_1}\Big(p_1(\cdot| \xi^*_1),p_2(\cdot| \xi^*_2)\Big)+
2c \ga^{d/4}
     \end{equation}
 \end{prop}

\medskip

{\bf Proof.} Given $\xi\in X_M$ we call  $C( \xi)=\{\und r:0\le
r(x,s)-\xi(x,s)< \ell_-^{-d/2}, \forall x\in \Delta_2, \forall s\}$
we define $H'(\und \xi|\xi^*_i)$ as
\begin{equation}
      \label{a14.5}
e^{-H'( \xi|\xi^*_i)}:= \int_{C( \xi)}e^{-\beta\ell_-^{-d}\big[\frac
12 (\und r, B_{\Delta_2}\und r)+(\und r, \xi^*_i)\big]} d\und r
     \end{equation}
and the following probabilities $m_i$ on  $X_M$
\begin{equation}
      \label{II.6.9}
\dis{m_i(\und \xi)=\frac{e^{-H'(\und \xi|\xi^*_i)}}{\sum_{\und
\xi\in X_M}e^{-H'(\und \xi|\xi^*_i)}}},\qquad i=1,2
     \end{equation}
By continuity there is a point $\und r_{ \xi}\in C( \xi)$ such that
   \begin{equation}
      \label{6.30}
H'(\xi|\xi^*_i)=\beta\big[\frac 12 (\und r_{ \xi}, B_{\Delta_2}\und r_{
\xi})+(\und r_{\xi}, \xi^*_i)\big]
     \end{equation}
Therefore
   \begin{equation}
      \label{II.6.11}
\Big|H'( \xi)-\beta\big[\frac 12 ( \xi, B_{\Delta_2} \xi)+( \xi,
\xi^*_i)\big]\Big|\le \sup_{r\in C(\xi)}\|\nabla\{ (\und r,B_{\Delta_2}
\und r)/2+ \xi^*_i)\}\| \ell_-^{-d/2}
     \end{equation}
where $\nabla \psi(\und r)$ is the vector defined as the gradient of
$\psi$ with respect to the variables $r(x,s)$ and $\|\cdot\|$ is the
norm of the vector $\cdot$.

Since $\|B_{\Delta_2}\|\le c^*\frac{|\Delta_2|}{\ell_-^d} $  then
   \begin{equation}
      \label{II.6.12}
\Big|H'( \xi)-\beta\big[\frac 12 ( \xi, B_{\Delta_2} \xi)+( \xi,
\xi^*_i)\big]\Big|\le c^* \frac{|\Delta_2|}{\ell_-^d}S \ell_-^\delta
\ell_-^{-d/2}\le c^*S N_\La\Big(\frac
{\ell_+}{\ell_-}\Big)^d\ell_-^{-d/2+\delta}
     \end{equation}
For $\ga$ small  $\Big(\frac
{\ell_+}{\ell_-}\Big)^d\ell_-^{-d/2+\delta}\le \ga ^{d/4}$, thus by
Theorem \ref{thmII.2.1} and the triangular inequality we get
    \begin{equation}
      \label{6.29bis}
R_{\Delta_1}\Big(\mu_1(\cdot|\bar
\rho_{1,\La\setminus\Delta_2}),\mu_2(\cdot|\bar
\rho_{2,\La\setminus\Delta_2})\Big)  \le R_{\Delta_1}(m_1,m_2)+ 2c
\ga^{d/4}
     \end{equation}
We now observe that at any coupling $Q$ of $p_1$ and $p_2$ we can
associate a coupling $Q^*$ of $m_1$ and $m_2$ by setting
        $$
Q^*(\xi',\xi'')=Q\big(C(\xi')\times C(\xi'')\big)
    $$
To prove that $Q^*$ is indeed a coupling of $m_1$ and $m_2$ we
compute for any function $\psi$ on $X_M$
    \begin{eqnarray*}
&&\hskip-3cm\sum_{\xi''}\sum_{\xi'}\psi(\xi')Q^*(\xi',\xi'')=
\sum_{\xi'}\psi(\xi')p_1(C(\xi'))\\&&\hskip1cm=\frac 1{Z_M(
\xi^*_i)}\sum_{\xi'}\psi(\xi') \int_{C(\xi')}
e^{-\beta\ell_-^{-d}\big[\frac 12 (\und r, B_{\Delta_2}\und r)+(\und r,
\xi^*_i)\big]} d\und r
\\&&\hskip1cm=\sum_{\xi'}\psi(\xi') m_1(\xi')
    \end{eqnarray*}
Thus
    \begin{equation}
      \label{6.52}
\forall Q, \quad  R_{\Delta_1}(m_1,m_2)\le \sum_{\xi'',\xi'}d_{\Delta_1}(\xi',\xi')Q^*(\xi',\xi'')
     \end{equation}
 We next observe that
    \begin{eqnarray*}
&&\hskip-3cm\sum_{\xi'',\xi'}d_{\Delta_1}(\xi',\xi')Q^*(\xi',\xi'')=
\sum_{\xi',\xi'}\int_{C(\xi')\times C(\xi'')}d_{\Delta_1}(\xi',\xi')
dQ(\und r',\und r'')\\&&\hskip1cm\le
\sum_{\xi',\xi'}\int_{C(\xi')\times C(\xi'')}d_{\Delta_1}(\und
r',\und r'') dQ(\und r',\und r'')
    \end{eqnarray*}
Taking the $\inf$ over the coupling $Q$ in the above inequality
and using \eqref{6.52}, we get that  $ R_{\Delta_1}(m_1,m_2)\le
R_{\Delta_1}\Big(p_1(\cdot| \xi^*_1),p_2(\cdot| \xi^*_2)\Big)$,
thus \eqref{6.29bis} implies \eqref{6.29}. \qed

 \vskip1cm

\subsection{Gaussian approximation}
        \label{sec:II.6.5}
We now  extend the measures $p_i(\cdot| \xi^*_1)$ on $Y_M$ to a
measures $P_i$, $i=1,2$, on the full Euclidean space, thus $P_i$,
$i=1,2$ are  the Gaussian measure defined by the r.h.s.\ of
\eqref{6.40} without the last characteristic function.

Thus letting $\und r=( r(x,s)\in \mathbb{R}^d: x\in \Delta_2,
s\in\{1,\dots S\})$,
    \begin{equation}
    dP_i(\und r|\xi^*_i)= \frac 1{Z(\xi^*_i)}
    e^{-\beta\big[\frac 12 (\und r, B_{\Delta_2}\und r)+(\und r,
\xi^*_i)\big]}d\und r
     \label{II.7.1}
     \end{equation}
 with $Z(\xi^*_i)$ the integral of the
numerator.

\bigskip

The following holds:
 \begin{prop}
    \label{prop11}
There is $\delta^*>0$ such that  the following holds:
    \begin{equation}
      \label{6.29pp}
R_{\Delta_1}\Big(p_1(\cdot| \xi^*_1),p_2(\cdot| \xi^*_2)\Big) \le
R_{\Delta_1}\Big(P_1(\cdot|\xi^*_1),P_2(\cdot| \xi^*_2)\Big)+
2\ga^{\delta^*}
     \end{equation}
 \end{prop}

\medskip

{\bf Proof.}
By the Chebischev's inequality, and recalling that $\mathrm{Var} P_{i}(\cdot | \xi_i^\ast) = \| B_{\Delta_2}\|^{-1}$, there is $c$ such that
    \begin{equation*}
P_i\big(\{|r(x,s)|\ge \ell_-^\delta\}\big)\le c
\ell_-^{-2\delta}(\frac{\ell_+}{\ell_-})^{-d},\qquad i=1,2
     \end{equation*}
By \eqref{delta} there is $\delta^*>0$ such that
    \begin{equation}
    \label{6.50}
P_i\big(Y_M^c\big)\le \sum_s\sum_{x\in\ell_-\cap\mathbb{Z}^d\cap
\Delta_2}P_i\big(\{|r(x,s)|\ge \ell_-^\delta\}\big)\le \ga^{\delta^*}
    \end{equation}
Since $p_i$ is equal to the probability $P_i$ conditioned to the
set $Y_M$, by using Theorem \ref{thmII.6.1} and the triangular
inequality, we get \eqref{6.29pp}. \qed

\bigskip

We are thus left with the estimate of $R_{\Delta_1}(P_1,P_2)$ that
we do next.

\vskip1cm
            \begin{prop}
            \label{prop12}
There is $\eps_3>0$ such that  the following holds:
    \begin{equation}
      \label{6.53}
R_{\Delta_1}\Big(P_1(\cdot|\xi^*_1),P_2(\cdot| \xi^*_2)\Big)\le
\eps_3
     \end{equation}
 \end{prop}

\medskip

{\bf Proof.} We first observe that from the definition of the
Wasserstein distance
        \begin{equation}
        \label{6.60}
R_{\Delta_1}\Big(P_1(\cdot|\und b_1),P_2(\cdot| \und
b_2)\Big)=\inf_Q Q\big(r_{\Delta_1}\ne r'_{\Delta_1}\big)
        \end{equation}
where $r_{\Delta_1}$ is the restriction of $\und r$ to $\Delta_1$,
namely $r_{\Delta_1}\in \mathcal{Y}_{\Delta_1}:=\{r(x,s)\in
\mathbb{R}^d,x\in \Delta_1, s=1,\dots S\}$. Thus the inf on the
r.h.s. of \eqref{6.60} can be restricted to  all couplings of the
marginals $P_{i,\Delta_1}$ on the set $\mathcal{Y}_{\Delta_1}$ of
the probabilities $P_i$, $i=1,2$.

Recalling \eqref{a6.39} we define
    \begin{equation}
    \label{6.52bi}
 b_i= B_{\Delta_2}^{-1}\xi^*_i= \ell_-^{-d/2}B_{\Delta_2}^{-1}\big( B[
\bar n_{i,\La\setminus\Delta_2}-a^*]\big),\qquad i=1,2
    \end{equation}
We call $ b_{i,\Delta_1}$ the restriction of the vector $ b_i$ to the
set $\Delta_1$.

We next call $C$  the matrix with entries $C_{i,j}=(
B_{\Delta_2})^{-1} _{i,j}$, $i=(x,s)$, $j=(x',s')$,
$x,x'\in\Delta_2$, $s,s'\in\{1,\dots S\}\}$, $C_{\Delta_1}^{-1}$
denotes the
restriction to $\Delta_1$ of  $C^{-1}$. 

Then remark that marginals of Gaussian variables are Gaussian themselves, so we get:
    \begin{equation}
      \label{8.4}
 dP_{i,\Delta_1}(r_{\Delta_1})= \psi(r_{\Delta_1}-b_{i,\Delta_1})dr_{\Delta_1},
 \qquad  \psi(r_{\Delta_1}-b_{i,\Delta_1}) =Z_i^{-1} e^{-\frac{1}{2}(r_{\Delta_1}
 -b_{i,\Delta_1},
 C_{\Delta_1}^{-1}(r_{\Delta_1}-b_{i,\Delta_1}))}
     \end{equation}

We use that the Wasserstein distance is related to the
variational distance via the following relation
   \begin{equation}
      \label{8.7}
2R_{\Delta_1}\big(P_{1,\Delta_1},P_{2,\Delta_1}\big)
=\|P_{1,\Delta_1}-P_{2,\Delta_1}\|
     \end{equation}
where
   \begin{equation}
      \label{8.5}
\|P_{1,\Delta_1}-P_{2,\Delta_1}\|:= \int
|\psi(r_{\Delta_1}-b_{1,\Delta_1})-\psi(r_{\Delta_1}-b_{2,\Delta_1})|
dr_{\Delta_1}
     \end{equation}

\vskip.5cm

We now prove that

   \begin{equation}
      \label{8.8}
 \|P_{1,\Delta_1}-P_{2,\Delta_1}\|
\le 2\|C_{\Delta_1}^{-1}\| \|b_{1,\Delta_1}- b_{2,\Delta_1}\|_{L^2}
\Big(\sum_{i=(x,s), x\in \ell_-\mathbb{Z}^d\cap \Delta_1}
C_{ii}\Big)^{1/2}
     \end{equation}
To prove \eqref{8.8} we interpolate defining
$M(t)=tb_{1,\Delta_1}+(1-t)b_{2,\Delta_1}$, $t\in [0,1]$. Then,
shorthanding $M=M(t)$,
   \begin{equation}
      \label{8.9}
 {\rm l.h.s.\ of \eqref{8.8}} \le 2\int_0^1\int |\Big(b_{1,\Delta_1}-
b_{2,\Delta_1},
 C_{\Delta_1}^{-1}(r_{\Delta_1}-M)\Big)|\psi(r_{\Delta_1}-M)dr_{\Delta_1}\, dt
     \end{equation}
Using Cauchy-Schwartz the r.h.s.\ is bounded by
   \begin{equation}
      \label{8.10}
\le 2\|C_{\Delta_1}^{-1}\| \|b_{1,\Delta_1}-
b_{2,\Delta_1}\|_{L^2}\int_0^1\int \Big(\sum_{s,x\in \Delta_1}
 (r(x,s)-M(x,s))^2\Big)^{1/2} \psi(r_{\Delta_1}-M)dr_{\Delta_1}\, dt
     \end{equation}
hence \eqref{8.8}.

To estimate $\|b_{1,\Delta_1}- b_{2,\Delta_1}\|_{L^2}$, we apply
Theorem \ref{thmappB.2} with $C'=C''=I$, $I$ the identity
matrix, and with $A=B_{\Delta_2}$, observing  
that $B_{\Delta_2}(x,s,x's')=0$ whenever $|x-x'|> \ga^{-1}N$. Thus
from \eqref{appB.2} and \eqref{appB.6}, using that $\bar
\rho_{i,\La\setminus\Delta_2}\in A_{\le,i}$, $i=1,2$,  \eqref{a6.39}
and \eqref{6.37} we get that there are $c$ and $c'$, such that for all
$x\in\Delta_1$ and since dist$\dis{({\Delta_1,\Delta_2^c})>
10^{-30}\ell_+}$
    \begin{eqnarray*}
&& \hskip-.8cm
\big|b_{1,\Delta_1}(x,s)-b_{2,\Delta_1}(x,s)\big|=\big|
    \sum_{s',y\in\La\setminus\Delta_2}B_{\Delta_2}^{-1}(x,s,y,s')
B(\ell_-^{-d/2}\bar n_{1,\La\setminus\Delta_2}(y,s)-\bar
n_{2,\La\setminus\Delta_2}(y,s))\big|\\&& \hskip3.4cm\le \|B\|
\ell_-^{\delta}
 \sum_{s',y\in\La\setminus\Delta_2} e^{-c|x-y|\ga} \le
 c'\ell_-^{\delta} e^{-c\ga 10^{-30}\ell_+}
    \end{eqnarray*}
Thus this inequality together with \eqref{8.7} and \eqref{8.8}
implies \eqref{6.53}.
 \qed

 \vskip1cm

 {\bf Proof of Theorem \ref{thm:6.2}}.
Recalling the definition \eqref{6.28} of the probabilities $\mu_i$,
and the conditional probabilities defined in \eqref{6.36}, from
Propositions \ref{prop10}, \ref{prop11}, \ref{prop12} we get that
for all $\bar \rho_{i,\La\setminus\Delta_2}\in A_{\le,i}$, $i=1,2$,
        \begin{equation*}
R_{\Delta_1}\big(\mu_1(\cdot|\bar \rho_{1,\La\setminus\Delta_2}),
\mu_2(\cdot|\bar \rho_{2,\La\setminus\Delta_2})\big)\le 2c
\ga^{d/4} +2\ga^{\delta^*}+\eps_3=\eps_4
    \end{equation*}
Thus, there is a coupling \- $\hat
Q\big(n'_{\Delta_2},n''_{\Delta_2}|\bar
\rho_{1,\La\setminus\Delta_2},\bar
\rho_{2,\La\setminus\Delta_2}\big)$ of the conditional
probabilities $\mu_i(\cdot|\bar \rho_{i,\La\setminus\Delta_2})$,
$i=1,2$  such that
    \begin{equation}
    \label{a6.67}
\hat Q\big(n'_{\Delta_1}\ne
n''_{\Delta_1}|\bar\rho_{1,\La\setminus\Delta_2},\bar
\rho_{2,\La\setminus\Delta_2}\big)\le 2\eps_4
    \end{equation}
We define for all $\bar \rho_{i,\La\setminus\Delta_2}$
    \begin{equation*}
Q\big(n'_{\Delta_2},
n''_{\Delta_2}|\bar\rho_{1,\La\setminus\Delta_2},\bar\rho_{2,\La\setminus\Delta_2}\big) =
\begin{cases} \hat Q(n'_{\Delta_2}, n''_{\Delta_2}
|\bar\rho_{1,\La\setminus\Delta_2},\bar
\rho_{2,\La\setminus\Delta_2})&\hskip -2cm \text{if $\bar
\rho_{i,\La\setminus\Delta_2} \in A_{\le,i}, i=1,2 $}
\\dG^0_\La(n'_{\Delta_2}|\bar\rho_{1,\La\setminus\Delta_2},\bar
q_{1,\La^c}) dG^0_\La(n''_{\Delta_2}|\bar
\rho_{2,\La\setminus\Delta_2},\bar q_{2,\La^c})&\text{otherwise}
        \end{cases}
 \end{equation*}
 We then define a coupling $Q$ of the measures $\mu_i$ by
letting
    \begin{equation}
    \label{a6.68}
Q\big(n'_{\Delta_2}, n''_{\Delta_2})=Q\big(n'_{\Delta_2},
n''_{\Delta_2}|\bar\rho_{1,\La\setminus\Delta_2},\bar
\rho_{2,\La\setminus\Delta_2}\big)dG^0_\La\big(\bar\rho_{1,\La\setminus\Delta_2}|\bar
q_{1,\La^c}\big)dG^0_\La\big(\bar\rho_{2,\La\setminus\Delta_2}|\bar
q_{2,\La^c}\big)
    \end{equation}
From \eqref{a6.19}, \eqref{a6.67} and Theorem \ref{thmII.5.1} it
follows that
    \begin{equation}
    \label{a6.69}
Q\big(n'_{\Delta_1}\ne n''_{\Delta_1}\big)\le
2\eps_4+2c\ga^{\tau}+2e^{-c \ell_-^{2\delta}}=\eps_5
    \end{equation}
Observe that \eqref{a6.69} implies that
    \begin{equation}
    \label{a6.70}
R_{\Delta_1}(\mu_1,\mu_2)\le\eps_5
    \end{equation}
Then, \eqref{a6.70}, Propositions \ref{prop1}, \ref{prop2},
\ref{prop3} implies \eqref{a6.4}. \qed
 \vskip1cm

\subsection{Proof of Theorem \ref{thme3.7.1}}
        \label{sec:II.8}

\bigskip

We need to construct a coupling $Q_\La$ such that \eqref{aa6.1} holds.
%

Recall  $\mathring{\Delta}_{1}= \Delta_1\setminus\delta_{\rm
in}^{\ga^{-1}}[\Delta_1]$ and that for any two configurations $\bar
q_{i,\La\setminus\mathring{\Delta}_{1}}$, $i=1,2$ on $\mathcal
X^{(k)}_{\La\setminus\mathring{\Delta}_{1}}$ we denote by $\bar q_{i,\mathring{\Delta}_{1}^c}=\bar q_{i,\La\setminus\mathring{\Delta}_{1}}\cup \bar
q_{i,\La^c}$, $i=1,2$. From Theorem \ref{thm:6.1} we have that, for
any $ n_{\Delta_1}$, there is a coupling $Q_{\mathring{\Delta}_{1}}\big(q'_{\bar \Delta_1},q''_{\mathring{\Delta}_{1}}|\bar q_{1,\mathring{\Delta}_{1}^c},\bar q_{2,\bar \Delta_1^c}, n_{\Delta_1}\big)$ of the
two conditional Gibbs measures $dG^0_\La(q_{\mathring{\Delta}_{1}}|q_{i,\mathring{\Delta}_{1}^c}, n_{\Delta_1})$, $i=1,2$ such that
    \begin{equation}
    \label{6.64}
\sum_{x\in\ell_{-,\ga}\mathbb{Z}^d \cap \Delta_0} Q_{\mathring{\Delta}_{1}}\big(q'_\La\cap C^{(\ell_{-,\ga})}_x\ne
    q''_\La\cap C^{(\ell_{-,\ga})}_x|\bar q_{1,\mathring{\Delta}_{1}^c},\bar q_{2,\bar \Delta_1^c}, n_{\Delta_1}\big) \le
2\eps_0
    \end{equation}
Given $\und n'$ and $\und n''$, we define a coupling $ \hat
Q_{\mathring{\Delta}_{1}}\equiv \hat Q_{\mathring{\Delta}_{1}} \big(q'_{\mathring{\Delta}_{1}}
    q''_{\mathring{\Delta}_{1}}|\bar q_{1,\mathring{\Delta}_{1}^c},\bar q_{2,\mathring{\Delta}_{1}^c},
\und n',\und n''\big)$  of $dG^0_\La(\cdot|\bar q_{1,\mathring{\Delta}_{1}^c}, \und n')$, $dG^0_\La(\cdot|,\bar q_{2,\mathring{\Delta}_{1}^c},
\und n'')$, by setting
        \begin{equation*}
\hat Q_{\mathring{\Delta}_{1}}=
\begin{cases} Q_{\mathring{\Delta}_{1}} &\text{if
$n'_{\Delta_1}=n''_{\Delta_1} $}
\\dG^0_\La(\cdot|\bar q_{1,\mathring{\Delta}_{1}^c}, \und
n')dG^0_\La(\cdot|,\bar q_{2,\mathring{\Delta}_{1}^c}, \und
n'')&\text{otherwise}
        \end{cases}
        \end{equation*}
From Theorem \ref{thm:6.2} there is a coupling $Q^*$ of
$G^0_\La(n_{\La}|q_{i,\La^c})$, $i=1,2$ such that
    \begin{equation}
    \label{a6.78}
Q^*(n'_{\Delta_1}\ne n''_{\Delta_1})\le 2\eps_1
    \end{equation}
Then the final coupling $Q_\La$ is defined as follows:
    \begin{eqnarray}
    \nn
&& Q_\La(q'_{\La},q''_{\La})=  \hat Q_{\mathring{\Delta}_{1}} \big(q'_{\mathring{\Delta}_{1}}
    q''_{\mathring{\Delta}_{1}}|\bar q_{1,\mathring{\Delta}_{1}^c},\bar q_{2,\mathring{\Delta}_{1}^c},
\und n',\und n''\big)
  \\&&\hskip2cm dG^0_\La( q'_{\La\setminus\mathring{\Delta}_{1}}| \bar q_{1,\La^c},\und
 n') dG^0_\La( q''_{\La\setminus\mathring{\Delta}_{1}}| \bar q_{2,\La^c},\und
 n'')Q^*(n',n'')
 \label{6.69}
    \end{eqnarray}

Thus from \eqref{6.64} and \eqref{a6.78}  we get
    \begin{equation}
    \label{6.65}
\sum_{x\in\ell_{-,\ga}\mathbb{Z}^d \cap \Delta_0} Q_\La\big(q'_\La\cap
C^{(\ell_{-,\ga})}_x\ne
    q''_\La\cap C^{(\ell_{-,\ga})}_x\big)\le \eps_6
    \end{equation}
To complete the proof of  \eqref{aa6.1} we need to show that
    \begin{equation}
    \label{a6.66}
    \sum_{ s=1}^S \sum_{x\in\Delta_0}  Q\big(q'_\La\cap
C^{(\ell_{-,\ga})}_x=
    q''_\La\cap C^{(\ell_{-,\ga})}_x,|\rho^{(\ell_{-,\ga})}({
q}'_{\La};x,s)-\rho^{(k)}_s|
 > \zeta_{K(\cdot;x)-1}\big)\le \eps
     \end{equation}
Since in the set on the l.h.s. of \eqref{a6.66}, $q'_\La=q''_\La$,
by using \eqref{a6.19} we have
    \begin{align}
   \nn
 Q_\La\big(q'_\La\cap C^{(\ell_{-,\ga})}_x & =
    q''_\La\cap C^{(\ell_{-,\ga})}_x,\,|\rho^{(\ell_{-,\ga})}({
q}'_{\La};x,s)-\rho^{(k)}_s|
 > \zeta_{K(\cdot;x)-1}\big)
\\
\nn
& \le G_{\La}^0\big(|\rho^{(\ell_{-,\ga})}({
q}'_{\La};x,s)-\rho^{(k)}_s > \zeta_{K(\cdot;x)-1} ; \bar q'_{\La}\big)
\\
\nn
  & \quad +   G_{\La}^0\big(|\rho^{(\ell_{-,\ga})}({
q}''_{\La};x,s)-\rho^{(k)}_s| > \zeta_{K(\cdot;x)-1} ; \bar q''_{\La}\big)
\\
\nn
& \le G_\La^*(|\rho^{(\ell_{-,\ga})}({ q}'_{\La};x,s)-\rho^{(k)}_s|
 > \zeta_{K(\cdot;x)-1}|\bar q'_{\La^c})
\\
& \quad + G_\La^*(|\rho^{(\ell_{-,\ga})}({ q}''_{\La};x,s)-\rho^{(k)}_s|
 > \zeta_{K(\cdot;x)-1}|\bar q''_{\La^c})
+2c\ga^{\tau}
     \label{6.67}
     \end{align}
From Theorem \ref{thmII.5.1} and (ii) of Theorem \ref{thmee5.0} it
follows that for all $x\in\Delta_0$ and for $\bar q_{\La^c}=\bar q'_{\La^c}$ or $\bar q''_{\La^c}) $,
    \begin{equation}
    \label{6.68}
G_\La^*(|\rho^{(\ell_{-,\ga})}({ q}'_{\La};x,s)-\rho^{(k)}_s|
 > \zeta_{K(\cdot;x)-1}|\bar q_{\La^c})\le e^{-c \ell_-^{2\delta}}
     \end{equation}
which together with \eqref{6.67} proves Theorem \ref{thme3.7.1}.\qed

 \vskip1cm

\part{Disagreement percolation}

\vskip.5cm

\nopagebreak

In this part we fix $t\in[0,1]$, a bounded $\mathcal
D^{\ell_{+,\ga}}$-measurable region $\La$, $k\in
\{1,\dots,S+1\}$;  $\mu'$ and $\mu''$ stand for the
measures $dG_\La(q_\La, \und \Ga|{\bar
q'}_{\La^c},\bar{\und \Ga'}_{\La^c})$ and $dG_\La(q_\La,
\und \Ga|{\bar q''}_{\La^c},\bar{\und \Ga''}_{\La^c})$.
They are obtained by conditioning measures $\nu'$ and
$\nu''$ which could be either DLR measures or   Gibbs
measures $dG_{\La'}(q_{\La'}, \und \Ga|{\bar
q}_{(\La')^c})$ with $\La'\supseteq \La$.  We will first
construct a coupling of $\mu'$ and $\mu''$ and, with the
help of such a coupling, we will then define a coupling of
$\nu'$ and $\nu''$ proving that it  satisfies the
requirements of Theorem \ref{thme3.6.1}.  The notation
which are most used in this part are reported below.

\vskip.5cm

\centerline{{\it Main notation and definitions.}}
 \nopagebreak
We call
    \begin{equation}
\hskip-2cm\xi=(q,\und \Ga)\in \mathcal X^{(k)}_{\La}\times
\mathcal B_\La
    \label{7.1.1}
        \end{equation}
Given a $\mathcal D^{(\ell_{+,\ga})}$ measurable subset
$\Delta$ of $\La$ and $\xi=(q,\und\Ga)$, we call
$\xi_\Delta=(q_\Delta,\und\Ga_\Delta)$ its restriction to
$\Delta$. Namely if $\und\Ga=(\Ga(1),\dots\Ga(n))$, then
    \begin{equation}
\hskip-2cm \Ga_\Delta(i)=\big(\rm{sp}[\Ga(i)]\cap\Delta,
\eta_{\rm{sp}[\Ga(i)]\cap\Delta}\big),\qquad \und
\Ga_{\Delta}=( \Ga_\Delta(1),\dots \Ga_\Delta(n))
    \label{7.1.2}
        \end{equation}
We will say that we vary $\xi$ in $\Delta^c$ if we change
$\xi$ leaving $\xi_\Delta$ invariant.

We denote by $\Omega_\La$  the product space,
    \begin{equation}
\hskip-2cm \Omega_\La=(\mathcal X^{(k)}_{\La}\times
\mathcal B_\La)^2,  \hskip 2cm \omega=(\xi,\xi')\in
\Omega_\La
    \label{7.1.3}
        \end{equation}
Given a subset $\Delta\subset \La$ and
$\omega=(\xi,\xi')\in \Omega_\La$, we call
$\omega_\Delta=(\xi_\Delta, \xi'_\Delta)\in\Omega_{\Delta}$
its restriction to $\Delta$.

We call $\mathcal{F}_\La$ the $\si$-algebra of all Borel
sets  in $\Omega_\La$  and for any $\mathcal
D^{(\ell_{+,\ga})}$ measurable set $\Delta$ in $\La$ we
call $\mathcal F_\Delta$ the $\si$-algebra of all Borel
sets $A$ such that $\text{\bf 1}_{A}(\om)$ does not vary
when we change $\om$ in $\Delta^c$.

%
%

 \vskip1cm

\section{Construction of the coupling}
    \label{sec:7}
The target of this section is to construct a ``good''
coupling $Q$ of $\mu'$ and $\mu''$. The basic idea is to
implement the disagreement percolation technique used in
van der Berg and Maes, \cite{vm}, Butta et al., \cite{BMP},
Lebowitz et al,\cite{LMP}. The first step is to introduce a
sequence of random sets $\La_n$, which is done in the next
subsection. We will  then introduce the notion of
``stopping sets''  and ``strong Markov couplings'' showing
that the sets $\La_n$ are indeed stopping sets and, using
the strong Markov coupling property, we will finally get
the desired coupling of $\mu'$ and $\mu''$.

\vskip2cm

\subsection{The sequence $\La_n$}
\label{subsec:7.2}

We will define here for each $\omega=(\xi',\xi'')\in \Om_\La$ a
decreasing sequence of $\mathcal D^{(\ell_{+,\ga})}$-measurable sets
$\La_n$, which are therefore  set valued  random variables.  We set
$\La_0=\La$ and for $n\ge 0$, define
$\La_{n+1}=\La_n\setminus\Si_{n+1}$, thus the sequence is defined
once we specify the ``screening sets'' $\Si_n$. Screening sets are
defined iteratively with the help of the notion of ``good'' and
``bad cubes''.

After defining in an arbitrary fashion an order among the
$\mathcal D^{(\ell_{+,\ga})}$ cubes of $\delta_{\rm
out}^{\ell_{+,\ga}}[\Delta]$, for any $\mathcal
D^{(\ell_{+,\ga})}$-measurable set $\Delta \subset \La$, we
start the definition by calling bad all the cubes of
$\delta_{\rm out}^{\ell_{+,\ga}}[\La_0]$. We then select
among these the first one (according to the pre-definite
order) which intersects a polymer (i.e.\ either $\ssp(\und
\Ga')\cap C\ne\emptyset$, or $\ssp(\und \Ga'')\cap
C\ne\emptyset$), if there is no such cube  we  then take
the first cube in $\delta_{\rm out}^{\ell_{+,\ga}}[\La_0]$.
Call $C_1$ the cube selected with such a rule. We
then define $\Si_1=\delta_{\rm out}^{\ell_{+,\ga}}[C_1]\cap
\La_0$ and call bad all cubes of $\Si_1$ if $C_1$ intersects a polymer.
If not, we
say that a $\mathcal D^{(\ell_{+,\ga})}$ cube
$C\in \Si_1$ is good if $\ssp(\und \Ga')\cap C=\ssp(\und
\Ga'')\cap C=\emptyset$ and if
    \begin{equation}
    \label{7.2.1}
\om\in \bigcap_{x\in \ell_{-\ga}\mathbb Z^d\cap
C }
\Theta_{\La_0}(x), \quad \text{$\Theta_{\La_0}$ has been
defined in \eqref{e3.7.1.0},}
    \end{equation}
otherwise $C\in \Si_1$ is called bad.  In this way
each cube of $\Si_1$ is classified as good or bad and
therefore all cubes of $\delta_{\rm
out}^{\ell_{+,\ga}}[\La_1]$ are classified as good or bad.
We then select $C_2$ in $\delta_{\rm
out}^{\ell_{+,\ga}}[\La_1]$ in the same way we had selected
$C_1$ in $\delta_{\rm out}^{\ell_{+,\ga}}[\La_0]$,
$\Si_2=\delta_{\rm out}^{\ell_{+,\ga}}[C_2]\cap \La_1$  and
the cubes of $\Si_2$ are then  classified as good or bad by
the same rule used for those of $\Si_1$. By iteration we
then define a sequence which becomes eventually constant,
as it stops changing at $\La_n$ if $\delta_{\rm
out}^{\ell_{+,\ga}}[\La_n]$ has no bad cube or if $\La_n$
is empty. Since $\La$ has $N^*:=|\La|/\ell_{+,\ga}^d$
cubes, $\La_n$ is certainly constant after $N^*$, but maybe
even earlier. In Appendix \ref{appCC} we will prove:

 \vskip1cm

    \begin{thm}
   \label{thm7.2.1 }
If the sequence $\{\La_n\}$ stops at $n=N$ and $\La_N$ is
non empty, then
   \begin{equation}
    \label{7.2.2}
q'_\La\cap \delta_{\rm out}^{\ga^{-1}}[\La_N]=q''_\La\cap
\delta_{\rm out}^{\ga^{-1}}[\La_N]
    \end{equation}
    and
   \begin{equation}
    \label{7.2.3}
 \ssp(\und
\Ga')\cap \delta_{\rm out}^{\ga^{-1}}[\La_N]=\ssp(\und
\Ga'')\cap  \delta_{\rm out}^{\ga^{-1}}[\La_N]=\emptyset
    \end{equation}

    \end{thm}

\vskip2cm

\subsection{Stopping sets}
        \label{subsec:7.3}

The random variables $\La_n$ are ``stopping sets'' and the
sequence $\La_n$ is decreasing, $\La_{n+1}\preccurlyeq
\La_n$, in the following sense.

\vskip.5cm

\begin{itemize}

\item
$\mathcal F_{\Delta^c}$, $\Delta$ a $\mathcal
D^{(\ell_{+,\ga})}$ measurable subset of $\La$, is the
$\si$ algebra of all Borel sets $A$ such that $\text{\bf
1}_{A}(\om)$ does not change if we vary $\om$ in $\Delta$.

\item
A  random variable $\mathcal{R}$ with values in the
$\mathcal D^{(\ell_{+,\ga})}$ measurable subsets of $\La$
is called a stopping  set if for all $\Delta$,
    \begin{equation}
    \label{7.3.1}
\{\omega\in\Omega:\mathcal{R}(\om)=\Delta\}\in \mathcal
F_{\Delta^c}
    \end{equation}

\item
Two stopping sets $\mathcal R'$ and $\mathcal R$
 are such that $\mathcal R'\preccurlyeq
\mathcal{R}$ if
    \begin{eqnarray*}
&&  \mathcal R'(\omega)\subset \mathcal R(\omega),
\;\;\text{ for all $\omega\in \Omega_\La$}
\\&&\{\omega: \mathcal R'(\omega)=\Delta'\}\cap \{\omega: \mathcal R(\omega)=\Delta\}\in
\mathcal F_{\Delta^c},\quad\text{for all $\Delta'\subset
\Delta$}
   \end{eqnarray*}

\end{itemize}

\vskip2cm

\subsection{Strong Markov couplings}
\label{subec:7.4}

A coupling $Q(d\omega)$ of $\mu'$ and $\mu''$ is called
strong Markov in  $\mathcal R$,  $\mathcal R$ a stopping
set, if the measure
    \begin{equation}
d\tilde Q(\omega):=\sum_{\Delta\subset\La} {\bf 1}_{\{
\mathcal R(\omega)=\Delta\}}d\pi_\Delta(\omega_\Delta|\bar
\omega_{\Delta^c})dQ({\bar\omega}_{\Delta^c})
     \label{7.4.1}
        \end{equation}
is also a coupling of $\mu'$ and $\mu''$ for all couplings
$d\pi_\Delta(\omega_\Delta|\bar \omega_{\Delta^c})$  of
$d\mu'(\xi_\Delta|\bar \xi_\Delta)$, and
$d\mu''(\xi'_\Delta|\bar \xi'_\Delta)$.

\vskip2cm
        \begin{thm}
        \label{thm:7.4.1}
Given any stopping set $\mathcal{R}$, let $Q$ be a coupling
of $\mu'$ and $\mu''$ which is strong Markov in
$\mathcal{R}$, Then any coupling $\tilde Q$ defined by
\eqref{7.4.1} is strong Markov in $\mathcal{R'}$ provided
the stopping set $\mathcal{R'}$ is such that $\mathcal
R'\preccurlyeq \mathcal{R}$,

    \end{thm}

    \bigskip

{\bf Proof.} We have to prove that for any family of
couplings $\{\hat \pi_\Delta(d\omega_\Delta|\bar
\omega_{\Delta^c}), \Delta\subset \La, \bar
\omega_{\Delta^c}\in \Omega_{\Delta^c} \}$,
 the probability  $\hat
Q(d\omega)$  defined as
    \begin{equation}
d\hat Q(\omega):=\sum_{A\subset\La} {\bf 1}_{\{ \mathcal
R'(\omega)=A\}}d\hat \pi_\Delta(\omega_A|\bar
\omega_{A^c})d\tilde Q(\omega_{A^c})
     \label{7.4.2}
        \end{equation}
is a coupling of $\mu'$ and $\mu''$. We thus take a
function $f(\xi)$ and we prove that $\hat Q(f)=\mu'(f)$,
where $\hat Q(f)$, $\mu'(f)$, is the expectation of $f$
under $Q$, respectively $\mu'$.

Using that $\mathcal{R'}$ is a stopping set we get
    \begin{eqnarray}
    \nn
\hat Q(f)&=&\sum_{A\subset\La}\int_{\Omega_{A^c}} {\bf
1}_{\{ \mathcal R'(\omega)=A\}}d\tilde
Q(\omega_{A^c})\int_{\Omega_A}f(\xi)d\hat
\pi_A(\omega_A|\bar \omega_{A^c})
\\&=&\nn\sum_{A\subset\La}\int_{\Omega_{A^c}} {\bf 1}_{\{
\mathcal R'(\omega)=A\}}d\tilde
Q(\omega_{A^c})\mu'(f|\xi_{A^c})
\\&=&\nn\sum_{A\subset\La}\int_{\Omega} {\bf
1}_{\{ \mathcal R'(\omega)=A\}}d\tilde Q(\omega)
\mu'(f|\xi_{A^c})
    \label{7.4.3}
        \end{eqnarray}
We now rewrite $d\tilde Q(\omega)$ by using its definition
\eqref{7.4.1} and since $\mathcal R'\preccurlyeq
\mathcal{R}$ we get
    \begin{equation}
\hat
Q(f)=\sum_{\Delta\subset\La}\sum_{A\subset\Delta}\int_{\Omega_{\Delta^c}}
{\bf 1}_{\{ \mathcal R(\omega)=\Delta\}}{\bf 1}_{\{
\mathcal R'(\omega)=A\}}
dQ(\bar\omega_{\Delta^c})\int_{\Omega_\Delta}d\pi_\Delta(\omega_\Delta|\bar
\omega_{\Delta^c})\mu'(f|\xi_{A^c})
    \label{7.4.4}
        \end{equation}
Observe that (recalling $A\subset \Delta$)
    \begin{equation}
\int_{\Omega_\Delta}d\pi_\Delta(\omega_\Delta|\bar
\omega_{\Delta^c})\mu'(f|\xi_{A^c})=\int
d\mu'(\xi_\Delta|\xi_{\Delta^c})\mu'(f|\xi_{\Delta^c},\xi_{\Delta\setminus
A}) =\mu'(f|\xi_{\Delta^c})
    \label{7.4.5}
        \end{equation}
We insert \eqref{7.4.5} in \eqref{7.4.4} and we get
     \begin{equation*}
\hat Q(f)=\sum_{\Delta\subset\La}\int_{\Omega_{\Delta^c}}
{\bf 1}_{\{ \mathcal R(\omega)=\Delta\}}
dQ(\omega_{\Delta^c})\mu'(f|\xi_{\Delta^c})=\mu'(f)
        \end{equation*}
The Theorem is proved.\qed

\vskip2cm

\subsection{Construction of couplings}
\label{subsec:7.5}

We use the sequence $\{\La_n\}$ of decreasing  stopping
sets (in the order $\preccurlyeq$) and Theorem
\ref{thm:7.4.1} to construct a sequence $\{Q^n\}$  of
couplings of $\mu'$ and $\mu''$, the desired coupling will
then be $Q^{N^*}$, where $N^*= |\La|/\ell_{+,\ga}^d$. The
sequence $\{Q^n\}$ is defined iteratively by setting
$Q^{0}$ equal to the product coupling: $Q^{0} = \mu'\times
\mu''$ which, as it can be easily checked, is strong Markov
in $\La_0$.  Then for any $n\ge 0$ we set
        \begin{equation}
        \label{7.17}
dQ^{n+1}( \om_\La ) = \sum_{\Delta\ne\emptyset} {\bf
1}_{\left\{\La_{n}(\om_\La) = \Delta \right\}}d
\pi_{\Delta}( \om_{\Delta} \vert \om_{\La\setminus \Delta},
\bar \om_{\La^c}  ) dQ^{n} (\om_{\La\setminus\Delta})+{\bf
1}_{\left\{\La_{n}(\om_\La) = \emptyset \right\}}dQ^{n}
(\om_{\La})
        \end{equation}
where $dQ^{n}(\om_{\Delta^{c}})$ is the marginal of $dQ^n$
over $\{\om_{\Delta^{c}}\}$ and
$\pi_{\Delta}$,   $\Delta\ne\emptyset$, is the coupling of
$d\mu'(\xi'_\Delta|  {\bar \xi}'_{\Delta^c})$, and
$d\mu''(\xi''_\Delta| {\bar \xi}''_{\Delta^c})$  defined
next.
 We distinguish three cases according to the values
of $\bar\om_{\Delta^{c}}=({\bar \xi}'_{\Delta^c},{\bar
\xi}''_{\Delta^c})$.

\vskip.5cm

\begin{itemize}

  \item
If  $\bar\om_{\Delta^{c}}$ is such that either $\ssp(\und\Ga')\cap
\delta_{\rm{out}}^{\ell_{+,\ga}}[\Delta]\ne\emptyset$, or
$\ssp(\und\Ga'')\cap
\delta_{\rm{out}}^{\ell_{+,\ga}}[\Delta]\ne\emptyset$, or both, then
$\pi_\Delta$ is the  product coupling: $d\pi_{\Delta}(
\xi'_{\Delta},\xi''_{\Delta} \vert \bar\om_{\Delta^{c}} )= d\mu'(
\xi'_{\Delta} \vert {\bar \xi}'_{\Delta^{c}} )d\mu'( \xi''_{\Delta}
\vert {\bar \xi}''_{\Delta^{c}} )$.


\vskip .5cm \noindent

  \item
If $\bar\om_{\Delta^{c}}$ is such that $\ssp(\und \Ga')\cap
\delta_{\rm{out}}^{\ell_{+,\ga}}[\Delta]= \ssp(\und
\Ga'')\cap
\delta_{\rm{out}}^{\ell_{+,\ga}}[\Delta]=\emptyset$ and
$q'\cap\delta_{\rm{out}}^{\ga^{-1}}[\Delta]=q''\cap\delta_{\rm{out}}^{\ga^{-1}}[\Delta]$
then $ d \pi_{\Delta}( \xi'_{\Delta},\xi''_{\Delta} \vert
\bar\om_{\Delta^{c}} )= d\mu'( \xi'_{\Delta} \vert {\bar
\xi}'_{\Delta^{c}} ) \delta(\xi'_\Delta-\xi''_\Delta)
d\xi''_\Delta$, namely $ d \pi_{\Delta}$ is the coupling
supported by the diagonal.
%
%
%
%
%

\vskip .5cm \noindent

  \item
Finally let $\bar\om_{\Delta^{c}}$ be such that $\ssp(\und\Ga')\cap
\delta_{\rm{out}}^{\ell_{+,\ga}}[\Delta]= \ssp(\und\Ga'')\cap
\delta_{\rm{out}}^{\ell_{+,\ga}}[\Delta]=\emptyset$ but
$q'\cap\delta_{\rm{out}}^{\ga^{-1}}[\Delta]\ne
q''\cap\delta_{\rm{out}}^{\ga^{-1}}[\Delta]$. Call $T =
\Sigma_{n+1}\cup \left( \delta^{\ell_{+,\ga}}_{\rm{out}}
[\Sigma_{n+1}] \cap \Delta\right)$, $U = \Delta \setminus T$. Let
$dP(q'_U,q''_U,\und \Ga',\und\Ga'')=d\mu'(q'_U, \und \Ga'\vert \bar
\xi'_{\Delta^c})d\mu''(q''_U,\und\Ga''\vert \bar \xi''_{\Delta^c})$
be the product of the marginal distributions of $d\mu'(\cdot\vert
\bar \xi'_{\Delta^c})$ and $d\mu''(\cdot\vert \bar
\xi''_{\Delta^c})$ over $\mathcal X^{(k)}_U\times\mathcal B_\Delta$.
Let $Q_T$ be the coupling defined in Theorem \ref{thme3.7.1} and
letting
  $\Xi=\{\om_{\Delta^{c}}: \und\Gamma' \cap (T\cup
\delta_{\rm{out}}^{\ell_{+,\ga}}[T])=\und\Gamma''  \cap (T\cup
\delta_{\rm{out}}^{\ell_{+,\ga}}[T])=\emptyset \}$, we denote by
${\bf 1}_{\Xi}$ the characteristic function of the set $\Xi$.

Then we define
       \begin{eqnarray*}
&& \hskip-1cm
d\pi_{\Delta}( \om _\Delta\, \vert \,\bar
\om_{\Delta^{c}})
 =  {\bf 1}_{\Xi}(\bar \om_{\Delta^{c}})\,
dQ_{T} \left( q'_T,q''_T  \vert q'_U, {\bar
    q}'_{\Delta^c},q''_U,
{\bar q}''_{\Delta^c} \right) dP(q'_U,q''_U,\und \Ga',\und\Ga'')
\\&& \hskip2cm+[1- {\bf 1}_{\Xi}(\bar \om_{\Delta^{c}})]
d\mu'(q'_\Delta, \und \Ga'\vert \bar
\xi'_{\Delta^c})d\mu''(q''_\Delta,\und\Ga''\vert \bar
\xi''_{\Delta^c})%
\end{eqnarray*}

\end{itemize}

\vskip1cm

By Theorem \ref{thm7.2.1 } the second case above occurs if
and only if all cubes of
$\delta_{\rm{out}}^{\ell_{+,\ga}}[\Delta]$ are good, while in the
third case there are bad cubes in $\delta_{\rm{out}}^{\ell_{+,\ga}}[\Delta]$
so that $\Si_{n+1}$ is non empty. The proof that cubes are good
with large probability will be based on Theorem
\ref{thme3.7.1} and the following lemma:

\vskip.5cm

  \begin{lemma}
Suppose $\La_n(\om)=\Delta$ and that the third case above
is verified, namely $\om_{\Delta^{c}}$ is such that
$\ssp(\und\Ga')\cap
\delta_{\rm{out}}^{\ell_{+,\ga}}[\Delta]=
\ssp(\und\Ga'')\cap
\delta_{\rm{out}}^{\ell_{+,\ga}}[\Delta]=\emptyset$ and
$q'\cap\delta_{\rm{out}}^{\ga^{-1}}[\Delta]\ne
q''\cap\delta_{\rm{out}}^{\ga^{-1}}[\Delta]$.  Suppose also
that $\und\Gamma'  \cap (T\cup
\delta_{\rm{out}}^{\ell_{+,\ga}}[T])=\und\Gamma''  \cap
(T\cup \delta_{\rm{out}}^{\ell_{+,\ga}}[T])=\emptyset $.
Let $C$ in $\Si_{n+1}$,  then $C$ is good if $\om_\Delta\in
\Theta_T(x)$ for all $x\in C$, $\Theta_T$ as in
\eqref{e3.7.1.0}.

  \end{lemma}

\vskip.5cm

{\bf Proof.} The proof follows from the definitions of good
cubes and $\Theta_T(x)$ because for all $x\in C$,
$\Theta_T(x)=\Theta_\Delta(x)$.  \qed


\vskip2cm

\section{Probability estimates.}
    \label{sec:8}
Recall from the beginning of Part III that $\mu'$ and
$\mu''$  are obtained by conditioning to the configurations
outside $\La$ the measures $\nu'$ and $\nu''$ which are
either DLR measures or Gibbs measures $dG_{\La'}(q_{\La'},
\und \Ga|{\bar q}_{(\La')^c})$ with $\La'\supseteq \La$.
Thus if $Q^{N^*}$ is the coupling of $\mu'$ and $\mu''$
defined in Subsection \ref{subsec:7.5}, we obtain a
coupling $P$ of $\nu',\nu''$ by writing
   \begin{equation}
      \label{8.1}
 d P(\om)= d\nu'(\bar\xi'_{\La^c})d\nu''(\bar\xi''_{\La^c})
 dQ^{N^*}(\om_\La|\bar \om_{\La^c}),\quad \om=(\om_\La,\bar \om_{\La^c}),\;
 \bar \om_{\La^c}=(\bar\xi'_{\La^c},\bar\xi''_{\La^c})
     \end{equation}
We will prove here that there is a constant $c$ such that for
all $\ga$ small enough, for any $\mathcal
D^{(\ell_{+,\ga})}$-measurable subset $\Delta$ of $\La$:
    \begin{equation}
      \label{8.2}
  P\Big(\{\omega:\La_{N^*}(\om) \supset \Delta\}\Big) \ge 1-
  c_1 e^{- c_2 \frac {{\rm
  dist}(\Delta,\La^c)}{\ell_{+,\ga}}}
     \end{equation}
This proves that $(q'_\La,\und \Ga')$ and $(q''_{\La'},\und
\Ga'')$  agree in $\Delta$, in the sense of \eqref{e3.6.0},
with probability $ \ge 1-   c_1 e^{- c_2 \frac {{\rm
  dist}(\Delta,\La^c)}{\ell_{+,\ga}}}$ from which Theorem
\ref{thme3.6.1} follows. Indeed if $\nu'$ and $\nu''$ are two DLR
measures, by the arbitrariness of $\Delta$ and $\La$,
\eqref{8.2} shows that  $\nu'=\nu''$, hence that there is a
unique DLR measure. If instead  $\nu'$ and $\nu''$ are two
Gibbs measures $dG_{\La'}(q_{\La'}, \und \Ga|{\bar
q}_{{\La'}^c})$ and $dG_{\La''}(q_{\La''}, \und \Ga|{\bar
q}_{{\La''}^c})$, $\La \subset \La'$, $\La\subset \La''$
then \eqref{8.2} yields \eqref{e3.6.1}.

\vskip2cm

\subsection{Reduction to a  percolation event}
\label{subsec:8.1}

Denote  by $\mathcal A=\mathcal A(\om)$ the  union of all bad cubes
contained in $\La$ and of the cubes in
$\delta_{\rm out}^{\ell_{+,\ga}}[\La]$ with a polymer, namely
those cubes $C$  such that  $C\subseteq {\rm sp}(\Ga)$,
$\Ga$ in $\und\Ga'\cup\und\Ga''$.
Since by its definition any screening set is
connected to a bad cube
and since any bad cube in $\La$ is necessarily
contained in a screening set, it follows
that if $\mathcal A\ne \emptyset$ then it is connected to $\La^c$.

Since the event in \eqref{8.2} is bounded by
    \begin{equation}
      \label{8.1.1}
 \{\omega:\La_{N^*}(\om) \supset \Delta\}^c
 \subset \{\mathcal A(\om)\cap \Delta \ne \emptyset\}
     \end{equation}
, it is therefore also bounded by the event that the bad cubes percolate from $\Delta$ to $\La^c$. Hence, denoting in the sequel by $A$ a connected, $\mathcal
D^{(\ell_{+,\ga})}$-measurable subset of $\La\cup \delta_{\rm
out}^{\ell_{+,\ga}}[\La]$,
    \begin{equation}
      \label{8.1.2}
  P\Big(\{\La_{N^*}\supset \Delta\}^c\Big) \le
  \sum_{x\in \ell_{+,\ga}\mathbb Z^d\cap \Delta}\sum_{A: A\ni x, A\cap  \delta_{\rm out}^{\ell_{+,\ga}}[\La]\ne \emptyset}
   P\big(\{\mathcal A=A\}\big)
     \end{equation}
We write $\mathcal A=\mathcal A_1\cup
\mathcal A_2\cup\mathcal A_3$, $\mathcal A_i$ the
union of cubes of ``type $i$''.  Cubes of type
1 are those with a polymer, namely $C$ is type 1
if there is $\Ga$ in $\und\Ga'\cup\und\Ga''$
such that $C\subseteq {\rm sp}(\Ga)$.  $C$ is
type 2 (also called unsuccessful) if $C$,
say in $\Si_{n+1}$, is
bad and all cubes of
$\delta_{\rm out}^{\ell_{+,\ga}}[\La_n]$ are
without polymers (in the above sense).
Cubes of type 3 are the remaining ones,
they are therefore in the union of all
$\Si_{n+1}$ with  $\Si_{n+1}$ connected
to a type 1 bad cube. Then calling $N_A= |A|/\ell_{+,\ga}^d$,
    \begin{equation}
      \label{8.1.3}
\text{l.h.s.\ of \eqref{8.1.2}}
 \le  \sum_{x\in \ell_{+,\ga}\mathbb Z^d\cap \Delta}\sum_{A: A\ni x, A\cap  \delta_{\rm out}^{\ell_{+,\ga}}[\La]\ne \emptyset} 3^{N_A} \max_{A_1\cup A_2\cup A_3=A}
  P\Big(\bigcap_{i=1}^3\{\mathcal A_i=A_i\}\Big)
     \end{equation}
Since $\dis{
\mathcal A_3 \subset
\bigcup_{C \in \mathcal A_1} \delta_{\rm out}^{\ell_{+,\ga}}[C]
}$,
     \begin{equation}
      \label{8.1.4}
N_{\mathcal  A_3} \le 3^{d} N_{\mathcal A_1}
     \end{equation}
Therefore  $N_{\mathcal A_1}+N_{\mathcal A_2}+3^{d}
N_{\mathcal A_1}\ge N_{\mathcal A}$ and
     \begin{equation}
      \label{8.1.5}
\bigcap_{i=1}^3\Big\{\mathcal A_i=A_i\Big\} \;\;\subset\;\; \Big\{
\mathcal A_2=A_2; N_{\mathcal A_2} \ge \frac{N_A}2\Big\} \,\cup \,
\Big\{\mathcal A_1=A_1; N_{\mathcal A_1} \ge \frac{N_A}{2(1+3^d)}\Big\}
     \end{equation}
We are thus reduced to estimate for any $(A_1, A_2, A_3)$,
     \begin{equation}
      \label{8.1.6}
P\big( \{
\mathcal A_2=A_2\}\big), \;\text{if $N_{A_2} \ge \frac{N_A}2$};\quad
P\big( \{
\mathcal A_1=A_1\}\big), \;\text{if $ N_{ A_1} \ge \frac{N_A}{2(1+3^d)}$}
     \end{equation}

\vskip2cm

\subsection{Peierls estimates}
\label{subsec:8.2}
We bound here $P\big( \{
\mathcal A_1=A_1\}\big)$ where $A_1$ is some given set in $\La\cup
\delta_{\rm out}^{\ell_{+,\ga}}[\La]$.  Thus each cube $C\subset A_1$
is either contained in sp$(\Ga)$, $\Ga\in \und \Ga'$ or
in sp$(\Ga)$, $\Ga\in \und \Ga''$ (or both).  Thus
     \begin{equation}
      \label{8.2.1}
P\big( \{
\mathcal A_1=A_1\}\big) \le 2^{N_{A_1}} \max_{B\subset A_1,
N_{B}\ge N_{A_1}/2}
\max\{ \nu'({\rm sp}(\und \Ga) \supset B);\nu''({\rm sp}(\und \Ga )\supset B)\}
     \end{equation}
where $\dis{{\rm sp}(\und \Ga )= \bigcup_{\Ga\in \und \Ga} {\rm sp}( \Ga )}$.
Let $B= C_1\cup\cdots \cup C_n$,
$C_i$ disjoint cubes of $\mathcal D^{(\ell_{+,\ga})}$,
then, since $\nu'$ and $\nu''$ satisfy the Peierls estimates,
     \begin{eqnarray}
      \label{8.2.2}
&& \hskip-1cm \nu'({\rm sp}(\und \Ga) \supset B) \le \sum_{\Ga_1,...,\Ga_n, {\rm sp}(\Ga_i)\supset
 C_i}  \nu'\big(\und \Ga\ni \Ga_1,...,\Ga_n\big) \le
 \sum_{\Ga_1,...,\Ga_n, {\rm sp}(\Ga_i)\supset
 C_i} e^{-c_{\rm pol}\zeta^2 \ell_{-,\ga}^d (N_{\Ga_1}+\cdots+N_{\Ga_n})} \nn\\&&
 \hskip1cm
 \le e^{-c_{\rm pol}\zeta^2 \ell_{-,\ga}^d N_{B}/2} \Big(\sum_{\Ga: {\rm sp}(\Ga)\ni C}
 e^{-c_{\rm pol}\zeta^2 \ell_{-,\ga}^d N_{\Ga}/2}\Big)^{N_B} \le 2^{N_B}
 e^{-c_{\rm pol}\zeta^2 \ell_{-,\ga}^d N_{B}/2}
     \end{eqnarray}
 for all $\ga$ small enough.  Thus
      \begin{equation}
      \label{8.2.3}
P\big( \{
\mathcal A_1=A_1\}\big) \le 2^{2N_{A_1}}
e^{-c_{\rm pol}\zeta^2 \ell_{-,\ga}^d N_{A_1}/4}
     \end{equation}

\vskip2cm

\subsection{Probability of unsuccessful cubes}
\label{subsec:8.3}
We will bound here $P\big( \{
\mathcal A_2=A_2\}\big)$.  Given any $n>0$ we define
      \begin{equation}
      \label{8.3.1}
\mathcal A_{2,n}(\om) = \mathcal A_2(\om) \cap\La_n(\om)^c, 
\quad O_n(\om)= N_{\La_n(\om)\cap A_2}
     \end{equation}
   \begin{equation}
   \label{8.3.2}
g_{n}(\om) = \chi_{n}(\om)\cdot \epsilon^{O_n(\om)},\quad
\chi_{n}(\om):= \text{\bf 1}_{\mathcal A_{2,n}(\om)=A_2
 \cap\La_n(\om)^c}
   \end{equation}
$\eps>0$ will be specified later.  We are going to prove that for all $n$,
       \begin{equation}
          \label{8.3.3}
\mathcal{E}(g_{n+1}) \le  \mathcal{E}(g_{n}) \le \ldots \le  \mathcal{E}(g_{0})
       \end{equation}
where $\mathcal{E}$ is the expectation with respect
to  $P$.  Since $A_2\subset \La=\La_0$ and
 $\mathcal A_{2,N^*}(\om) = \mathcal A_2(\om)$, we then get from \eqref{8.3.3},
       \begin{equation}
          \label{8.3.4}
P\big( \{
\mathcal A_2=A_2\}\big) \le \eps^{N_{A_2}}
       \end{equation}
Recalling \eqref{8.1}, we set $P^{N^*}=P$ and for $n< N^*$,
   \begin{equation}
      \label{8.3.5}
 d P^n(\om)= d\nu'(\bar\xi'_{\La^c})d\nu'(\bar\xi''_{\La^c})
 dQ^{n}(\om_\La|\bar \om_{\La^c}),\quad \om=(\om_\La,\bar \om_{\La^c}),\;
 \bar \om_{\La^c}=(\bar\xi'_{\La^c},\bar\xi''_{\La^c})
     \end{equation}
calling  $\mathcal{E}^{n}$ the expectation w.r.t.\ $P^n$.   We have
$ \mathcal{E}(g_{n+1})= \mathcal{E}^{n+1}(g_{n+1})$,
hence by \eqref{7.17},
      \begin{equation}
            \label{8.3.6}
      \begin{split}
\mathcal{E}(g_{n+1}) = \sum_{\Delta} \sum_{\Delta'\subset \Delta}
\epsilon^{N_{A_2\cap \Delta'}} & \int P^{n}(d\,\om_{\Delta^{c}}) \Big[
\text{\bf 1}_{\left\{\Lambda_{n}= \Delta,\Lambda_{n+1} =
\Delta'\right\}}
\chi_n(\om) \\
& \quad \times \int_{\mathcal A_2(\om_\Delta)\supset A_2\cap \Si_{n+1}}  \pi_{\Delta}(d\om_{\Delta}
\,\vert\, \xi_{\Delta^{c} })  \Big]
      \end{split}
     \end{equation}
where $\Si_{n+1}= \Delta\setminus \Delta'$.
The last integral is equal to 1 if $A_2\cap \{ \Delta\setminus \Delta'\}=\emptyset$,
 while, if this is not the case,  by \eqref{e3.7.1a}
           \begin{eqnarray}
      \label{8.3.7}
&&\hskip-1cm \int_{\mathcal A_2(\om_\Delta)\supset A_2\cap \{ \Delta\setminus \Delta'\}}  \pi_{\Delta}(d\om_{\Delta}
\,\vert\, \xi_{\Delta^{c} }) \le  c( \eps_g +
e^{-c_{\rm pol}\zeta^2 \ell_{-,\ga}^d /2}),\quad
A_2\cap \{ \Delta\setminus \Delta'\}\ne\emptyset
     \end{eqnarray}
We then get from \eqref{8.3.6},
      \begin{equation}
            \label{8.3.8}
\mathcal{E}(g_{n+1}) \le \mathcal{E}^n(g_{n})\max\{1, \frac{c(
\eps_g + e^{-c_{\rm pol}\zeta^2 \ell_{-,\ga}^d /2})}{\eps^{3^d}}\}
     \end{equation}
We choose
     \begin{equation}
            \label{8.3.9}
            \eps =\frac 12 \Big( c( \eps_g +
e^{-c_{\rm pol}\zeta^2 \ell_{-,\ga}^d /2}) \Big)^{3^{-d}}
     \end{equation}
so that the $\max$ on the r.h.s. of \eqref{8.3.8} is 1  which  thus
proves \eqref{8.3.3} and \eqref{8.3.4}.

\subsection{Proof of Theorem \ref{thme3.6.1}}
\label{subsec:8.4}

As we have shown at the beginning of this Section,  Theorem
\ref{thme3.6.1} follows from \eqref{8.2} that we prove here.

Given $\eps$ as in \eqref{8.3.9}, for $\ga$ small enough we bound
the r.h.s of \eqref{8.2.3} as
    \begin{equation}
      \label{8.4.1}
P\big( \{ \mathcal A_1=A_1\}\big) \le 2^{2N_{A_1}} e^{-c_{\rm
pol}\zeta^2 \ell_{-,\ga}^d N_{A_1}/4}\le \eps^{N_{A_1}}
     \end{equation}

From \eqref{8.1.2}, \eqref{8.1.3}, \eqref{8.1.6}, \eqref{8.4.1} and
\eqref{8.3.4} we then get
    \begin{eqnarray*}
 &&\hskip-1cm P\Big(\{\La_{N^*}\supset \Delta\}^c\Big) \le
  \sum_{x\in \ell_{+,\ga}\mathbb Z^d\cap \Delta}
  \sum_{A: A\ni x, A\cap  \delta_{\rm out}^{\ell_{+,\ga}}[\La]\ne \emptyset}
  3^{N_A}2\eps^{N_A}\\&&\hskip2cm\le 2|\Delta|\sum_{n\ge \frac {{\rm
  dist}(\Delta,\La^c)}{\ell_{+,\ga}}}(3\eps)^n
     \end{eqnarray*}
that implies \eqref{8.2}.\qed

\newpage

\vskip1cm

\part{Appendices}

\vskip.5cm

\appendix

\setcounter{equation}{0}

\section{ Operators on Euclidean spaces}
\label{sec:B}

For the sake of completeness we recall here some elementary
properties of operators on finite dimensional Hilbert
spaces used in the previous sections. We call $\mathcal H$
the real Hilbert space of vectors $u=\{u(i)\}$ with scalar
product
   \begin{equation}
      \label{7.0}
(u, v) = \sum_{i}u(i)v(i)
     \end{equation}
where $i$ above ranges in a finite index set on which a distance
$|i-j|$ is defined (in our applications $i$ stands for a pair
$(x,s)$, with $x\in \ell \mathbb Z^d\cap \La, s\in \{1,..,S\}$,
and either $\ell=\ell_{-,\ga}$ or $\ell=\ga^{-1/2}$, $\La$ being a
fixed $\mathcal D^{(\ell_{-,\ga})}$-measurable bounded subset of
$\mathbb R^d$. Operators on $\mathcal H$ are identified to
matrices $B=B(i,j)$ by setting $\dis{Bu(i)=\sum_jB(i,j)u(j)}$.  We
write $\dis{|u|_\infty=\max_i |u(i)|}$,
   \begin{equation}
      \label{7e.0}
\|B\|^2 = \sup_{u\ne 0} \frac{(Bu,Bu)}{(u,u)},\quad
 \|B\|_\infty=\sup_{u\ne 0}\frac{|Bu|_\infty}{|u|_\infty}
     \end{equation}
Recall that
   \begin{equation}
      \label{7e.0.0}
 \|B\|_\infty \le \max_i \sum_j |B(i,j)|,\quad \|B\|\le \max_i \{\sum_j
 |B(i,j)|,\sum_j
 |B(j,i)|\}=:|B|
     \end{equation}
The first inequality in \eqref{7e.0.0} is obvious. To prove
the second one  we write
   \begin{eqnarray*}
&&\sum_i \Big(\sum_j B(i,j) u(j)\Big)^2 \le
\sum_{i,j_1,j_2}| B(i,j_1)| |B(i,j_2)|\frac 12
\big(u(j_1)^2+u(j_2)^2\big)\\&& \hskip2cm \le
\sum_{i,j_1,j_2}|B(i,j_1)| |B(i,j_2)| u(j_1)^2 \le |B|^2
\sum_i u(i)^2
    \end{eqnarray*}

\vskip1cm

In Theorem \ref{thmappB.2} below  we consider matrices of
the form $B=C'A^{-1}C''$, thus including $(QAQ)^{-1}$
(after restricting to $Q\mathcal H$) and $PA(QAQ)^{-1}QA$,
the matrix considered in \eqref{Ie.3.2.32.1}.  With in mind
these two applications we will suppose the diagonal
elements of $A$ strictly positive and large.

\vskip1cm

\begin{thm}
 \label{thmappB.2}
Let $B=C'A^{-1}C''$ with  $A= D+R$,  $D$  a diagonal matrix, and
suppose there are $c>0$, $c'>0$ and $b>0$ such that the following
holds (recall the definition of the  norm $|C|$ given in
\eqref{7e.0.0}).
   \begin{equation}
      \label{appB.6.1}
|C'|+|C''|+|R| \le c
    \end{equation}
The diagonal elements $D(i,i)$ of $D$ are such that $D(i,i)\ge b$
for every $i$. Finally $C'(i,j)=C''(i,j)=R(i,j)=0$ whenever
$|i-j|\ge c'\ga^{-1}$. Then if $b$ is large enough,
   \begin{equation}
      \label{appB.6}
\|B\| \le \frac {2c^2}{b},\quad \|B\|_\infty \le \max_{i}
\sum_{j}|B(i,j)| e^{\ga |i-j|} \le \frac {2c^2 e^{2c'}}b
    \end{equation}

\end{thm}

\vskip.5cm

{\bf Proof.} By \eqref{7e.0.0},
{$\|R\|\le c$}.
On the other hand  $\|D\|^{-1}\le b^{-1}$
and for $b$ so large that $b^{-1}c<1$ the sum on the
r.h.s.\ of \eqref{appB.7.1} below converges and
    \begin{equation}
      \label{appB.7.1}
A^{-1} = D^{-1} - D^{-1} R D^{-1}+ D^{-1} R  D^{-1} R
D^{-1}-\cdots= \sum_{n=0}^\infty \Big(-D^{-1} R\Big)^n
D^{-1}
     \end{equation}
as seen by multiplying  the r.h.s.\ of \eqref{appB.7.1}
from the left by $A$: we then get $
AD^{-1}\big(1-RD^{-1}+\cdots\big)$ which is equal to 1
after writing $AD^{-1}= 1+RD^{-1}$ and after telescopic
cancellations. Thus \eqref{appB.7.1} holds and
    \begin{equation}
      \label{appB.8}
\|A^{-1} \|\le   \sum_{n=0}^\infty b^{-n-1}\|R\|^n \le
\frac 1{b(1-c/b)}
     \end{equation}
hence, {recalling   \eqref{7e.0.0},
we get }
 the first inequality in \eqref{appB.6}. We write
   \begin{eqnarray*}
&& \sum_{j}|B(i,j)| e^{\ga |i-j|} \le \sum_{i_1}|C'(i,i_1)| e^{\ga
|i-i_1|}\sum_{i_2}|A^{-1}(i_1,i_2)|e^{\ga |i_1-i_2|}\sum_j |
C''(i_2,j)|e^{\ga |i_2-j|}
\\&&\hskip2cm \le c^2 e^{2c'}\max_{i_1}
\sum_{i_2} |A^{-1}(i_1,i_2)|e^{\ga |i_1-i_2|}
    \end{eqnarray*}
Since $\dis{\sum_{i_2}|R(i_1,i_2)|e^{\ga |i_1-i_2|} \le
e^{c'}c}$, by \eqref{appB.7.1}
    \begin{equation*}
 \sum_{i_2}
|A^{-1}(i_1,i_2)|e^{\ga |i_1-i_2|} \le \sum_{n=0}^\infty
b^{-n-1}[e^{c'}c]^n
     \end{equation*}
hence the second inequality in \eqref{appB.6}.

 \qed

\vskip1cm

In the next two theorems we consider a matrix $R_1$ with
small norm, it represents in our applications the matrix
$PA(QAQ)^{-1}QA$ which by Theorem \ref{thmappB.2} has
indeed a small norm (if $b$ is large).

\vskip.5cm

\vskip1cm

\begin{thm}
 \label{thmappB.1}
Let  $B=A+R_1$; suppose  $A$  symmetric, $(u,Au)\ge \kappa
(u,u)$ for all $u$;  $\|R_1\|\le \eps$ and $\kappa>\eps>0$.
Then $B$ is invertible and
   \begin{equation}
      \label{appB.0}
\|B^{-1}\| \le \frac 1{\kappa'},\quad \kappa'=\kappa-\eps
     \end{equation}
Suppose further that
   \begin{equation}
      \label{appB.1}
\sup_{i} \sum_{j}|B(i,j)| e^{\ga |i-j|} \le a <\infty 
     \end{equation}
then
   \begin{equation}
      \label{appB.2}
|B^{-1}(i,j)| \le  (\frac 1a+\frac 1
{\kappa'})\exp\Big\{-\frac{\kappa' \ga|i-j|}
{a+\kappa'}\Big\}
     \end{equation}

\end{thm}

\vskip.5cm

{\bf Proof.}  By the integration by parts formula,
   \begin{equation}
      \label{appB.2.1}
e^{-B t}= e^{-At}-\int_0^t e^{-B s}R_1 e^{-A(t-s)}
     \end{equation}
Since $\|e^{-At}\|\le e^{-\kappa t}$,
   \begin{equation}
      \label{appB.2.2}
\|e^{-B t}\|\le e^{-\kappa t}+ e^{-\kappa
t}\sum_{n=1}^\infty \frac{(\eps t)^n}{n!} \le  e^{-(\kappa
-\eps) t}
     \end{equation}
Then $\dis{\int_0^\infty e^{-Bt}}$ is well defined and
equal to $B^{-1}$; \eqref{appB.0} also follows.

Calling $e_i$ the vector with components $e_i(j) =\text{\bf
1}_{i=j}$,
   \begin{equation}
      \label{appB.3}
B^{-1}(i,j)=
 \int_0^{\tau}  \big(e_i,e^{-Bt} e_j\big)+  \int_{\tau}^\infty  \big(e_i,e^{-Bt} e_j\big)
    \end{equation}
By \eqref{appB.2.2},
   \begin{equation}
      \label{appB.4}
|\int_{\tau}^\infty  \big(e_i,e^{-Bt} e_j\big)| \le
\frac{e^{-\kappa' \tau}}{\kappa'},   \quad
\kappa'=\kappa-\eps
    \end{equation}
By a Taylor expansion:
   \begin{equation}
      \label{appB.4.1}
| \big(e_i,e^{-Bt} e_j\big)| \le \sum_{n=0}^\infty\,
\frac{t^n}{n!} \,e^{-\ga|i-j|} \sum_{i_1,..,i_{n-1}}
|B(i,i_1)| e^{\ga|i-i_1|}\cdots  |B(i_{n-1},j)|
e^{\ga|j-i_{n-1}|}
    \end{equation}
hence using \eqref{appB.1},
   \begin{equation}
      \label{appB.5}
| \int_0^{\tau}  \big(e_i,e^{-Bt} e_j\big)| \le \frac
{e^{a\tau -\ga|i-j|}}{a}
    \end{equation}
By choosing $\dis{\tau= \frac{\ga|i-j|}{a+\kappa'}}$ we
then get \eqref{appB.2} from \eqref{appB.4} and
\eqref{appB.5}.


\qed

\vskip1cm

\begin{thm}
 \label{thmappB.1.1}

Let $B=A+R_1$ as in Theorem \ref{thmappB.1}; call $D$ the
diagonal part of $A$, $R_0:=A-D$, $R=R_0+R_1$  and suppose
that $\|R\|_\infty<\infty$. Then
   \begin{equation}
      \label{appB.4.2}
\|B^{-1}\|_\infty \le \frac 1 \kappa  + \frac
{\|R\|_\infty}{\kappa^{2}} \Big(1 +
\frac{\|R\|_\infty}{\kappa-\eps} \Big)
     \end{equation}

\end{thm}

\vskip.5cm

{\bf Proof.} Recalling that $B=D+R$, we use the identity
   \begin{equation*}
 B^{-1} =  D^{-1} -  D^{-1} R D^{-1} + D^{-1} R  B^{-1} R D^{-1}
     \end{equation*}
Then
   \begin{equation*}
 B^{-1}(i,j)  =  (e_i,D^{-1}e_j) -  (D^{-1}e_i, R  D^{-1}e_j)
  +\sum_{k,h}(e_i, D^{-1} R e_k)(e_k,  B^{-1} e_h)(e_h, R D^{-1} e_j)
     \end{equation*}
so that
        \begin{eqnarray*}
&&\sum_j |B^{-1}(i,j)|\le  \kappa^{-1} + \kappa^{-2}
\|R\|_\infty + \|B^{-1}\|  \kappa^{-2} \|R\|_\infty^2
    \end{eqnarray*}
and \eqref{appB.4.2} follows using  \eqref{appB.0}.

\qed

\vskip2cm
\section{Proof of Theorem \ref{thm7.2.1 }}
        \label{appCC}
In the sequel cubes are always cubes in $\mathcal
D^{(\ell_{+,\ga})}$ and a cube $C$ is called ``older'' than $C'$ if
there is $n$ such that $C'\subset \La_n$ and $C\subset
\La_n^c$.
We will prove the theorem as a consequence of the following
property:

\vskip.5cm

{\bf Property P.}  {\sl Let $C$ be a good cube, $x\in
\ell_{-,\ga}\mathbb Z^d\cap C$, $\{C_i\}$ the cubes older
than $C$ which intersect $B_x(2^d10^{-10}\ell_{+,\ga})$. If
either $\{C_i\}$ is empty or if all $C_i$ are good, then
$q'_\La\cap C_x^{(\ell_{-,\ga})}=q''_\La\cap
C_x^{(\ell_{-,\ga})}$.}

\vskip.5cm

Before proving Property P, we will use it to prove Theorem
\ref{thm7.2.1 }.  Suppose that for some $N$, $\La_N$ is non
empty and that all cubes in $\delta_{\rm
out}^{\ell_{+,\ga}}[\La_N]$ are good (thus the sequence
$\La_n$ stops at $N$).  Let $C$ be a cube in $\delta_{\rm
out}^{\ell_{+,\ga}}[\La_N]$, $x\in \ell_{-,\ga}\mathbb
Z^d\cap C$ and at distance $\le \ga^{-1}$ from $\La_N$.
Then $B_x(2^d10^{-10}\ell_{+,\ga})\cap \La_N^c$ intersects
only cubes of $\delta_{\rm out}^{\ell_{+,\ga}}[\La_N]$,
which are by assumption good; then by Property P,
$q'_\La\cap C_x^{(\ell_{-,\ga})}=q''_\La\cap
C_x^{(\ell_{-,\ga})}$, hence \eqref{7.2.2}. \eqref{7.2.3}
holds because all cubes of $\delta_{\rm
out}^{\ell_{+,\ga}}[\La_N]$ are good.

\vskip1cm

We start the proof of Property P by introducing a new
function $M(x)$, $x\in \ell_{-,\ga}\mathbb Z^d$. We set
$M(x)=\infty$ outside $\La$ and at all $x$ which are in bad
cubes. The definition of $M(x)$ on the good cubes is given
iteratively in $\La_n^c$. We thus suppose to have already
defined $M(x)$ on all cubes of $\La_n^c$ and have to define
it on $\Si_{n+1}=\La_{n+1}^c\setminus \La_n^c$. Let thus
$C\subset \Si_{n+1}$ and $x\in C$. We set $M(x)=0$ if
$B_x(10^{-10}\ell_{+,\ga})\cap \La_n^c = \emptyset$,
otherwise
          \begin{equation}
     \label{CC.1}
M(x):= 1+\max\Big\{ M(y)\big| y\in \ell_{-,\ga}\mathbb Z^d\cap B_x(10^{-10}\ell_{+,\ga}),
\text{\;\;$y$ such that $C_y^{(\ell_{+,\ga})}\subset \La_n^c$}\Big\}
     \end{equation}
To compute the value of $M(x)$, $x\in C$, $C\subset \Si_{n+1}$, we need to look at all sequences $y_1,y_2,....$ such that: $|y_h-y_{h-1}|\le 10^{-10}\ell_{+,\ga}$, $C^{(\ell_{+,\ga})}_{y_h}$ is older than $C^{(\ell_{+,\ga})}_{y_{h-1}}$, $h_0=x$ and
to know whether   the
cubes $C^{(\ell_{+,\ga})}_{y_h}$ are good or bad.  In principle the sequence
may be arbitrarily long but in fact it is not:

\vskip1cm

{\bf Lemma 1.} \;{\sl  Let $C$ be a good cube, $x\in C$,
then the value of $M(x)$ depends only on
whether the cubes
$\{C_i\}$ are good or bad, where $\{C_i\}$ is the
collection of cubes older than $C$ which intersect
$B_x(2^d10^{-10}\ell_{+,\ga})$.}

\vskip.5cm

{\bf Proof.}  Since
any ball of radius $(2^{d}10^{-10}+1)\ell_{+,\ga}$ intersects at most
$2^d$ cubes of the partition  $\mathcal
D^{(\ell_{+,\ga})}$, then any sequence  $y_1,y_2,....$  as above consists at most of
$2^d$ elements.  \qed

\vskip1cm

Since $\bar m= 2^{d}+2$, then
          \begin{equation}
     \label{CC.2}
\text{either $M(x)<\bar m-2$ or $M(x)=+\infty$}
      \end{equation}
We will next prove:

\vskip.5cm

{\bf Lemma 2.} \;{\sl Let $C$ be a good cube, $x\in C$,
then, if $\bar m- M(x)=h>0$, }
          \begin{equation}
     \label{CC.3}
q'_\La\cap C_x^{(\ell_{-,\ga})}=q''_\La\cap
C_x^{(\ell_{-,\ga})},\quad
\max_{s\in\{1,..,S\}}|\rho^{(\ell_{-,\ga})}({
q}'_{\La};x,s)-\rho^{(k)}_s|
 \le \zeta_{h}
      \end{equation}

\vskip.5cm

{\bf Proof.} The proof is by induction on the ``age'' of
the cubes.  We thus suppose that  the above statements
holds for all cubes of $\La_n^c$.  Let $C$ be a good cube
in $\Si_{n+1}$, then the above properties hold by the
definition of the function $K$ and of good cubes.  \qed

\vskip.5cm

 Property P is then an immediate consequence of Lemma  2 and \eqref{CC.2}.
  \vskip1cm

\section{ Mean field}
\label{sec:C}

In this appendix  we prove  Theorems \ref{thme2.1} and \ref{thme2.2}. Our approach is based on the recent works \cite{GM,GMRZ}, having in mind that in \cite{GM} the total density was set to $1$, the temperature being the free parameter, while here we  fix the (inverse) temperature $\beta =1$, the total density $x$ being the free parameter. The two approaches are equivalent, see \eqref{e2.3}.

To achieve our goal, we will need Lemmas \ref{lem=zmonotone}, \ref{lem=slopejump}, \ref{lem=notation_agreement}, \ref{lem=pvariations} and  \ref{lem=common_tangent} below. The first lemma is an essential property relating the total density $x$ to the corresponding constrained minimizer in a one-to-one way. The second and third lemmas respectively deal with the first and second derivatives of the free energy. They show in particular that the sign of the second derivative depends on the roots of some peculiar second degree polynomial. The fourth lemma studies the locations of these roots, while the fifth and last lemma gives a general condition for a piecewise-convex function to have a common tangent at two different points.

The section is organized as follows. We first give some notations and reformulate known results, before stating our auxiliary lemmas. Then we prove Theorems \ref{thme2.1} and \ref{thme2.2}, while the proofs of lemmas are deferred to the end of the present section.

\vskip.5cm

{\bf Notations}

\vskip.5cm

 For any $x\in (0,+\infty), z\in [0,1]$, we will denote by $\rho^{(z,x)}$ the density vector $\rho$ defined as follows:
\begin{equation}
\label{rhozx}
  \rho^{(z,x)}_i =
  \left\{
    \begin{array}{crl}
      \frac{1+(S-1)z}{S}x & \mbox{for} & i = 1 \\
      \frac{1-z}{S}x & \mbox{for} & i=2,\ldots, S.
    \end{array}
  \right.
\end{equation}
Notice that $\sum \rho^{(z,x)}_i = x $ and rewrite \eqref{e2.4} as follows:
\begin{equation}
f^{\rm mf}(x) = \inf\big\{F^{\rm mf}(\rho^{(z,x)}); 0\leq z \leq 1 \big\}.
\end{equation}

Now, remarking that $zx = \rho^{(z,x)}_1- \rho^{(z,x)}_2$, we adapt a result from \cite{GMRZ,GM}. Namely, recalling Theorem A.1 in \cite{GM} or section 3 in \cite{GMRZ}, and comparing \eqref{treshold} and \eqref{jump} below with (A.10) and (A.22) in \cite{GM}, we know that for any $S>2$ there exists a threshold
\begin{equation}
\label{treshold}
x_{S}: = 2\frac{S-1}{S-2}\ln (S-1)
\end{equation}
such that
\begin{itemize}
\item for all $x<x_{S}$, the function $z \mapsto F^{\rm mf}(\rho^{(z,x)})$ reaches its minimum at $z=0$;
\item for all $x>x_{S}$, the function  $z \mapsto F^{\rm mf}(\rho^{(z,x)})$ reaches its minimum at $z=z(x)$, defined as the largest solution of the equation $R(z)=x$ where
\begin{equation}
\label{jump}
R(z):=\frac{1}{z} \ln \frac{1+(S-1)z}{1-z};
\end{equation}
\item at $x = x_{S}$, the function  $z \mapsto F^{\rm mf}(\rho^{(z,x)})$ reaches its minimum at $z=0$ \emph{and} at $z = z(x_{S}) = \frac{S-2}{S-1}$.
\end{itemize}
The statement above means that we have
\begin{equation}
\label{branches}
f^{\rm mf}(x) = \begin{cases}
f^{\rm dis}(x) := F^{\rm mf}(\rho^{(0,x)}) & \text{ if }  x \leq x_{S} \\
f^{\rm ord}(x) :=F^{\rm mf}(\rho^{(z(x),x)}) & \text{ if }  x \geq x_{S}.
\end{cases}
\end{equation}
First of all, we will see that
\begin{lemma}[Monotony of $R$ and $z$]
\label{lem=zmonotone}
The functions $R:z\to R(z)$ and $z:x\to z(x)$ are both increasing respectively on $[z_{S},1)$ and $[x_{S},+\infty)$, where $z_S=\frac{S-2}{S-1}$. They satisfy the relations $R\circ z = \textrm{Id}_{[x_{S},+\infty)}$ and $z\circ R = \textrm{Id}_{[z_{S},1)}$.
\end{lemma}
Moreover
\begin{lemma}
\label{lem=slopejump}
\begin{equation}
\label{slopeext}
 \lim_{x\to 0 } (f^{\rm mf})'(x) = - \infty \quad \text{ and } \lim_{x \to +\infty } (f^{\rm mf})'(x)  = + \infty,
\end{equation}
\begin{equation}
\label{slopejump}
 \lim_{x\uparrow x_{S} } (f^{\rm mf})'(x) -  \lim_{x\downarrow x_{S} } (f^{\rm mf})'(x)  = \left(1-\frac{2}{S}\right)\ln(S-1).
\end{equation}
\end{lemma}

\begin{lemma}
\label{lem=notation_agreement}
\begin{align}
\label{second_derivative}
\forall x \leq x_{S}, \quad \frac{d^{2}f^{\rm dis}}{dx^{2}}(x) & =  \frac{S-1}{S} + \frac{1}{x}\\
\nn \\
\label{notation_agreement}
\forall x \geq x_{S}, \quad \frac{d^{2} f^{\rm ord}}{dx^{2}}(x) & =  \left(\frac{S-1}{S}\right)\frac{z'(x)}{xz(x)}\left[R^{+}_{z(x)}-x\right]\left[x-R^{-}_{z(x)}\right] ,
\end{align}
where $R^{\pm}_{z}$ denotes the roots of the second degree polynomial $P_{z}(X) := X^{2} - b_{z}X - c_{z}$  given by
\begin{equation*}
\left\{
\begin{array}{rcl}
b_{z} & := & \frac{S(S-2)}{(S-1)[1+(s-1)z]} \\
c_{z} & := & \frac{S^{2}}{(S-1)(1-z)[1+(s-1)z]}.
\end{array}
\right.
\end{equation*}

\end{lemma}
According to Lemma \ref{lem=notation_agreement}, the convexity properties of $f^{\rm ord}$ will follow from the position of the roots of $P_{z(x)}$ with respect to $x$. We will actually prove the lemma below
\begin{lemma}[Roots of $P_{z}$]\mbox{}
\label{lem=pvariations}
The roots of the polynomial $P_{z}$ are such that $R^{-}_{z} < 0 < R^{+}_{z}$ and
\begin{itemize}
\item for any $S \geq 60$, and for all $z \in [z_{S},1)$, $R^{+}_{z} >  R(z)$;
\item for any $3 \leq S \leq 59 $, there exists a unique $z_S^{\star} \in (z_S,1)$ such that $R^{+}_{z_S^{\star}} = R(z_S^{\star})$. Moreover, $R^{+}_{z} < R(z)$ on $[z_S,z_S^{\star})$ and $R^{+}_{z} > R(z)$ on $(z_S^{\star},1)$.
\end{itemize}
\end{lemma}
Eventually, the following fact will be helpful to analyze the convex envelope of $f^{\rm mf}$:
\begin{lemma}
\label{lem=common_tangent}
Let $f:(a,b] \to \mathbb{R}$ and $g:[b,c) \to \mathbb{R}$ be convex functions with continuous second derivatives. If $f(b)=g(b)$ and if $\dis{ \inf_{x<b} f'(x) <g'(b) < f'(b) < \sup_{x>b} g'(x)} $, then there exists a common tangent to their respective graphs $\Gamma_f,\Gamma_g$.
\end{lemma}

\vskip.5cm

We are now ready to prove our theorems.

\vskip.5cm

{\bf Proof of  Theorem \ref{thme2.1}.}
By \eqref{second_derivative}, $f^{\rm dis}$ is strictly convex. Let us now study the convexity of $f^{\rm ord}$. Fixing $x \geq x_{S}$ we remark that Lemma \ref{lem=zmonotone} implies $z(x) \geq z_S$ and $R(z(x)) = x$ so that Lemma \ref{lem=pvariations} gives:
\begin{itemize}
\item for all $S$,  $x > 0 > R^{-}_{z(x)}$;
\item if $S\ge 60$ then $R^{+}_{z(x)} >  x $;
\item if $3 \le S \le 59$,  $R^{+}_{z(x)} < x $ if $ x < x_{S}^{\star} $ and $R^{+}_{z(x)} > x $ if $ x > x_{S}^{\star} $,  where $x_{S}^{\star}:= R(z_S^{\star})$.
\end{itemize}
Therefore, \eqref{notation_agreement} shows that  if $S \geq 60$ then $f^{\rm ord}$ is strictly convex  on $[x_{S}, \infty)$, while if $S \leq 59$ then $f^{\rm ord}$ is strictly concave on $[x_{S},x_{S}^{\star}]$ and strictly convex on $[x_{S}^{\star},+\infty)$.

\vskip.5cm

Let us analyze the convex envelope of $f^{\rm mf}$.
\begin{itemize}
\item If $S \ge 60$, \eqref{branches} and Lemma \ref{lem=slopejump} show that Lemma \ref{lem=common_tangent} applies to $f = f^{\rm dis}$, $g=f^{\rm ord}$, $a =0$, $b = x_{S}$ and $c=+\infty$.
\item If $S \le 59$ we first have to deal with the concave part of $f^{\rm ord}$. We introduce the function $g$ defined by
\begin{equation*}
g(x) =
\begin{cases}
f^{\rm ord}(x_{S}^{\star}) + (f^{\rm ord})'(x_{S}^{\star})\cdot( x -x_{S}^{\star} ) & \text{ if } x \le x_{S}^{\star} \\
f^{\rm ord}(x) & \text{ if } x \ge x_{S}^{\star}.
\end{cases}
\end{equation*}
 Since $(f^{\rm ord})''(x_{S}^{\star})= 0$, $g$ is convex and has continuous second derivatives. Moreover, on $[x_{S},x_{S}^{\star}]$, the graph of $g$ is a line located above the graph of $f^{\rm ord}$ (concavity of $f^{\rm ord}$); since the latter intersects the (convex) graph of $f^{\rm dis}$, the graph of $g$ and the graph  of $f^{\rm dis}$ intersect at some point with abscisse  $b \in (x_{S},x_{S}^{\star})$. Besides, the concavity of  $f^{\rm ord}$ implies $g'(b) = (f^{\rm ord})'(x_{S}^{\star}) <  (f^{\rm ord})'(x_{S})$, while the convexity of $f^{\rm dis}$ implies $(f^{\rm dis})'(b) >(f^{\rm dis})'(x_{S}) $. Thus Lemma \ref{lem=slopejump} shows that Lemma \ref{lem=common_tangent} applies to $f = f^{\rm dis}$ and $g$ defined above.
\end{itemize}
In any case, Lemma \ref{lem=common_tangent} implies that there exists a line $T_{1}$ which is simultaneously tangent to  the disordered branch of $f^{\rm mf}$ (at some point $x_{-}< x_{S}$) and to the ordered branch of $f^{\rm mf}$ (at some other point $x_{+} > x_{S}$). The function $f^{\rm mf}(x)$ being strictly convex  outside $[x_{-},x_{+}]$, the graph of its convex envelope necessarily coincides with $T_{1}$ (resp. with the graph of $f^{\rm mf}$) inside (resp. outside) $[x_{-},x_{+}]$. Denoting by $\lambda_{1}$ the slope of $T_{1}$, the convex envelope of $f_{\lambda_{1}}^{\rm mf}(x) = f^{\rm mf}(x)-\lambda_{1}x$ is horizontal on $[x_{-},x_{+}]$ and strictly convex outside this segment of minimizers. \qed
\vskip1cm

{\bf Proof of Theorem \ref{thme2.2}.}
If $\rho$ is a minimizer of $F=F_{1,\lambda_{1}}^{\rm mf}$, then $x = \sum_{s}\rho_{s}$ is a minimizer of $f_{1,\lambda_{1}}^{\rm mf}$ so that $x \in \left\{ x_{-},x_{+} \right\}$. If  $x = x_{-}<x_{S}$, then $\rho = \rho^{(S+1)}$; if $x = x_{+}>x_{S}$, then there exists $k\in\left\{1,\ldots,S \right\}$ such that $\rho = \rho^{(k)} := \tau^{1,k} \cdot \rho^{z(x),x} $, where $\tau_{1,k}$ exchanges the first and the $k^{\rm th}$ coordinates. Reciprocally, the above $S+1$ vectors $\rho^{(k)}$ are all minimizers of $F$. Moreover,
\begin{align*}
\sum_{s} \rho^{(1)}_s = x^{+} > x^{-} = \sum_{s} \rho^{(S+1)}_s,
\end{align*}
thus proving \eqref{e2.5}.

We now show the second part of  Theorem \ref{thme2.2} dealing with the Hessian of $F$. Straightforward computations show:
\begin{align*}
 L^{(k)}(s,s') =\frac
 {\partial^2 F}{\partial\rho_s
 \partial\rho_{s'}}\Big|_{\rho=\rho^{(k)}}= \frac{1}{\rho^{(k)}_s}
 \text{\bf 1}_{s=s'} + \text{\bf 1}_{s\ne s'}
\end{align*}

Since $\rho^{(k)}$ is a minimizer, $L^{k}:= D^{2}F(\rho^{(k)})$ is semi-definite positive. Actually, $L^{(k)}$ is definite positive, or else the third order corrections in the Taylor--Lagrange formula would contradict the extremality of $\rho^{k}$:
\begin{align*}
\forall s,t,u \quad \frac{\partial^{3} F}{\partial \rho_{s} \partial \rho_{t} \partial \rho_{u}} &  = - \frac{1}{\rho_{s}^{2}}\text{\bf 1}_{s=t=u}.
\end{align*}
Taking an orthonormal basis of eigenvectors, the estimate \eqref{2.4} holds with $\kappa^{\ast}> 0 $ the smallest eigenvalue of $L^{(1)},L^{(S+1)}$.\qed

\vskip.5cm

This section ends with the proofs of Lemma \ref{lem=zmonotone}, Lemma  \ref{lem=slopejump}, Lemma \ref{lem=notation_agreement}, Lemma  \ref{lem=pvariations} and  Lemma \ref{lem=common_tangent}  which are stated at the beginning of the section and used in the proofs above.

\vskip.5cm

{\bf Proof of Lemma \ref{lem=zmonotone}.}
We express $R'(z) = \frac{g(z)}{z^{2}}$ and show that $g$ is always positive. Recalling \eqref{jump} we have:
\begin{align*}
R'(z) & =  \frac{1}{z}\cdot \left[ \frac{ S -1 }{ 1 + (S-1)z }
+ \frac{1}{1-z}  \right]  -\frac{1}{z^{2}} \cdot \ln \frac{ 1 + (S-1)z }{ 1 -z }, \\
 & = \frac{1}{z^{2}} \left[  \frac{1}{ 1 -z }   -  \frac{1}{1 + (S-1)z } - \ln  \frac{ 1 + (S-1)z }{ 1 -z }\right], \\
 & = \frac{1}{z^{2}} g(z).
\end{align*}

We now show that $g$ is always positive:
\begin{align*}
  g(z) & = \frac{1}{ 1 -z }  - \frac{1}{1 + (S-1)z }  - \ln  \frac{ 1 + (S-1)z }{ 1 -z }, \\
  g'(z) & = \frac{1}{(1-z)^{2}}  + \frac{S-1}{[1+(S-1)z]^{2}} - \frac{S-1}{1+(S-1)z} + \frac{1}{1-z}\\
  & = \frac{Sz\left[2(S-1)z - (S - 2)\right]}{(1-z)^{2}[1+(S-1)z]^{2}}.
\end{align*}
We see immediately that $g'>0$ for all $z > \frac{S-2}{2(S-1)}$, so that $g$ increases on  $[\frac{S-2}{S-1},1)$. On this subinterval, $g$ is thus minimal at $\left(\frac{S-2}{S-1}\right)$ where it takes the value
\begin{equation*}
  g\left(\frac{S-2}{S-1}\right)  = - 2 \ln (S-1) - \frac{1}{S-1} + (S-1),
\end{equation*}
which increases with $S$, vanishes at $S=2$, and is strictly positive for all $S \geq 3$. From this it follows that $g$ is strictly positive  on $[\frac{S-2}{S-1},1)$, which implies that $R$ is strictly increasing with $z$. Since $R$ goes to $+\infty$ when $z\to 1$, Lemma \ref{lem=zmonotone} is proved.\qed

\vskip1cm

{\bf Proof of Lemma \ref{lem=slopejump}.}

Since $\rho^{(0,x)}$ is the vector $(\frac{x}{S},\ldots,\frac{x}{S})$,  equations  \eqref{e2.1}, \eqref{branches}, \eqref{rhozx} give for all $x < x_{S}$:
\begin{align}
\nn f^{\rm mf}(x) & = \frac{S(S-1)}{2}\left(\frac{x}{S}\right)^{2} + S \frac{x}{S}\left( \ln \frac{x}{S} - 1 \right)\\
\label{fddis}\left(f^{\rm mf}\right)'(x) &= \frac{S-1}{S}x +  \ln \frac{x}{S} .
\end{align}

Recalling \eqref{e2.1}, \eqref{branches} and \eqref{rhozx}, $f^{\rm mf}(x) = F \left( x,z(x) \right)$ holds  for all $x > x_{S}$, where
\begin{align}
F(x,z) & = \frac{1}{2} \frac{S-1}{S} x^{2}(1-z^2)
+(S-1)\frac{x(1-z)}{S} \ln\frac{x(1-z)}{S}
\nonumber
\\
&\quad +\frac{x(1+(S-1)z)}{S}\ln \frac{x(1+(S-1)z)}{S} - x.\label{FRz}
\end{align}

Using \eqref{FRz} and recalling that $\left(\frac{\partial F}{\partial z}\right)_{| z(x)}=0$, we have  for all $x > x_{S}$:
\begin{align}
\left(f^{\rm mf}\right)'(x) &= \left(\frac{\partial F}{\partial x}\right)_{\big\vert x,z(x)} + z'(x)\left(\frac{\partial F}{\partial z}\right)_{\big \vert z(x)} \nonumber \\
                  & = \frac{S-1}{S} x(1-z^2) +(S-1)\frac{1-z}{S} \left[ \ln\frac{x(1-z)}{S}+1 \right] \nonumber \\
\label{fdordaux}               & \quad + \frac{1+(S-1)z}{S}\left[\ln \frac{x(1+(S-1)z)}{S}+1\right] -1  \\
\label{fdord}
 & = \frac{S-1}{S}x + \ln \frac{x}{S} + \ln (1-z) + \frac{xz}{S}.
\end{align}

From \eqref{jump}, we know that $x \geq \frac{1}{z}\log \frac{1}{1-z}$, thus
\begin{align}
\nn \left(f^{\rm mf}\right)'(x) & \ge \left(\frac{S-1}{Sz} - 1 \right)\log  \frac{1}{1-z} + \ln \frac{x}{S} + \frac{xz}{S}, \\
\label{est}                           & \geq  \ln \frac{x}{S} + \frac{xz}{S}.
\end{align}
From Lemma \ref{lem=zmonotone}, $z(x) \to z_{S}$ as $x \to x_{S}$ thus \eqref{slopejump} follows from \eqref{fddis} - \eqref{fdord}. Similarly, $z(x) \to 1$ as $x \to \infty$, thus \eqref{slopeext} follows by taking limits in \eqref{fddis} and \eqref{est}.
\qed

\vskip1cm

{\bf Proof of Lemma \ref{lem=notation_agreement}.}

First notice that \eqref{second_derivative} follows from \eqref{fddis}. Using \eqref{fdordaux} we get:
\begin{align*}
\left(f^{\rm mf}\right)''(x) &  = \frac{1}{x} + \frac{S-1}{S}(1-z^2 - 2xz'z) +  \frac{S-1}{S}z'\ln \frac{1+(S-1)z}{1-z} \\
& =  \frac{1}{x} + \frac{S-1}{S}\left(1-z^2 - 2xz'z\right) + \frac{S-1}{S}z'zx\\
& = \frac{z'}{x}\left(  R' + \frac{S-1}{S}\left[(1-z^2)xR' - zx^{2}\right]  \right)
\\
& = -\frac{1}{x}\frac{z'}{z}P_{z(x)}(x)
\end{align*}
where we used $\frac{1}{z'(x)} = R'(z(x))$ ($ = R'$ by abusing notations) and $zR' = -R(z) + \frac{S}{(1-z)\left[ 1+ (S-1)z\right]}$ (from \eqref{jump}). This achieves the proof of \eqref{notation_agreement}.\qed

\vskip1cm

{\bf Proof of Lemma \ref{lem=pvariations}.}

\begin{itemize}
\item roots of $P_{z}$

We notice that the discriminant of $P_{z}$
\begin{equation*}
 \Delta(P_{z}) = \frac{S[S+(3S-4)z]}{(1-z)}
\end{equation*}
is always positive, so that the two distinct roots of $P_{z}$ are given by
\begin{equation}\label{tworoots}
R^{\pm}_{z} =  \frac{S}{2(S-1)}\left[  \frac{(S-2) \pm \sqrt{\frac{S[S+(3S-4)z]}{1-z}}}{1+(S-1)z} \right].
\end{equation}
For all positive $z$ we have $ \frac{S+(3S-4)z}{1-z} \geq S$, thus $R^{-}_{z}$ is negative while $R^{+}_{z}$ is positive.

\item sign of $R^{+}_{z}-R(z)$

We will actually analyze the sign of $ H_{S}(z) := z [R^{+}_{z} - R(z)]$, showing  it is strictly monotone and thus  vanishes at most once. Using  \eqref{jump} and \eqref{tworoots} we get
\begin{align}
H_{S}(z) & =  \frac{Sz}{2(S-1)(1+(S-1)z)}\left[S-2 + \sqrt{\Delta(P_{z})}\right] - \log  \frac{1+(S-1)z}{1-z} \label{HS}\\
H_{S}'(z) &  = \frac{S^{2}\left[ S + 2(2S-3)z + (S-2)(2S-3)z^{2} + \left(-1+2(S+2)z + (2S-3)z^{2}\right) \sqrt{\Delta(P_{z})}\right]}{2(S-1)(1-z)^{2}(1+(S-1)z)^{2} \sqrt{\frac{S(S+(3S-4)z)}{1-z}}}\nonumber\\
 &  = \frac{S^{2}\left[A(z) + B(z) \sqrt{\Delta(P_{z})}\right]}{2(S-1)(1-z)^{2}(1+(S-1)z)^{2} \sqrt{\Delta(P_{z})}}\label{Hpz}
\end{align}
In the formula \eqref{Hpz} above, the denominator as well as the polynomial $A(z)$ in the numerator are clearly positive for all $z>0$. Since the polynomial $B(z)$ is increasing for $z>0$ and since $B(z_{S}) \geq  -1 + 2(S-1) z_{S} =  2S-5 > 0$, we deduce that  $H_{S}'(z)$ is always positive for $z\in[z_{S},1)$.

\item $H_{S}$ vanishes exactly once $\iff S \leq 59$

We now check for which values of $S$ the function $H_{S}$ actually vanishes somewhere on $[z_{S},1)$. As $z\to 1$, the leading term in $H_{S}$ diverges like $(1-z)^{-1/2}$, so that $H_{S}(z) \to + \infty$. Thus $H_{S}$ will vanish exactly once if and only if $H_{S}(z_{S}) \le 0$.

\begin{align*}
G(S) & = H_{S}(z_{S})  = \frac{S(S-2)(S-2+\sqrt{S(8-11S+4S^{2})})}{2(S-1)^{3}} - 2\log (S-1)\\
G'(S)& = \frac{S\left[2S^{4}-10S^{3}+27S^{2}-40S+24-(4S^{2}-13S+12)\sqrt{S(8-11S+4S^{2})} \right]}{2(S-1)^{4}\sqrt{S(8-11S+4S^{2})}}
\end{align*}
and
\begin{align*}
G'(S)=0 & \iff 2S^{4}-10S^{3}+27S^{2}-40S+24 = (4S^{2}-13S+12)\sqrt{S(8-11S+4S^{2})} \\
& \iff  (2S^{4}-10S^{3}+27S^{2}-40S+24)^{2} = S(8-11S+4S^{2})(4S^{2}-13S+12)^{2}\\
& \iff 4(S-2)^{2}(S-1)^3(S^3-19S^{2}+48S-36) = 0.
\end{align*}
The last bracket reaches a local (negative) maximum at $S = \frac{19-\sqrt{217}}{3} \approx 1.4$ and a local (negative) minimum at $S = \frac{19-\sqrt{217}}{3} \approx 11.2$. Therefore it has exactly one root $S^{\star}$, is negative before this root and positive after it. Numerical computations give $S^{\star} \approx 16.2$.

From this, we know that $G$ is decreasing on $[3,S^{\star}]$ and increasing on $[S^{\star},\infty)$. Since $G(3) < 0 $ and since $G(S)$ diverges like $+\sqrt{S}$ as $S\to \infty$, we get that $G$ has exactly one root $\bar{S} > S^{\star}$, is negative before it and positive after it. Numerical computations show $\bar{S} \approx 59.1$. \qed
\end{itemize}

\vskip1cm

{\bf Proof of Lemma \ref{lem=common_tangent}.}
We will use the notation
\begin{equation*}
K := \left\{ \alpha \in [b,c); \alpha \geq b \text{ and } T_{g}(\alpha)\cap \Gamma_{f} \neq \emptyset \right\},
\end{equation*}
where $T_{g}(\alpha)$ denotes the tangent to  $\Gamma_g$ at $\alpha$.

Since $f(b)=g(b)$, we have $b \in K$, and $K$ is non-empty. Besides, by continuity of $g'$, there exists $b_0 \in (b,c)$ such that $g'(b_0) = f'(b)$; since $f,g$ are strictly convex, elements of $K$ are bounded from above by $b_0$ and  $\alpha^{\star}:= \sup K \le b_{0}$ is well defined.

Now, let $ \alpha_{n}$ an increasing sequence converging to $\alpha^{\star}$. By definition, $T_{g}(\alpha_{n})$ intersects $\Gamma_f$, and we denote by $x_{n}$ the abscisse of the intersection point which is the closest to $b$, so that $f'(x_{n}) \ge g'(\alpha_{n}) \ge g'(b)$. We now show that $x_n$ is a bounded decreasing  sequence:
\begin{itemize}

\item On $\left\{ x \geq x_{n}  \right\}$, $\Gamma_{f}$ is above $T_{f}(x_n)$ (convexity of $f$), which in turn is above $T_{g}(\alpha_n)$ (definition of $x_n$), and therefore above $T_{g}(\alpha_{n+1})$ (convexity of $g$). Thus $\Gamma_{f}$ may not intersect $T_{g}(\alpha_{n+1})$ after abscisse $x_n$, and $x_{n+1} \leq x_{n}$.

\item By continuity of $f'$, there exists $b_{1} \in (a,b)$ such that $f'(b_{1}) = g'(b)$, thus $f'(x_{n}) \ge g'(b)$ implies $x_{n} \ge b_{1}$ (convexity of $f$).

\end{itemize}

Thus $x_{n} \to x^{\star} \in [b_{1},b] \subset (a,b]$, and by continuity of $f,g,g'$, $T_{g}(\alpha^{\star})$ intersects $\Gamma_{f}$ at $(x^{\star},f(x^{\star})) $. In particular, $\alpha^{\star} \in K$ and $f'(x^{\star}) \geq g'(\alpha^{\star})$.

If we had $f'(x^{\star}) > g'(\alpha^{\star})$ we could apply the implicit function function theorem to $\Psi(\alpha,x) = g(\alpha) + g'(\alpha)(x-\alpha) - f(x)$ to deduce that $K$ contains a neighborhood of $\alpha^{\star}$, thus contradicting the maximality of $\alpha^{\star}$. Therefore  $f'(x^{\star}) = g'(\alpha^{\star})$ and $T_{g}(\alpha^{\star}) = T_{f}(x^{\star})$ is actually tangent to $\Gamma_{f}$.
\qed

\bibliographystyle{amsalpha}

\begin{thebibliography}{99}
\label{t5}

\bibitem{BKMP} P. Baffioni, T.Kuna,  I Merola, E. Presutti:
{A liquid vapor phase transition in quantum statistical mechanics. Submmitted to {\em Memoirs AMS }
(2004).}



\bibitem{vb} J. van der Berg: A uniqueness condition
for Gibbs measures with application to the two dimensional
antiferromagnet {\em Commun. Math. Phys. \bf 152}(1993), 161--166.



\bibitem{vm} J. van der Berg, C. Maes: Disagreement percolation in
the study of Markov fields {\em Ann. Prob.\bf 22}(1994), 749--763.


\bibitem{vs} J. van der Berg, J.E. Steif:  Percolation and
the hard core lattice model  {\em Stochastic Processes and Appl.
\bf 49}(1994), 179--197.

 \bibitem{BZ} A. Bovier, M. Zahradnik: {The low temperature phase of Kac-Ising models
{\em J.Stat. Phys. \bf 87}  (1997),
311-332.}

 \bibitem{BMP} P. Butt\`{a}, I. Merola, E. Presutti: On the validity
of the van der Waals theory in Ising systems with long range
interactions {\em Markov Provesses and Related Fields \bf 3} (1977)
63--88



\bibitem{CP} M. Cassandro, E. Presutti:
{Phase transitions in Ising systems with long but finite range interactions
{\em Markov Processes and Related Fields \bf 2}(1996) 241--262.}


\bibitem{DMPV2} A. De Masi, I. Merola, E. Presutti,
Y. Vignaud: Coexistence of ordered and disordered phases in Potts
models in the continuum {\em in preparation}


\bibitem{DS}{R.L. Dobrushin, S.B. Shlosman:}
{Completely analytical interactions: constructive description
{\em J.Stat. Phys. \bf 46(5--6)}(1987) 983--1014.}


\bibitem{GMRZ} {Hans-Otto Georgii, Salvador Miracle-Sole, Jean Ruiz, Valentin Zagrebnov: Mean field theory of the Potts Gas {\em J. Phys. A} {\bf 39}  (2006) 9045--9053.}


\bibitem{GH} {Hans-Otto Georgii, O. H\"{a}ggstr\"{o}m: Phase transition in continuum Potts models. {\em  Comm. Math. Phys.} {\bf  181} (1996)  507--528.}


\bibitem{GM} T. Gobron, I. Merola: First order
phase transitions in Potts models with finite range interactions
{\em J. Stat. Phys., \bf 126} (2006).


\bibitem{LMP} J.L. Lebowitz, Mazel, E. Presutti: Liquid vapour
phase transitions for systems with finite range interactions {\em
J. Stat. Phys} (1999).



\bibitem{leipzig} E. Presutti:
From Statistical Mechanics towards Continuum Mechanics. {\em course
given at Max-Plank Institute, Leipzig } (1999).

\label{t6}

\bibitem{ruelle} D. Ruelle: Widom-Rowlinson:
{Existence of a phase transition in a continuous classical system. {\em Phys. Rev. Lett. \bf 27} (1971) 1040--1041.}


\bibitem{cluster-expansion} {M. Zahradn\`\i k: A short course on the Pirogov-Sinai theory.  {\em Rend. Mat. Appl.} {\bf 18}, 411--486 (1998).}
\label{y6}




\end{thebibliography}

\end{document}